\documentclass[11pt]{amsart}
\usepackage{amsmath,amssymb,amsfonts}
\usepackage{tikz-cd}
\usepackage{amsbsy}
\usepackage{amscd}
\usetikzlibrary{matrix,arrows,decorations.pathmorphing}
\usepackage{graphicx}
\usepackage{float}
\usepackage{enumitem}
\usepackage{setspace}
\usepackage[margin=1.0in]{geometry}
\usepackage[all]{xy}
\newtheorem{thm}{Theorem}[section]
\newtheorem{lem}[thm]{Lemma}
\newtheorem{cor}[thm]{Corollary}
\newtheorem{prop}[thm]{Proposition}

\theoremstyle{definition}

\theoremstyle{definition}
\newtheorem{examples}[thm]{Examples}
\theoremstyle{definition}
\newtheorem{defn}[thm]{Definition}
\theoremstyle{definition}
\newtheorem{remark}[thm]{Remark}

\newcommand{\mc}[1]{\mathcal{#1}}
\newcommand{\e}[1]{\emph{#1}}

\newcommand{\la}{\langle}
\newcommand{\ra}{\rangle}

\newcommand{\rmv}[1]{}

\newcommand{\hs}{\hskip10pt}

\newcommand{\LG}{VN(G)}

\newcommand{\LO}{L^1(G)}
\newcommand{\LOQ}{L^1(\mathbb{G})}
\newcommand{\LOQs}{L^1_*(\mathbb{G})}

\newcommand{\LOQH}{L^1(\widehat{\mathbb{G}})}
\newcommand{\LOQHs}{L^1_*(\widehat{\mathbb{G}})}
\newcommand{\LOQHP}{L^1(\widehat{\mathbb{G}}')}

\newcommand{\LT}{L^2(G)}
\newcommand{\LTQ}{L^2(\mathbb{G})}

\newcommand{\LI}{L^{\infty}(G)}
\newcommand{\LIQ}{L^{\infty}(\mathbb{G})}

\newcommand{\LIQH}{L^{\infty}(\widehat{\mathbb{G}})}

\newcommand{\LIQHP}{L^{\infty}(\widehat{\mathbb{G}}')}

\newcommand{\BH}{\mc{B}(H)}

\newcommand{\LUC}{\mathrm{LUC}(\mathbb{G})}

\newcommand{\BLT}{\mc{B}(L^2(G))}
\newcommand{\BLTQ}{\mc{B}(L^2(\mathbb{G}))}

\newcommand{\TCQ}{\mc{T}(L^2(\mathbb{G}))}

\newcommand{\Mcb}{M_{cb}A(G)}

\newcommand{\McbQl}{M_{cb}^l(L^1(\mathbb{G}))}

\newcommand{\McbQHl}{M_{cb}^l(L^1(\widehat{\mathbb{G}}))}

\newcommand{\McbQr}{M_{cb}^r(L^1(\mathbb{G}))}

\newcommand{\Nphi}{\mc{N}_\varphi}
\newcommand{\Mphi}{\mc{M}_\varphi}
\newcommand{\Mphip}{\mc{M}_{\varphi}^+}

\newcommand{\vphi}{\varphi}

\newcommand{\Lpsi}{\Lambda_\psi}
\newcommand{\al}{\alpha}
\newcommand{\be}{\beta}
\newcommand{\lm}{\lambda}

\newcommand{\Gam}{\Gamma}
\newcommand{\om}{\omega}

\newcommand{\ten}{\otimes}
\newcommand{\oten}{\overline{\otimes}}
\newcommand{\pten}{\widehat{\otimes}}
\newcommand{\hten}{\otimes^h}
\newcommand{\iten}{\otimes^{\vee}}

\newcommand{\whten}{\otimes^{w^*h}}

\newcommand{\id}{\textnormal{id}}
\newcommand{\h}[1]{\widehat{#1}}

\newcommand{\Irr}{\mathrm{Irr}(\mathbb{G})}

\newcommand{\Ad}{\mathrm{Ad}}
\newcommand{\ad}{\mathrm{ad}}
\newcommand{\ep}{\varepsilon}

\providecommand{\norm}[1]{\lVert#1\rVert}

\newcommand{\G}{\mathbb{G}}
\newcommand{\bG}{b\mathbb{G}}

\newcommand{\bGhp}{\widehat{b\mathbb{G}}'}
\newcommand{\Hb}{\mathbb{H}}
\newcommand{\C}{\mathbb{C}}
\newcommand{\N}{\mathbb{N}}

\newcommand{\R}{\mathbb{R}}

\DeclareSymbolFont{lettersA}{U}{txmia}{m}{it}
\DeclareMathSymbol{\W}{\mathord}{lettersA}{151}

\newcommand{\supp}{\operatorname{supp}}

\begin{document}

\title[]{Mapping ideals of quantum group multipliers}
\author{Mahmood Alaghmandan}
\email{mahmood.alaghmandan@carleton.ca}
\author{Jason Crann}
\email{jasoncrann@cunet.carleton.ca}
\author{Matthias Neufang}
\email{matthias.neufang@carleton.ca}
\address{School of Mathematics and Statistics, Carleton University, Ottawa, ON, Canada H1S 5B6}


\begin{abstract} We study the dual relationship between quantum group convolution maps $\LOQ\rightarrow\LIQ$ and completely bounded multipliers of $\h{\G}$. For a large class of locally compact quantum groups $\G$ we completely isomorphically identify the mapping ideal of row Hilbert space factorizable convolution maps with $M_{cb}(\LOQH)$, yielding a quantum Gilbert representation for completely bounded multipliers. We also identify the mapping ideals of completely integral and completely nuclear convolution maps, the latter case coinciding with $\ell^1(\widehat{\bG})$, where $\bG$ is the quantum Bohr compactification of $\G$. For quantum groups whose dual has bounded degree, we show that the completely compact convolution maps coincide with $C(b\G)$. Our techniques comprise a mixture of operator space theory and abstract harmonic analysis, including Fubini tensor products, the non-commutative Grothendieck inequality, quantum Eberlein compactifications, and a suitable notion of quasi-SIN quantum group, which we introduce and exhibit examples from the bicrossed product construction. Our main results are new even in the setting of group von Neumann algebras $VN(G)$ for quasi-SIN locally compact groups $G$.\end{abstract}

\begin{spacing}{1.0}

\maketitle

\section{Introduction}

In \cite{Gil} Gilbert established an important representation theorem for completely bounded multipliers of the Fourier algebra of a locally compact group $G$, showing that a function $\vphi:G\rightarrow\C$ lies in $\Mcb$ if and only if there exist a Hilbert space $H$ and continuous maps $\xi,\eta:G\rightarrow H$ such that
$$\vphi(st^{-1})=\la\eta(t),\xi(s)\ra, \ \ \ s,t\in G.$$
Moreover, $\norm{\vphi}_{cb}=\inf\{\norm{\xi}_\infty\norm{\eta}_\infty\}$ (see \cite{Jol} for a short proof using the representation theorem for completely bounded maps). From the perspective of convolution, Gilbert's representation is the (completely) isometric identification
\begin{equation}\label{e:Gil}\Gam^{2,r}_{\LO}(\LO,\LI)\cong\Mcb,\end{equation}
where $\Gam^{2,r}_{\LO}(\LO,\LI)$ is the space of completely bounded right $\LO$-module maps $\LO\rightarrow\LI$ which factor through a row Hilbert space. The isomorphism (\ref{e:Gil}) reveals an interesting manifestation of quantum group duality: As $\mc{CB}_{\LO}(\LO,\LI)\cong\LI$, we can recover the dual object $M_{cb}A(G)=M_{cb}(\widehat{\LO})$ inside $\LI$ as convolution maps which factor through a row Hilbert space. It is therefore natural to study a quantum group version of Gilbert's representation theorem, not to mention its usefulness in abstract harmonic analysis (e.g. \cite{BF,CH,HdL,HK,Kn,NRS,Sp}).

Throughout the paper we focus on two large classes of quantum groups, those whose dual is quasi-SIN or has bounded degree. The latter class has recently appeared in the literature in connection with Property (T) \cite{DSV}, Thoma type results \cite{BC} and unimodularity \cite{KSol} for discrete quantum groups. The former class has appeared implicitly in \cite{BY,RX,Wil}, but we make a point of introducing a proper notion of a quasi-SIN quantum group, generalizing the well-studied notion from harmonic analysis. After a preliminary section, we begin with the definition of a quasi-SIN, or QSIN quantum group in section \ref{s:QSIN}. We establish some basic properties and present non-trivial examples arising from the bicrossed product construction.

One our main results, established in section \ref{s:Gilbert}, is the complete isomorphism
\begin{equation}\label{e:Gilbert}\Gam^{2,r}_{\LOQ}(\LOQ,\LIQ)\cong M^l_{cb}(\LOQH),\end{equation}
for the aforementioned classes of quantum groups, where $\Gam^{2,r}_{\LOQ}(\LOQ,\LIQ)$ is the space of completely bounded right $\LOQ$-module maps $\LOQ\rightarrow\LIQ$ which factor through a row Hilbert space. Our techniques are remarkably different for each class. In the bounded degree setting, we make use of the structure of subhomogeneous $C^*$-algebras together with two manifestations of a commutation relation expressing quantum group duality; one at the level of multipliers \cite[Theorem 5.1]{JNR} and the other at the level of co-multiplications \cite[Proposition 6.3(1)]{KS}. In the QSIN setting, the isomorphism (\ref{e:Gilbert}) is completely isometric, and the proof relies on a particular $\LOQ$-bimodule left inverse of the co-multiplication which behaves well with respect to the weak* Haagerup tensor product. In either case, the complete isomorphism (\ref{e:Gilbert}) is weak*-weak* homeomorphic, allowing us to identity the module Haagerup tensor product $\LOQ\hten_{\LOQ}\LOQ$ with the operator predual $Q^l_{cb}(\LOQH)=(M^l_{cb}(\LOQH))_*$. As further corollaries, we (a) deduce a spectral synthesis result for the bivariate Fourier algebra $A_h(G)=A(G)\hten A(G)$ (see \cite{ATT}) of a QSIN group $G$, which, in particular, fills a gap in \cite{LSS} (see Remark \ref{r:LSS}), and (b) show that $\LOQ$ is completely isomorphic to an operator algebra if and only if $\G$ is finite. The isomorphism (\ref{e:Gilbert}) also simultaneously generalizes, and provides a new proof of, the classical inclusion $M_{cb}A(G)\subseteq\mathrm{WAP}(G)$ \cite{X}, where $\mathrm{WAP}(G)$ is the space of weakly almost periodic functions on $G$. This partially answers a question raised in \cite[Remark 5.7]{HNR}. In the QSIN setting we obtain a new proof of the inclusion $M(\G)\subseteq\mathrm{WAP}(\h{\G})$, established in \cite[Theorem 5.6]{HNR} (see \cite[Theorem 2.8 and Chapter 8]{DR} for the co-commutative setting). Since $\mathrm{WAP}(\G):=\{x\in\LIQ\mid L_x\in\mc{CB}_{\LOQ}(\LOQ,\LIQ) \ \textnormal{is weakly compact}\}$, combining \cite[Proposition 3.13]{DL} with \cite[Corollary 1]{DFJP} and its cb-version \cite[Theorem 2.1]{PS}, we have
$$\mathrm{WAP}(\G)=\Gam^{\mathrm{ref}}_{\LOQ}(\LOQ,\LIQ),$$
where $\Gam^{\mathrm{ref}}_{\LOQ}(\LOQ,\LIQ)$ is the space of completely bounded right $\LOQ$-module maps $\LOQ\rightarrow\LIQ$ which factor through a reflexive operator space. Thus, for our two classes of quantum groups, the difference between $M^l_{cb}(\LOQH)$ and $\mathrm{WAP}(\G)$ is precisely the difference between factoring through a row Hilbert space and a reflexive operator space. 

We remark that a Gilbert type representation has been studied for $M^l_{cb}(\LOQH)$ at the level of $\mc{CB}^\sigma(\LIQH)$ through the theory of Hilbert $C^*$-modules \cite{D3}. Our approach is in some sense dual to \cite{D3}, where we witness $M^l_{cb}(\LOQH)$ through its nature as a mapping ideal of convolution maps $\LOQ\rightarrow\LIQ$.

Using work of Gilbert \cite{Gil}, Racher showed that the integral $\LO$-module maps from $\LO$ to $\LI$ are precisely the completely bounded multipliers of the Fourier algebra \cite[Proposition 5.1]{R}. Thus, a convolution map $\LO\rightarrow\LI$ is integral if and only if it factors through a Hilbert space. We establish a quantum group version of this result (for the two aforementioned classes) in section \ref{s:ci}, showing that a convolution map $\LOQ\rightarrow\LIQ$ is completely integral if and only if it factors through a row Hilbert space, that is
$$\mc{I}_{\LOQ}(\LOQ,\LIQ)\cong M^l_{cb}(\LOQH)\cong\Gam^{2,r}_{\LOQ}(\LOQ,\LIQ).$$
At the level of module tensor products, we obtain the (Banach space) isomorphism
$$\LOQ\ten^{\vee}_{\LOQ}\LOQ\cong\LOQ\hten_{\LOQ}\LOQ.$$
Here, our techniques rely on our Gilbert representation (\ref{e:Gilbert}), the non-commutative Grothendieck inequality \cite{HM,PS}, and the local reflexivity of von Neumman algebra preduals \cite{EJR}.

Answering a question of Crombez and Govaerts \cite{CG}, Racher showed that the nuclear $\LO$-module maps from $\LO$ to $\LI$ are precisely convolution with almost periodic elements of the Fourier--Stieltjes algebra \cite[Theorem]{R}, that is, $N_{\LO}(\LO,\LI)\cong B(G)\cap AP(G)$ isomorphically. Since $B(G)\cap AP(G)=A(bG)$ \cite[Proposition 2.1]{RS}, where $bG$ is the Bohr compactification of $G$, duality suggests that the completely nuclear convolution maps ought to be related to the quantum Bohr compactification \cite{Sol}. Indeed, another major result of the paper, established in section \ref{s:cn}, is the (Banach space) isomorphism 
\begin{equation}\label{e:cn}\mc{N}_{\LOQ}(\LOQ,\LIQ)\cong \ell^1(\widehat{\bG})\end{equation}
for the two classes of quantum groups, where $\mc{N}_{\LOQ}(\LOQ,\LIQ)$ are the completely nuclear convolution maps and $\widehat{\bG}$ is the discrete dual of the quantum Bohr compactification of $\G$. In the QSIN case, the techniques we employ are a refinement of those used in the proof of (\ref{e:Gilbert}), based on ideas from \cite{D2}, together with the non-commutative Grothendieck inequality and the theory of quantum Eberlein compactifications \cite{DD}. In the case of bounded degree, we first prove a result of independent interest, which is the completely isometric isomorphism
\begin{equation}\label{e:ck}\mc{CK}_{\LOQ}(\LOQ,\LIQ)\cong C(b\G),\end{equation}
where $\mc{CK}_{\LOQ}(\LOQ,\LIQ)$ are the completely compact convolution maps, and then argue via (\ref{e:cn}). The isomorphisms (\ref{e:cn}) and (\ref{e:ck}) complement the analysis of \cite{Chou,D2,Rindler}, where, contrary to the classical setting, it was shown that, in general, complete compactness was insufficient to recover the quantum Bohr compactification. Our results imply that for a large class of quantum groups one \textit{can} recover the (discrete dual of the) quantum Bohr compactification by considering completely \textit{nuclear} convolution maps, as opposed to completely compact ones. In connection with almost periodic elements, our techniques entail the equality
$$\mathrm{AP}(\G)=\overline{\{x\in\LIQ\mid\Gam(x)\in\LIQ\hten\LIQ\}}^{\norm{\cdot}},$$
where $\mathrm{AP}(\G)$, denoted $\mathbb{AP}(C_0(\G))$ in \cite{D2}, is the norm closure of matrix coefficients of finite-dimensional admissible co-representations inside $M(C_0(\G))$ \cite{D2,Sol}. 

\section{Preliminaries}

\subsection{Mapping Ideals and Tensor Products} Throughout the paper we let $\pten$, $\ten^{\vee}$ and $\hten$ denote the operator space projective, injective, and Haagerup tensor products, respectively. The algebraic and Hilbert space tensor products will be denoted by $\ten$, the relevant product being clear from context. The von Neumann tensor product will be denoted by $\oten$. On a Hilbert space $H$, we let $\mc{K}(H)$, $\mc{T}(H)$ and $\BH$ denote the spaces of compact, trace class, and bounded operators, respectively. We briefly outline the mapping ideals of interest in the paper, referring the reader to \cite{ER} for details and notation.

A \textit{mapping ideal} $\mc{O}$ is an assignment to each pair of operator spaces $X,Y$, a linear space $\mc{O}(X,Y)$ of completely bounded mappings $\vphi:X\rightarrow Y$, together with an operator space matrix norm $\norm{\cdot}_{\mc{O}}$ such that for each $\vphi\in M_n(\mc{O}(X,Y))$,
\begin{enumerate}
\item $\norm{\vphi}_{cb}\leq\norm{\vphi}_{\mc{O}}$, and
\item for any linear mappings $r:V\rightarrow X$ and $s:Y\rightarrow W$, $\norm{s_n\circ \vphi\circ r}_{\mc{O}}\leq\norm{s}_{cb}\norm{\vphi}_{\mc{O}}\norm{r}_{cb}$.
\end{enumerate}

Given operator spaces $X$ and $Y$ there are canonical complete contractions
$$X\pten Y\xrightarrow{\Phi_{h}}X\hten Y\xrightarrow{\Phi_{h,\vee}}X\ten^{\vee}Y,$$
whose composition we denote by $\Phi_{\vee}$. The \textit{completely nuclear mappings} $\mc{N}(X,Y)$ is the image of
$$\Phi_{\vee}:X^*\pten Y\rightarrow X^*\ten^{\vee}Y\subseteq\mc{CB}(X,Y)$$
with the quotient operator space structure from $X^*\pten Y/\mathrm{Ker}(\Phi_{\vee})$. For a linear mapping $\vphi:X\rightarrow Y$ in the range of $\Phi_{\vee}$, we let $\nu(\vphi)$ denote the corresponding quotient norm. Then $\norm{\vphi}_{cb}\leq\nu(\vphi)$, and $\mc{N}(X,Y)$ is a mapping ideal.  

A liner map $\vphi:X\rightarrow Y$ is \textit{completely integral} if
$$\iota(\vphi):=\sup\{\nu(\vphi|_E)\mid E\subseteq X \ \textnormal{finite dimensional}\}<\infty.$$
We let $\mc{I}(X,Y)$ denote the completely integral mappings. The operator space structure on $\mc{I}(X,Y)$ is defined by
$$\iota_n(\vphi):=\sup\{\nu_n(\vphi|_E)\mid E\subseteq X \ \textnormal{finite dimensional}\}, \ \ \ \vphi\in M_n(\mc{I}(X,Y)), \ n\in\N.$$
If $Y$ is locally reflexive then $\mc{I}(X,Y^*)\cong(X\ten^{\vee}Y)^*$, completely isometrically.

A linear mapping $\vphi:X\rightarrow Y$ is \textit{completely 1-summing} if
$$\pi^1(\vphi):=\norm{\id\ten\vphi:\mc{T}(\ell^2)\ten^{\vee}X\rightarrow\mc{T}(\ell^2)\pten Y}_{cb}<\infty.$$
We let $\Pi^1(X,Y)$ denote the space of completely 1-summing maps with the operator space structure inherited from the inclusion $\Pi^1(X,Y)\subseteq\mc{CB}(\mc{T}(\ell^2)\ten^{\vee}X,\mc{T}(\ell^2)\pten Y)$. There is an operator space tensor product, which we denote $\ten^{\vee/}$, such that $\Pi^1(X,Y^*)\cong (X\ten^{\vee/}Y)^*$ completely isometrically (see \cite[Corollary 5.5]{ER2}, where $\vee/$ is denoted $v_q$).

Finally, a liner mapping $\vphi:X\rightarrow Y$ factors through a row Hilbert space if there is a Hilbert space $H$ and completely bounded maps $r:X\rightarrow H_r$ and $s:H_r\rightarrow Y$ for which the following diagram commutes.

\begin{equation*}
\begin{tikzcd}
&H_r \arrow[rd, "s"]\\
X \arrow[ru, "r"] \arrow[rr, "\vphi"] & & Y
\end{tikzcd}
\end{equation*}

We let $\Gam^{2,r}(X,Y)$ denote the space of such mappings. Given $\vphi=[\vphi_{ij}]\in M_n(\Gam^{2,r}(X,Y))$, the associated mapping $\vphi:X\rightarrow M_n(Y)$ satisfies

\begin{equation*}
\begin{tikzcd}
&M_{n,1}(H_r) \arrow[rd, "s_{n,1}"]\\
X \arrow[ru, "r"] \arrow[rr, "\vphi"] & & M_n(Y)
\end{tikzcd}
\end{equation*}
where $s:H_r\rightarrow M_{1,n}(Y)$. The norm $\gamma^{2,r}(\vphi)=\inf\{\norm{r}_{cb}\norm{s}_{cb}\}$ then determines an operator space structure on $\Gam^{2,r}(X,Y)$. It is well-known that $\Gam^{2,r}(X,Y^*)\cong(X\ten^h Y)^*$ completely isometrically, via
\begin{equation}\label{e:Haa}\la\vphi, x\ten y\ra=\la\vphi(x),y\ra, \ \ \ x\in X, \ y\in Y.\end{equation}
A similar construction exists for column Hilbert spaces, and one has $\Gam^{2,c}(X,Y^*)\cong(Y\hten X)^*$. 

For any operator spaces $X,Y$ and linear $\vphi:X\rightarrow Y$, we have 
\begin{equation}\label{e:chain}\norm{\vphi}_{cb}\leq\gamma^{2,r}(\vphi)\leq\pi^1(\vphi)\leq\iota(\vphi)\leq\nu(\vphi),\end{equation}
which yield the set-theoretic inclusions
$$\mc{N}(X,Y)\subseteq\mc{I}(X,Y)\subseteq\Pi^1(X,Y)\subseteq\Gam^{2,r}(X,Y)\subseteq\mc{CB}(X,Y).$$

Given an operator space tensor product $\alpha\in\{\vee,h\}$, and operator space inclusions $X_1\subseteq Y_1$ and $X_2\subseteq Y_2$, we define the \textit{Fubini tensor product} by
$$\mc{F}(X_1,X_2;Y_1\ten^{\alpha}Y_2):=\{u\in Y_1\ten^{\alpha}Y_2\mid (m_1\ten\id)(u)\in X_2, \ m_1\in Y_1^*, (\id\ten m_2)(u)\in X_1, \ m_2\in Y_2^*\}.$$
Then $\mc{F}(X_1,X_2;Y_1\ten^{\alpha}Y_2)$ is a closed subspace of $Y_1\ten^{\alpha}Y_2$ containing $X_1\ten^{\alpha}X_2$ by injectivity of $\vee$ and $h$. When $\alpha=h$, it is known that $\mc{F}(X_1,X_2;Y_1\ten^{h}Y_2)=X_1\ten^{h}X_2$ \cite[Corollary 4.8]{Smith}, whereas when $\alpha=\vee$, the equality $\mc{F}(X_1,X_2;Y_1\ten^{\vee}Y_2)=X_1\ten^{\vee}X_2$ does not hold in general, and is related to approximation properties of the associated operator spaces \cite[Theorem 11.3.1]{ER}.

Another important mapping ideal was defined by Saar in his PhD thesis \cite{Saar} as an analogue of compact mappings for operator spaces. A completely bounded map $\vphi\in\mc{CB}(X,Y)$ between operator spaces $X$ and $Y$ is said to be \textit{completely compact} if for every $\ep>0$, there is a finite-dimensional subspace $Z\subseteq Y$ such that $\norm{q_Z\circ\vphi}_{cb}<\ep$, where $q_Z:Y\rightarrow Y/Z$ is the quotient map. We let $\mc{CK}(X,Y)$ denote the space of completely compact maps. Then $\mc{CK}(X,Y)$ is a closed subspace of $\mc{CB}(X,Y)$ containing all finite-rank maps.

Given dual operator spaces $X^*\subseteq\BH$ and $Y^*\subseteq\mc{B}(K)$, the \textit{weak*-Haagerup tensor product} $X^*\whten Y^*$ is the space of $u\in\BH\oten\mc{B}(K)$ for which there exist an index set $I$ and $(x_i)_{i\in I}\subseteq X$ and $(y_i)_{i\in I}\subseteq Y$ satisfying 
$\norm{\sum_{i}x_ix_i^*},\norm{\sum_iy_i^*y_i}<\infty$ and $u=\sum_i x_i\ten y_i$, where each sum is understood in the respective weak* topologies. Then
$$\norm{u}_{w^*h}:=\inf\{\norm{\sum_{i}x_ix_i^*},\norm{\sum_iy_i^*y_i}\mid u=\sum_i x_i\ten y_i\}$$
and the infimum is actually attained \cite[Theorem 3.1]{BS}. There are corresponding matricial norms on $M_n(X^*\whten Y^*)$ giving an operator space structure on $X^*\whten Y^*$ which is independent of the weak* homeomorphic inclusions $X^*\subseteq\BH$ and $Y^*\subseteq\mc{B}(K)$. Moreover, $(X\hten Y)^*\cong X^*\whten Y^*$ completely isometrically \cite[Theorem 3.2]{BS}. If $M\subseteq\BH$ is a von Neumann algebra, then
$$\mc{CB}^{\sigma}_M(\BH)\cong M'\whten M',$$
completely isometrically and weak*-weak* homeomorphically \cite{Haa} (see also \cite[Theorem 4.2]{BS}), where $\mc{CB}^{\sigma}_M(\BH)$ is the space of normal completely bounded $M$-bimodule maps on $\BH$.

\subsection{Quantum Groups} A \e{locally compact quantum group} is a quadruple $\G=(\LIQ,\Gam,\vphi,\psi)$, where $\LIQ$ is a Hopf-von Neumann algebra with co-multiplication $\Gam:\LIQ\rightarrow\LIQ\oten\LIQ$, and $\vphi$ and $\psi$ are fixed left and right Haar weights on $\LIQ$, respectively \cite{KV1,KV2}. Throughout the paper we adopt the following standard notations from weight theory
$$\Mphip:=\{x\in \LIQ^+ : \vphi(x)<\infty\}, \ \ \ \Nphi:=\{x\in\LIQ : \vphi(x^*x)<\infty\}.$$ 
It is well-known that $\Nphi$ is a left ideal in $\LIQ$, and that $\Mphi:=\text{span}\{x^*y : x,y\in\Nphi\}=\text{span} \ \Mphip$ is a *-subalgebra of $\LIQ$ \cite{T2}. Similarly for the right Haar weight $\psi$.

For every locally compact quantum group $\G$, there exists a \e{left fundamental unitary operator} $W$ on $L^2(\G,\vphi)\ten L^2(\G,\vphi)$ and a \e{right fundamental unitary operator} $V$ on $L^2(\G,\psi)\ten L^2(\G,\psi)$ implementing the co-multiplication $\Gam$ via
\begin{equation*}\Gam(x)=W^*(1\ten x)W=V(x\ten 1)V^*, \ \ \ x\in\LIQ.\end{equation*}
Both unitaries satisfy the \e{pentagonal relation}; that is,
\begin{equation}\label{e:pentagonal}W_{12}W_{13}W_{23}=W_{23}W_{12}\hs\hs\mathrm{and}\hs\hs V_{12}V_{13}V_{23}=V_{23}V_{12}.\end{equation}
For simplicity we write $\LTQ$ for $L_2(\G,\vphi)$ throughout the paper. There is a strictly positive operator $\delta$ affiliated with $\LIQ$, called the \textit{modular element}, satisfying $\psi(x)=\vphi(\delta^{1/2}x\delta^{1/2})$, $x\in \mc{M}_\psi$. 

Let $\LOQ$ denote the predual of $\LIQ$. Then the pre-adjoint of $\Gam$ induces an associative completely contractive multiplication on $\LOQ$, defined by
\begin{equation*}\star:\LOQ\pten\LOQ\ni f\ten g\mapsto f\star g=\Gam_*(f\ten g)\in\LOQ.\end{equation*}
The multiplication $\star$ is a complete quotient map from $\LOQ\pten\LOQ$ onto $\LOQ$, implying
\begin{equation*}\la\LOQ\star\LOQ\ra=\LOQ,\end{equation*}
where here, and in what follows $\la\cdot\ra$ denotes closed linear span. Moreover, $\LOQ$ is always \textit{self-induced}, meaning $\LOQ\cong\LOQ\pten_{\LOQ}\LOQ$ completely isometrically, where $\LOQ\pten_{\LOQ}\LOQ$ is the module tensor product defined as the quotient of $\LOQ\pten\LOQ$ by the closed subspace $\la f\star g\ten h - f\ten g\star h\mid f,g,h\in\LOQ\ra$. The proof follows from \cite[Theorem 2.7]{Vaes} (see \cite[Proposition 3.1]{C} for details).

The canonical $\LOQ$-bimodule structure on $\LIQ$ is given by
\begin{equation*}f\star x=(\id\ten f)\Gam(x)\hs\hs\mathrm{and}\hs\hs x\star f=(f\ten\id)\Gam(x)\end{equation*}
for $x\in\LIQ$, and $f\in\LOQ$. A \e{left invariant mean on $\LIQ$} is a state $m\in \LIQ^*$ satisfying
\begin{equation*}\la m,x\star f \ra=\la f,1\ra\la m,x\ra, \ \ \ x\in\LIQ, \ f\in\LOQ.\end{equation*}
Right and two-sided invariant means are defined similarly. A locally compact quantum group $\G$ is said to be \e{amenable} if there exists a left invariant mean on $\LIQ$. It is known that $\G$ is amenable if and only if there exists a right (equivalently, two-sided) invariant mean. We say that $\G$ is \e{co-amenable} if $\LOQ$ has a bounded left (equivalently, right or two-sided) approximate identity. We say that $\G$ is \e{discrete} if $\LOQ$ is unital, in which case we denote $\LOQ$ by $\ell^1(\G)$.

For general $\G$, the \e{left regular representation} $\lm:\LOQ\rightarrow\BLTQ$ is defined by
\begin{equation*}\lm(f)=(f\ten\id)(W), \ \ \ f\in\LOQ,\end{equation*}
and is an injective, completely contractive homomorphism from $\LOQ$ into $\BLTQ$. Then $\LIQH:=\{\lm(f) : f\in\LOQ\}''$ is the von Neumann algebra associated with the dual quantum group $\h{\G}$. Analogously, we have the \e{right regular representation} $\rho:\LOQ\rightarrow\BLTQ$ defined by
\begin{equation*}\rho(f)=(\id\ten f)(V), \ \ \ f\in\LOQ,\end{equation*}
which is also an injective, completely contractive homomorphism from $\LOQ$ into $\BLTQ$. Then $\LIQHP:=\{\rho(f) : f\in\LOQ\}''$ is the von Neumann algebra associated to the quantum group $\h{\G}'$. It follows that $\LIQHP=\LIQH'$ and $\LIQ\cap\LIQH=\LIQ\cap\LIQHP=\C1$ \cite[Proposition 3.4]{VD}.

If $G$ is a locally compact group, we let $\G_a=(\LI,\Gam_a,\vphi_a,\psi_a)$ denote the \e{commutative} quantum group associated with the commutative von Neumann algebra $\LI$, where the co-multiplication is given by $\Gam_a(f)(s,t)=f(st)$, and $\vphi_a$ and $\psi_a$ are integration with respect to a left and right Haar measure, respectively. The fundamental unitaries in this case are given by
$$W_a\xi(s,t)=\xi(s,s^{-1}t), \ \ V_a\xi(s,t)=\xi(st,t)\Delta(t)^{1/2}, \ \ \ \xi\in L^2(G\times G).$$
The dual $\h{\G}_a$ of $\G_a$ is the \e{co-commutative} quantum group $\G_s=(\LG,\Gam_s,\vphi_s,\psi_s)$, where $\LG$ is the left group von Neumann algebra with co-multiplication $\Gam_s(\lm(t))=\lm(t)\ten\lm(t)$, and $\vphi_s=\psi_s$ is Haagerup's Plancherel weight. Then $L^1(\G_a)$ is the usual group convolution algebra $\LO$, and $L^1(\G_s)$ is the Fourier algebra $A(G)$. It is known that every commutative locally compact quantum group is of the form $\G_a$ \cite[Theorem 2]{T}. By duality, every co-commutative locally compact quantum group is of the form $\G_s$.

For a locally compact quantum group $\G$, we let $R$ and $(\tau_t)_{t\in\R}$ denote the \textit{unitary antipode} and \textit{scaling group} of $\G$, respectively. The unitary antipode satisfies
\begin{equation}\label{e:antipode} (R\ten R)\circ\Gam = \Sigma\circ\Gamma\circ R,\end{equation}
where $\Sigma:\LIQ\oten\LIQ\rightarrow\LIQ\oten\LIQ$ denotes the flip map. The \textit{antipode} of $\G$ is $S=R\tau_{-\frac{i}{2}}$, and is a closed densely defined operator on $\LIQ$, whose domain we denote by $\mc{D}(S)$. Let $\LOQs$ be the subspace of $\LOQ$ defined by
\begin{equation*}\LOQs=\{f\in\LOQ : \exists \ g\in\LOQ\hs\text{s.t.}\hs g(x)=f^*\circ S(x)\hs\forall x\in\mc{D}(S)\},\end{equation*}
where $f^*(x)=\overline{f(x^*)}$, $x\in\LIQ$. There is an involution on $\LOQs$ given by $f^o=f^*\circ S$, such that $\LOQs$ becomes a Banach $*$-algebra under the norm $\norm{f}_*=\text{max}\{\norm{f},\norm{f^o}\}$. A quantum group $\G$ is said to be a \textit{Kac algebra} if $S=R$, and the modular element $\delta$ is affiliated to the center of $\LIQ$. In this case $\h{\G}$ is also a Kac algebra, and $\LOQ$ is a Banach $*$-algebra.

For general $\G$, we let $C_0(\G):=\overline{\hat{\lm}(\LOQH)}^{\norm{\cdot}}$ denote the \e{reduced quantum group $C^*$-algebra} of $\G$. We say that $\G$ is \e{compact} if $C_0(\G)$ is a unital $C^*$-algebra, in which case we denote $C_0(\G)$ by $C(\G)$. It is well-known that $\G$ is compact if and only if $\h{\G}$ is discrete. In general, the operator dual $M(\G):=C_0(\G)^*$ is a completely contractive Banach algebra containing $\LOQ$ as a norm closed two-sided ideal via the map $\LOQ\ni f\mapsto f|_{C_0(\G)}\in M(\G)$. The co-multiplication satisfies $\Gam(C_0(\G))\subseteq M(C_0(\G)\iten C_0(\G))$, where $M(C_0(\G)\ten^{\vee} C_0(\G))$ is the multiplier algebra of the $C^*$-algebra $C_0(\G)\iten C_0(\G)$.

We let $C_u(\G)$ be the \e{universal quantum group $C^*$-algebra} of $\G$, which is the universal enveloping $C^*$-algebra of the Banach $*$-algebra $\LOQHs$, and denote the canonical surjective $*$-homomorphism onto $C_0(\G)$ by $\Lambda_{\G}:C_u(\G)\rightarrow C_0(\G)$ \cite{K}. The space $C_u(\G)^*$ then has the structure of a unital completely contractive Banach algebra such that the map $\LOQ\rightarrow C_u(\G)^*$ given by the composition of the inclusion $\LOQ\subseteq M(\G)$ and $\Lambda_{\G}^*:M(\G)\rightarrow C_u(\G)^*$ is a completely isometric homomorphism, and it follows that $\LOQ$ is a norm closed two-sided ideal in $C_u(\G)^*$ \cite[Proposition 8.3]{K}. There is a universal co-multiplication $\Gam_u:C_u(\G)\rightarrow M(C_u(\G)\ten^{\vee}C_u(\G))$ satisfying
$$(\Lambda_{\G}\ten\Lambda_{\G})\circ\Gam_u=\Gam\circ\Lambda_{\G}.$$
As above, throughout the paper we will use the same notation for a strict maps $A\rightarrow B$ between $C^*$-algebras $A$ and $B$ and their (unique) strictly continuous extensions $M(A)\rightarrow M(B)$. See \cite{Lance} for details.

A \textit{(right) unitary co-representation} of $\G$ is a unitary $U\in M(\mc{K}(H)\ten^{\vee}C_u(\G))$ satisfying $(\id\ten\Gam_u)(U)=U_{12}U_{13}$. We will also abuse terminology and refer to the reduced version of $U$, $(\id\ten\Lambda_{\G})(U)\in M(\mc{K}(H)\ten^{\vee}C_0(\G))$, as a unitary co-representation. Left unitary co-representations are defined analogously, and it follows that $W\in M(C_0(\G)\ten^{\vee}C_0(\h{\G}))$ and $V\in M(C_0(\h{\G}')\ten^{\vee}C_0(\G))$ are left and right unitary co-representations of $\G$, respectively. The \textit{conjugation co-representation} is given by $W\sigma V\sigma \in M(C_0(\G)\ten^{\vee}\mc{K}(H))$. From \cite{K} it is known that there exists a one-to-one correspondence between unitary co-representations in $M(\mc{K}(H)\ten^{\vee}C_u(\G))$ and non-degenerate $*$-homomorphisms $\LOQHs\rightarrow\mc{B}(H)$.

For any locally compact quantum group $\G$, the GNS Hilbert space of the dual left Haar weight $\h{\vphi}$ on $\LIQH$ can be identified with $\LTQ$. In this case, the product of the modular conjugations arising from the left Haar weights $\vphi$ and $\h{\vphi}$ defines a unitary $U:=\h{J}J$ on $\LTQ$ intertwining the left and right regular representations via $\rho(f)=U\lm(f)U^*$, $f\in\LOQ$ (it will be clear from context whether the notation $U$ refers to $\h{J}J$ or to a unitary co-representation). At the level of the fundamental unitaries, this relation becomes
\begin{equation}\label{e:WV}V=\sigma(1\ten U)W(1\ten U^*)\sigma,\end{equation}
where $\sigma:\LTQ\ten\LTQ\rightarrow\LTQ\ten\LTQ$ is the flip map. We also record the adjoint formulas $(\widehat{J}\ten J)W(\widehat{J}\ten J)=W^*$ and $(J\ten\widehat{J})V(J\ten\widehat{J})=V^*$, which, in terms of the unitary antipodes reads $(R\ten\h{R})(W)=W$ and $(\h{R}\ten R)(V)=V$. The left and right fundamental unitaries of the dual $\h{\G}$ are given by
\begin{equation}\label{e:hat}\h{W}=\sigma W^*\sigma, \ \ \ \h{V}=(J\ten J)W(J\ten J).\end{equation}

%Equations (\ref{e:WV}) and (\ref{e:hat}) allow us to interpret the pentagonal relation as a commutation relation between $\G$ and $\h{\G}$:
%\begin{equation}\label{e:commutation}\h{W}_{23}W_{13}=W_{13}\h{W}_{23}W_{12}, \ \ \ \h{V}_{12}V_{13}=V_{13}\h{V}_{12}\h{W}_{23}^*.\end{equation}
 
An element $\hat{b}\in\LIQH$ is said to be a \e{completely bounded left multiplier} of $\LOQ$ if $\hat{b}\lm(f)\in\lm(\LOQ)$ for all $f\in\LOQ$ and the induced map
\begin{equation*}m_{\hat{b}}^l:\LOQ\ni f\mapsto\lm^{-1}(\hat{b}\lm(f))\in\LOQ\end{equation*}
is completely bounded on $\LOQ$. We let $\McbQl$ denote the space of all completely bounded left multipliers of $\LOQ$, which is a completely contractive Banach algebra with respect to the norm
\begin{equation*}\norm{[\hat{b}_{ij}]}_{M_n(\McbQl)}=\norm{[m^l_{\hat{b}_{ij}}]}_{cb}.\end{equation*}
Completely bounded right multipliers are defined analogously and we denote by $\McbQr$ the corresponding completely contractive Banach algebra.

Given $\hat{b}\in\McbQl$, the adjoint $\Theta^l(\hat{b}):=(m_{\hat{b}}^l)^*$ defines a normal completely bounded left $\LOQ$-module map on $\LIQ$. When $\hat{b}=\lm(f)$, for some $f\in\LOQ$, the map $\Theta^l(\lm(f))$ is nothing but the convolution action of $\LOQ$ on $\LIQ$, that is,
$$\Theta^l(\lm(f))(x)=x\star f=(f\ten\id)\Gam(x), \ \ \ x\in\LIQ.$$
For any $\hat{b}\in\McbQl$, the map $\Theta^l(\hat{b})$ extends to a unique normal completely bounded $\LIQHP$-bimodule map on $\BLTQ$ which leaves $\LIQ$ globally invariant \cite{JNR}, and the corresponding representation
\begin{equation*}\label{e:Theta}\Theta^l:\McbQl\cong\mc{CB}^{\sigma,\LIQ}_{\LIQHP}(\BLTQ)\end{equation*}
yields a completely isometric isomorphism of completely contractive Banach algebras. It is known that $\McbQl$ is a dual space (see \cite[Theorem 3.5]{HNR2}, for instance), and that $\Theta^l$ is a weak*-weak* homeomorphism \cite[Theorem 6.1]{D}. Similar constructions exist for right multipliers, see \cite{HNR2} for details.

\section{QSIN Quantum Groups}\label{s:QSIN}

If $G$ is a locally compact group and $p\in[1,\infty]$, then $G$ acts by conjugation on $L^p(G)$ via
$$\beta_p(s)f(t)=f(s^{-1}ts)\Delta(s)^{1/p}, \ \ \ s,t\in G, \ f\in L^p(G).$$
When $p=2$, we obtain a strongly continuous unitary representation $\beta_2:G\rightarrow\BLT$ satisfying 
$$\beta_2(s)=\lm(s)\rho(s)=(\delta_s\ten\id)(W_a\sigma V_a\sigma), \ \ \ s\in G,$$
The group $G$ is said to \textit{quasi-SIN (QSIN)} if there exists a bounded approximate identity (BAI) $(f_i)$ for $\LO$ satisfying
\begin{equation}\label{e:QSIN}\norm{\beta_{1}(s)f_i-f_i}\rightarrow0,  \ \ \ s\in G.\end{equation}
Every amenable group is QSIN \cite[Theorem 3]{LR}, and, trivially, any SIN group is QSIN. In the latter case, $G$ possesses a conjugation invariant neighbourhood basis of the identity which generates a bounded approximate identity (b.a.i.) $(f_i)$ for $\LO$ invariant under $\beta_1$. 

For a QSIN group $G$, one can always choose a b.a.i. $(f_i)$ consisting of states. From the quantum group perspective, such a b.a.i. generates a net of unit vectors $(\xi_i)$ in $\LT$ which is simultaneously asymptotically invariant under the conjugation co-representation $W_a\sigma V_a\sigma$ and the right fundamental unitary of the dual $\widehat{V_a}=V_s$. This motivates the following. 

\begin{defn}\label{d:QSIN} Let $\G$ be a locally compact quantum group. We say that $\G$ is \textit{quasi-SIN (or QSIN)} if there exists a net $(\xi_i)$ of unit vectors in $\LTQ$ such that 
\begin{enumerate}[label=(\roman*)]
\item $\norm{W\sigma V\sigma\eta\ten\xi_i-\eta\ten\xi_i}\rightarrow0$, $\eta\in\LTQ$;
\item $\norm{\h{V}\xi_i\ten\eta-\xi_i\ten\eta}\rightarrow0$, $\eta\in\LTQ$;
\end{enumerate}
\end{defn}

\begin{examples}${}$\vskip5pt
\begin{enumerate}
\item[(1)] By definition, a commutative quantum group $\G_a=\LI$ is QSIN precisely when $G$ is QSIN.
\item[(2)] A co-commutative quantum group $\G_s=VN(G)$ is QSIN if and only if $G$ is amenable. In this case, $V_s=W_a$, so that $W_s\sigma V_s\sigma=W_sW_s^*=1$, that is, the conjugation co-representation is trivial. The QSIN condition therefore reduces to co-amenability of $\G_s$, that is, amenability of $G$. 
\item[(3)] Any discrete Kac algebra is QSIN as the unit vector $\xi:=\Lambda_\vphi(1)$ satisfies conditions (i)--(iii), where $\vphi$ is the Haar trace on the compact dual. 
\item[(4)] Any co-amenable compact Kac algebra is QSIN. In this case, one takes a b.a.i. $(f_i)$ for $\LOQ$ and performs the standard averaging technique (see \cite[\S4]{Ru4}) to make it conjugation invariant. 
\end{enumerate}
\end{examples}

The QSIN condition for quantum groups has appeared implicitly in \cite{BY,RX,Wil}. In particular, Ruan and Xu proved that $\LOQ$ is relatively 1-biflat for any Kac algebra $\G$ whose dual $\h{\G}$ is QSIN \cite[Theorem 4.3]{RX}. Here, relative 1-biflatness means the existence of a completely contractive $\LOQ$-bimodule left inverse to the co-multiplication $\Gamma$. Their construction is as follows: let $(\xi_i)$ be a net witnessing the QSIN property of $\h{\G}$, and define
\begin{equation}\label{e:map}\Phi_i:\LIQ\oten\LIQ\ni X\mapsto (\om_{\xi_i}\ten\id)W^*(U^*\ten 1)X(U\ten 1)W\in\LIQ.\end{equation}
Any weak* cluster point $\Phi$ of the net $(\Phi_i)$ in $\mc{CB}(\LIQ\oten\LIQ,\LIQ)$ will be a $\LOQ$-bimodule inverse to $\Gam$: each $\Phi_i$ is a left module map so $\Phi$ is as well; the conjugation invariance implies $\Phi\circ\Gam=\id_{\LIQ}$, while the approximate identity condition implies that $\Phi$ is a right module map. We provide a quicker argument using \cite[Proposition 5.1]{C}.

\begin{prop}\label{p:biflat} Let $\G$ be a locally compact quantum group whose dual $\h{\G}$ is QSIN. Then $\LOQ$ is relatively 1-biflat.\end{prop}

\begin{proof} Let $(\xi_i)$ be a net of unit vectors in $\LTQ$ witnessing the QSIN property of $\h{\G}$, and without loss of generality, suppose $M:=w^*-\lim_i\om_{\xi_i}\in\BLTQ^*$. Then $M$ is a state satisfying
$$\la M,\om_\eta\star x\ra=\lim_i\la V(x\ten 1)V^*\xi_i\ten\eta,\xi_i\ten\eta\ra=\lim_i\la(x\ten 1)\xi_i\ten\eta,\xi_i\ten\eta\ra=\la M,x\ra\la\om_\eta,1\ra$$
for every $\eta\in\LTQ$ and $x\in\LIQ$. Hence, $M|_{\LIQ}$ is a right invariant mean on $\LIQ$. Moreover, as $\sigma\h{V}\sigma\in\LIQH\oten\LIQ'$, we have
\begin{align*}\la M, (\id\ten\om_\eta)\h{V}'(x\ten 1)\h{V}'^*\ra&=\lim_i\la\h{V}'(x\ten 1)\h{V}'^*\xi_i\ten\eta,\xi_i\ten\eta\ra\\
&=\lim_i\la V^*(1\ten x)V\eta\ten\xi_i,\eta\ten\xi_i\ra\\
&=\lim_i\la\h{W}(1\ten x)\h{W}^*(U^*\eta\ten\xi_i),U^*\eta\ten\xi_i\ra\\
&=\lim_i\la\h{W}\sigma\h{V}\sigma(1\ten x)\sigma\h{V}^*\sigma\h{W}^*(U^*\eta\ten\xi_i),U^*\eta\ten\xi_i\ra\\
&=\lim_i\la (1\ten x)U^*\eta\ten\xi_i,U^*\eta\ten\xi_i\ra\\
&=\la M,x\ra\la\om_\eta,1\ra.\end{align*}
The hypotheses of \cite[Proposition 5.1]{C} are then satisfied, yielding the claim.

\end{proof}

\begin{remark} Proposition \ref{p:biflat} was established for co-commutative $\G$ in \cite{ARS}, by a different approach than \cite{RX}. Even in the co-commutative setting, the converse of Proposition \ref{p:biflat} is still open in full generality, although it is known to hold for any almost connected group, as well as any totally disconnected group \cite{CT}.
\end{remark}

\begin{cor} Let $\G$ be a QSIN locally compact quantum group. If, in addition, $\G$ is either compact or discrete, then $\G$ is Kac algebra.
\end{cor}

\begin{proof} Suppose that $\G$ is QSIN. If $\G$ were also compact then Proposition \ref{p:biflat} implies that $\LOQH$ is relatively 1-biflat over itself, upon which \cite[Corollary 5.4]{C} implies that $\G$ is a Kac algebra. If $\G$ were discrete, then again by Proposition \ref{p:biflat} $\LOQH$ is relatively 1-biflat which by \cite[Theorem 1.1]{CLR} forces $\G$ to be a Kac algebra.
\end{proof}

\begin{remark} It would be interesting to see whether the QSIN condition for a general quantum group $\G$ implies that $\G$ is necessarily a Kac algebra, or at least that $\G$ has trivial scaling group.\end{remark}

\subsection{Examples arising from the bicrossed product construction} Let $G, G_1$ and $G_2$ be locally compact groups with fixed left Haar measures for which there exists a homomorphism $i:G_1\rightarrow G$ and an anti-homomorphism $j:G_2\rightarrow G$ which have closed ranges and are homeomorphisms onto these ranges. Suppose further that $G_1\times G_2\ni (g,s)\mapsto i(g)j(s)\in G$ is a homeomorphism onto an open subset of $G$ having complement of measure zero. Then $(G_1,G_2)$ is said to be a \textit{matched pair} of locally compact groups \cite[Definition 4.7]{VV}. Any matched pair $(G_1,G_2)$ determines a matched pair of actions $\alpha:G_1\times G_2\rightarrow G_2$ and $\beta:G_1\times G_2\rightarrow G_1$ satisfying mutual co-cycle relations \cite[Lemma 4.9]{VV}. It is known that the von Neumann crossed product $G_1\rtimes_\alpha L^\infty(G_2)$ admits a quantum group structure, called the \textit{bicrossed product} of the matched pair $(G_1,G_2)$. The von Neumann algebra of the dual quantum group is given by the crossed product $L^{\infty}(G_1)^\beta\ltimes G_2$, and therefore, following \cite{Majid}, we denote the bicrossed product quantum group by $VN(G_1)^\beta\bowtie_\alpha L^\infty(G_2)$. 

Below we present sufficient conditions on the matched pair $(G_1,G_2)$ under which the bicrossed product is QSIN. In preparation we collect some useful formulae from \cite[\S4]{VV}, to which we refer the reader for details. To ease the presentation we suppress the notations $i$ and $j$ for the embeddings into $G$.

The co-cycle relations between $\alpha$ and $\beta$ are
\begin{align*}\label{e:co-cycle}\alpha_g(st)&=\alpha_{\beta_t(g)}(s)\alpha_g(t),  \ \ \ \beta_s(gh)=\beta_{\alpha_h(g)}\beta_s(h)\\
\alpha_{gh}(s)&=\alpha_g(\alpha_h(s)) \ \ \ \ \  \beta_{st}(g)=\beta_s(\beta_t(g))\end{align*}
which hold for almost every $(g,s),(h,t)$ in $G_1\times G_2$ (see \cite[Lemma 4.9]{VV} for a precise statement). We also record
$$\alpha_g(e)=e=\beta_s(e), \ \ \ \alpha_e(s)=s, \ \ \ \beta_e(g)=g, \ \ \ g\in G_1, \ s\in G_2.$$
The fundamental unitary $W$ satisfies
$$W^*\xi(g,s,h,t)=\xi(\beta_{t}(h)^{-1}g,s,h,\alpha_{\beta_t(h)^{-1}g}(s)t), \ \ \ \xi\in L^2(G_1\times G_2\times G_1\times G_2).$$
Letting $\Delta$, $\Delta_1$, and $\Delta_2$ denote the modular functions for the groups $G$, $G_1$, and $G_2$, respectively, the modular conjugations $J$ and $\h{J}$ of the dual Haar weights satisfy
\begin{align*}J\xi(g,s)&=\Delta(\beta_s(g))^{-1/2}\Delta_1(\beta_s(g)g^{-1})^{1/2}\Delta_2(\alpha_g(s)s^{-1})^{1/2}\overline{\xi}(g^{-1},\alpha_g(s))\\
\h{J}\xi(g,s)&=\Delta(\alpha_g(s))^{1/2}\Delta_1(\beta_s(g)g^{-1})^{1/2}\Delta_2(\alpha_g(s)s^{-1})^{1/2}\overline{\xi}(\beta_s(g),s^{-1}),\end{align*}
for all $\xi\in L^2(G_1\times G_2)$. Let $\Psi:G_2\times G_1\rightarrow(0,\infty)$ be the (continuous) function determined by the Radon-Nikodym derivatives $\Psi(s,g):=\frac{d\beta_s(g)}{dg}$. It follows that 
$$\Psi(s,g)=\Delta(\alpha_g(s))\Delta_1(\beta_s(g)g^{-1})\Delta_2(\alpha_g(s)), \ \ \ g\in G_1, \ s\in G_2.$$
The action $\beta$ determines a unitary representation of $G_2$ on $L^2(G_1)$ given by
$$v(s)\xi(g)=\bigg(\frac{d\beta_{s^{-1}}(g)}{dg}\bigg)^{1/2}\xi(\beta_{s^{-1}}(g)), \ \ \ \xi\in L^2(G_1).$$
A matched pair $(G_1,G_2)$ is said to be \textit{modular} if 
$$\frac{\Psi(s,g)}{\Psi(s,e)}=1 \ \ \ \textnormal{and} \ \ \ \frac{\Delta_1(\beta_s(g))}{\Delta_1(g)}=\frac{\Delta_2(\alpha_g(s))}{\Delta_2(s)}, \ \ \ g\in G_1, \ s\in G_2.$$
It is known that the bicrossed product $VN(G_1)^\beta\bowtie_\alpha L^\infty(G_2)$ is a Kac algebra if and only if the matched pair $(G_1,G_2)$ is modular \cite[Theorem 2.12]{Majid}, \cite[Proposition 4.16]{VV}. 

\newpage

\begin{prop}\label{p:bicrossed} Let $(G_1,G_2)$ be a modular matched pair of locally compact groups such that
\begin{enumerate}[label=(\roman*)]
\item $G_1$ is amenable, witnessed by a net $(\xi_i)$ of unit vectors in $L^2(G_1)$ satisfying
$$\norm{\rho(s)\xi_i-\xi_i}, \ \ \norm{v(s)\xi_i-\xi_i}\rightarrow0, \ \ \ s\in G,$$
uniformly on compacta.
\item $G_2$ is discrete.
\item $\Delta|_{G_1}=\Delta_1$.
\end{enumerate}
Then $VN(G_1)^\beta\bowtie_\alpha\ell^\infty(G_2)$ is QSIN.
\end{prop}

\begin{proof} Let $(\xi_i)$ be a net of unit vectors in $L^2(G_1)$ satisfying $(i)$. We first verify condition $(ii)$ in Definition \ref{d:QSIN}. Let $\eta\in C_c(G_1\times G_2)$. Then
\begin{align*}&\norm{W^*(J_1\xi_i\ten\delta_{e_2}\ten\eta)-J_1\xi_i\ten\delta_{e_2}\ten\eta}^2\\
&=\iiiint|J_1\xi_i(\beta_{t}(h)^{-1}g)\delta_{e_2}(s)\eta(h,\alpha_{\beta_t(h)^{-1}g}(s)t)-J_1\xi_i(g)\delta_{e_2}(s)\eta(h,t)|^2 \ dg \ ds \ dh \ dt\\
&=\iiint|J_1\xi_i(\beta_{t}(h)^{-1}g)-J_1\xi_i(g)|^2|\eta(h,t)|^2 \ dg \ dh \ dt\\
&=\iint\norm{\lm(\beta_{t}(h))J_1\xi_i-J_1\xi_i}^2|\eta(h,t)|^2 \ dh \ dt\\
&=\iint\norm{\rho(\beta_{t}(h))\xi_i-\xi_i}^2|\eta(h,t)|^2 \ dh \ dt\\
&\rightarrow0.\end{align*}
Now, since $\h{V}=(J\ten J)W(J\ten J)$, the modularity of the pair $(G_1,G_2)$ together with condition $(iii)$ imply 
$$J(\xi_i\ten\delta_{e_2})(g,s)=\Delta(\beta_s(g))^{-1/2}\overline{\xi_i}(g^{-1})\delta_{e_2}(\alpha_g(s))=\Delta_1(g)^{-1/2}\overline{\xi_i}(g^{-1})\delta_{e_2}(s)=J_1\xi\ten\delta_{e_2}(g,s).$$
Thus,
$$\norm{\h{V}^*(\xi_i\ten\delta_{e_2}\ten\eta)-\xi_i\ten\delta_{e_2}\ten\eta}^2=\norm{W^*(J_1\xi_i\ten\delta_{e_2}\ten J\eta)-J_1\xi_i\ten\delta_{e_2}\ten J\eta}^2\rightarrow0, $$
and Definition \ref{d:QSIN} $(ii)$ is satisfied.

Now, using the co-cycle properties of $\alpha$ and $\beta$ \cite[Lemma 4.9]{VV} together with the definitions of $W$ and $\h{J}$, one sees that
\begin{align*}&(\h{J}\ten\h{J})W^*(\h{J}\ten\h{J})\xi(g,s,h,t)\\
&=\Delta(\alpha_{h^{-1}\beta_s(g)}(s^{-1}))^{1/2}\xi(\beta_{\alpha_{\beta_s(g)}(s^{-1})}(h^{-1})g,s,\beta_{\alpha_{\beta_s(g)}(s^{-1})}(h^{-1})^{-1},t\alpha_{h^{-1}\beta_s(g)}(s^{-1})^{-1})\end{align*}
for all $\xi\in L^2(G_1\times G_2\times G_1\times G_2)$. Then for any $\eta\in C_c(G_1\times G_2)$, and any $i$, we have
\begin{align*}&\la W\sigma V\sigma(\eta\ten \xi_i\ten\delta_{e_2}),\eta\ten \xi_i\ten\delta_{e_2}\ra=\la(\h{J}\ten\h{J})W^*(\h{J}\ten\h{J})(\eta\ten \xi_i\ten\delta_{e_2}),W^*(\eta\ten \xi_i\ten\delta_{e_2})\ra\\
&=\iiiint\Delta(\alpha_{h^{-1}\beta_s(g)}(s^{-1}))^{1/2} \ \eta(\beta_{\alpha_{\beta_s(g)}(s^{-1})}(h^{-1})g,s) \ \xi_i(\beta_{\alpha_{\beta_s(g)}(s^{-1})}(h^{-1})^{-1}) \ \delta_{e_2}(t(\alpha_{h^{-1}\beta_s(g)}(s^{-1}))^{-1})\\
&\times \overline{\eta}(\beta_t(h)^{-1}g,s) \ \overline{\xi_i}(h) \ \delta_{e_2}(\alpha_{\beta_t(h)^{-1}g}(s)t) \ dg \ ds \ dh \ dt.\end{align*}
We must therefore have $t=\alpha_{h^{-1}\beta_s(g)}(s^{-1})$. Using the co-cycle properties \cite[Lemma 4.9]{VV} this forces
$$\beta_t(h)^{-1}=\beta_{\alpha_h(t)}(h^{-1})=\beta_{\alpha_h(\alpha_{h^{-1}\beta_s(g)}(s^{-1}))}(h^{-1})=\beta_{\alpha_{\beta_s(g)}(s^{-1})}(h^{-1})=\beta_{\alpha_g(s)^{-1}}(h^{-1}).$$
Also, the modularity condition and the discreteness of $G_2$ imply that
$$\Delta(\alpha_g(s))=\Delta(s), \ \ \ g\in G_1, \ s\in G_2.$$
Plugging this into the integral we obtain
$$\iiint\Delta(s^{-1})^{1/2}|\eta(\beta_{\alpha_g(s)^{-1}}(h^{-1})g,s)|^2 \ \xi_i(\beta_{\alpha_g(s)^{-1}}(h^{-1})^{-1}) \ \overline{\xi_i}(h)  \ dg \ ds \ dh.$$
Upon the substitution $h'=\beta_{\alpha_g(s)^{-1}}(h^{-1})$, the modularity condition implies 
\begin{align*}dh'&=\frac{d\beta_{\alpha_g(s)^{-1}}(h^{-1})}{dh^{-1}}\Delta_1(h^{-1})dh\\
&=\Psi(\alpha_g(s)^{-1},h^{-1})\Delta_1(h^{-1})dh\\
&=\Delta(\alpha_{h^{-1}}(\alpha_g(s)^{-1}))\Delta_1(h^{-1})dh\\
&=\Delta(\alpha_{h^{-1}}(\alpha_{\beta_s(g)}(s^{-1})))\Delta_1(h^{-1})dh\\
&=\Delta(\alpha_{h^{-1}\beta_s(g)}(s^{-1}))\Delta_1(\beta_{\alpha_g(s)^{-1}}(h^{-1}))dh\\
&=\Delta(s^{-1})\Delta_1(h')dh\end{align*}
Making the substitution (and replacing $h'$ with $h$) we get
\begin{align*}&\iiint\Delta(s)^{1/2}|\eta(hg,s)|^2 \ \xi_i(h^{-1}) \ \overline{\xi_i}(\beta_{\alpha_g(s)}(h)^{-1})\Delta_1(h)^{-1}  \ dg \ ds \ dh\\
&=\iiint\Delta(s)^{1/2}|\eta(g,s)|^2 \ \xi_i(h^{-1}) \ \overline{\xi_i}(\beta_{\alpha_{h^{-1}g}(s)}(h)^{-1})\Delta_1(h)^{-1}  \ dg \ ds \ dh.\end{align*}
But
$$\beta_{\alpha_{h^{-1}g}(s)}(h)=\beta_{\alpha_{h^{-1}}(\alpha_g(s))}(h)=\beta_{\alpha_g(s)}(h^{-1})^{-1},$$
so the integral becomes
\begin{align*}&\iiint\Delta(s)^{1/2}|\eta(g,s)|^2 \ \xi_i(h^{-1}) \ \overline{\xi_i}(\beta_{\alpha_g(s)}(h^{-1}))\Delta_1(h)^{-1}  \ dg \ ds \ dh.\\
&=\iiint\Delta(s)^{1/2}|\eta(g,s)|^2 \ \overline{J_1\xi_i}(h) \ \overline{v(\alpha_g(s)^{-1})\xi_i}(h^{-1})\bigg(\frac{d\beta_{\alpha_g(s)}(h^{-1})}{dh^{-1}}\bigg)^{-1/2}\Delta_1(h)^{-1/2}  \ dg \ ds \ dh.\\
&=\iiint\Delta(s)^{1/2}|\eta(g,s)|^2 \ \overline{J_1\xi_i}(h) \ \overline{v(\alpha_g(s)^{-1})\xi_i}(h^{-1})\Psi(\alpha_g(s),h^{-1})^{-1/2}\Delta_1(h)^{-1/2}  \ dg \ ds \ dh.\\
&=\iiint\Delta(s)^{1/2}|\eta(g,s)|^2 \ \overline{J_1\xi_i}(h) \ \overline{v(\alpha_g(s)^{-1})\xi_i}(h^{-1})\Delta(\alpha_{h^{-1}}(\alpha_g(s)))^{-1/2}\Delta_1(h)^{-1/2}  \ dg \ ds \ dh.\\
&=\iiint\Delta(s)^{1/2}|\eta(g,s)|^2 \ \overline{J_1\xi_i}(h) \ \overline{v(\alpha_g(s)^{-1})\xi_i}(h^{-1})\Delta(s)^{-1/2}\Delta_1(h)^{-1/2}  \ dg \ ds \ dh.\\
&=\iiint|\eta(g,s)|^2 \ \overline{J_1\xi_i}(h) \ J_1(v(\alpha_g(s)^{-1})\xi_i)(h) \ dg \ ds \ dh.\\
&=\iint|\eta(g,s)|^2 \ \la J_1(v(\alpha_g(s)^{-1})\xi_i),J_1\xi_i\ra \ dg \ ds.\\
&=\iint|\eta(g,s)|^2 \ \la \xi_i,v(\alpha_g(s)^{-1})\xi_i\ra \ dg \ ds.\\
&\rightarrow\norm{\eta}^2\end{align*}
It follows that the net $(\xi_i\ten\delta_{e_2})$ satisfies Definition \ref{d:QSIN}, and $VN(G_1)^\beta\bowtie_\alpha \ell^\infty(G_2)$ is QSIN.
\end{proof}

\begin{examples}${}$\vskip5pt
\begin{enumerate}
\item[(1)] Let $(G_1,G_2)$ be a matched pair. If one of the actions is trivial, then the pair is modular by \cite[Corollary 4.17]{VV}. Thus, if $\beta$ is trivial, $G_1$ is amenable, $G_2$ is discrete, and $\Delta|_{G_1}=\Delta_1$, then $VN(G_1)^\beta\bowtie_\alpha\ell^\infty(G_2)$ is a QSIN quantum group. In particular, if $G_1$ is an amenable discrete group acting by homomorphisms on a discrete group $G_2$, then $VN(G_1)^\beta\bowtie_\alpha\ell^\infty(G_2)$ is QSIN. In this case, the ambient group $G$ is $G_1\times G_2$ with the multiplication $(g,s)\cdot(h,t)=(gh,\alpha_g(t)s)$.

\item[(2)] Let $H(\R)$ denote the Heisenberg group of $3\times 3$ upper triangular matrices with ones on the diagonal. By \cite[Example 4.1]{Majid} the actions
$$\alpha_g(s)=(1+g(s^{-1}-1))^{-1}, \ \ \ \beta_s(g)=(1+s(g^{-1}-1))^{-1}, \ \ \ g,s\in H(\R),$$
determine a modular matched pair $(H(\R),H(\R))$. Here, we put the discrete topology on the second copy of $H(\R)$. As above, the ambient group $G$ is $H(\R)\times H(\R)$ with the Zappa--Sz\'{e}p product $(g,s)\cdot(h,t)=(\beta_t(g)h,\alpha_g(t)s)$. In this case $i$ and $j$ are simply the coordinate maps. Letting $g=(g_1,g_2,g_3)$ denote the matrix
$$\begin{bmatrix}1 & g_1 & g_2\\0 & 1 & g_3\\0 & 0 & 1\end{bmatrix}\in H(\R),$$
it follows that 
$$\beta_s(g)=(g_1,g_2+s_1g_3,g_3)=(s_1,0,0)\cdot(g_1,g_2,g_3)\cdot(s_1,0,0)^{-1}, \ \ \ s,g\in H(\R).$$
Similarly for $\alpha$. In particular, the actions preserve the Haar measures on the respective copies of $H(\R)$. It follows that $\Delta\equiv 1$. 

Let $(\xi_i)$ be a net of unit vectors in $L^2(H(\R))$ satisfying $\norm{\lm(s)\xi_i-\xi_i},\norm{\rho(s)\xi_i-\xi_i}\rightarrow 0$ uniformly on compacta. Such a net exists by approximating a two-sided invariant mean on $L^\infty(H(\R))$ by normal states in $L^1(H(\R))$ and then taking square roots. Then
$$\norm{v(s)\xi_i-\xi_i}^2=\norm{\lm(s_1,0,0)\rho(s_1,0,0)\xi_i-\xi_i}^2\rightarrow 0.$$
Thus, the conditions of Proposition \ref{p:bicrossed} are satisfied, and $VN(H(\R))^\beta\bowtie_\alpha \ell^\infty(H(\R))$ is QSIN.
\end{enumerate}
\end{examples}

\section{Gilbert Factorization}\label{s:Gilbert} 

We now establish the Gilbert factorization $\Gam^{2,r}_{\LOQ}(\LOQ,\LIQ)\cong M^l_{cb}(\LOQH)$ for QSIN and bounded degree quantum groups. We split the two cases into separate subsections.

\subsection{QSIN Quantum Groups}

In \cite{LSS} it was shown that for any QSIN group $G$, pointwise multiplication $A(G)\hten A(G)\rightarrow C_0(G)$ is a complete quotient map. Quite recently a similar result was established for a large class of quantum groups $\G$ \cite{BY}, including those whose dual $\h{\G}$ is QSIN in our terminology. It follows that for $\G$ whose dual $\h{\G}$ is QSIN, the induced map 
\begin{equation}\label{e:m}\LOQ\hten_{\LOQ}\LOQ\rightarrow C_0(\h{\G}')\end{equation}
is also a complete quotient, where $\LOQ\hten_{\LOQ}\LOQ$ is the module Haagerup tensor product, defined as the quotient of $\LOQ\hten\LOQ$ by the closed subspace $\la f\star g\ten h-f\ten g\star h\mid f,g,h\in\LOQ\ra$. Using different techniques, we build on these results by showing that in this case (\ref{e:m}) is actually a completely isometric isomorphism. 

\begin{thm}\label{t:Gilbert} Let $\G$ be a locally compact quantum group whose dual $\h{\G}$ is QSIN. Then
$$\Gam^{2,r}_{\LOQ}(\LOQ,\LIQ)\cong M(\h{\G}')$$
completely isometrically and weak*-weak* homeomorphically. Thus, 
$$\LOQ\hten_{\LOQ}\LOQ\cong C_0(\h{\G}')$$
completely isometrically.\end{thm}

\begin{proof} The composition
\begin{equation}\label{e:m_h} m_h:\LOQ\hten\LOQ\xrightarrow{\rho\ten\rho}C_0(\h{\G}')\hten C_0(\h{\G}')\xrightarrow{m} C_0(\h{\G}')\end{equation}
is a complete contraction, and its adjoint $(m_h)^*:M(\h{\G}')\rightarrow\LIQ\whten\LIQ$ satisfies
$$\la (m_h)^*(\hat{\mu}'),f\ten g\ra=\la\hat{\mu}',\rho(f)\rho(g)\ra=\la\hat{\mu}',\rho(f\star g)\ra=\la\Gamma(\rho_*(\hat{\mu}')),f\ten g\ra$$
for all $f,g\in\LOQ$, where $\rho_*(\hat{\mu}')=(\hat{\mu}'\ten\id)(V)$. Thus, $(m_h)^*(\hat{\mu}')=\Gamma(\rho_*(\hat{\mu}'))$, so under the canonical identification (\ref{e:Haa})
$$\Gam^{2,r}_{\LOQ}(\LOQ,\LIQ)=(\LOQ\hten_{\LOQ}\LOQ)^*,$$
the adjoint of (\ref{e:m_h}) defines a complete contraction $(m_h)^*:M(\h{\G}')\rightarrow\Gam^{2,r}_{\LOQ}(\LOQ,\LIQ)$. 

The idea is to show that any cluster point of (\ref{e:map}) maps $\LIQ\whten\LIQ$ completely contractively into $M(\h{\G}')$. Let $(\xi_i)_{i\in I}$ be a net witnessing the QSIN property of $\h{\G}$. By passing to a subnet, we may assume without loss of generality that $M=w^*-\lim_{i}\om_{\xi_i}\in\BLTQ^*$ is a state. Define 
$$\Phi:\LIQ\oten\LIQ\ni X\mapsto (M\ten\id)W^*(U^*\ten 1)X(U\ten 1)W\in\LIQ.$$
Then, as in \cite{RX}, $\Phi$ is an $\LOQ$-bimodule left inverse to $\Gam$. On simple tensors, 
$$\la\Phi(x\ten y),f\ra=\la(M\ten\id)((U^*xU\ten 1)\Gam(y)),f\ra=\la M,U^*xU(f\star y)\ra,$$
for all $x,y\in\LIQ$, $f\in\LOQ$. 

Let $X\in\LIQ\whten\LIQ\subseteq\LIQ\oten\LIQ$. Then there exist families $(x_j),(y_j)$ in $\LIQ$ such that $\sum_jx_jx_j^*,\sum_jy_j^*y_j<\infty$ and $X=\sum_jx_j\ten y_j$. Condition (ii) of the QSIN property of $\h{\G}$ implies
$$\norm{V(\xi_i\ten\eta)-\xi_i\ten\eta}\rightarrow0.$$
Therefore, if $f=\om_{\eta,\zeta}$, for some $\eta,\zeta\in\LTQ$, then 
\begin{align*}&\lim_{i}|\la\sum_j(U^*x_jU(f\star y_j)\xi_i-U^*x_jU\rho(f)y_j\xi_i),\xi_i\ra|\\
&=\lim_{i}|\la\sum_j(\id\ten\zeta^*)((U^*x_jU\ten 1)V(y_j\ten 1)V^*(\xi_i\ten\eta)-(U^*x_jU\ten 1)V(y_j\ten 1)\xi_i\ten\eta),\xi_i\ra|\\
&=\lim_{i}|\la\sum_j(U^*x_jU\ten 1)V(y_j\ten 1)(V^*(\xi_i\ten\eta)-\xi_i\ten\eta),\xi_i\ten\zeta\ra|\\
&\leq\lim_{i}\norm{\sum_j(U^*x_jU\ten 1)V(y_j\ten 1)}\norm{V^*(\xi_i\ten\eta)-\xi_i\ten\eta}\norm{\xi_i\ten\zeta}\\
&\leq\lim_{i}\norm{\sum_jx_jx_j^*}^{1/2}\norm{\sum_{j}y_j^*y_j}^{1/2}\norm{V^*(\xi_i\ten\eta)-\xi_i\ten\eta}\norm{\zeta}\\
&=0.\end{align*}
Hence,
\begin{align*}\la\Phi(X),f\ra&=\lim_{i}\la\om_{\xi_i},(\id\ten f)W^*(U^*\ten 1)X(U\ten 1)W\ra\\
&=\lim_{i}\sum_j\la U^*x_jU(f\star y_j)\xi_i,\xi_i\ra\\
&=\lim_{i}\sum_j\la U^*x_jU\rho(f)y_j\xi_i,\xi_i\ra\\
&=\lim_{i}\la\Theta_{X'}(\rho(f))\xi_i,\xi_i\ra\\
&=\la M,\Theta_{X'}(\rho(f))\ra,\end{align*}
where $\Theta_{X'}$ is the element of $\mc{CB}^\sigma(\BLTQ)$ associated to $(\mathrm{Ad}(U^*)\ten\id)(X)\in\LIQ'\whten\LIQ$. Note that 
$$\LIQ\whten\LIQ\ni X\mapsto\Theta_{X'}\in\mc{CB}^\sigma(\BLTQ)$$
is a complete isometry. Thus, the restriction of $\Phi$ to $\LIQ\whten\LIQ$ defines a complete contraction into $M(\h{\G}')$: if $[X_{ij}]\in M_n(\LIQ\whten\LIQ)$, $[f_{kl}]\in M_m(\LOQ)$, then
\begin{align*}\norm{\la\Phi([X_{ij}]),[f_{kl}]\ra}_{M_n(M_m)}&=\norm{\la M,\Theta_{[X'_{ij}]}([\rho(f_{kl})])\ra}_{M_n(M_m)}\\
&\leq\norm{\Theta_{[X'_{ij}]}}_{cb}\norm{[\rho(f_{kl})]}_{M_m(C_0(\h{\G}'))}\\
&=\norm{[X_{ij}]}_{M_n(\LIQ\whten\LIQ)}\norm{[\rho(f_{kl})]}_{M_m(C_0(\h{\G}'))}.\end{align*}
Let $\tilde{\Phi}:\LIQ\whten\LIQ\rightarrow M(\h{\G}')$ denote the resulting map. Given $X\in\LIQ\whten\LIQ$, observe that the associated measure satisfies
$$\la\tilde{\Phi}(X),\rho(f)\ra:=\la\Phi(X),f\ra=\la\rho_*(\tilde{\Phi}(X)),f\ra, \ \ \ f\in\LOQ.$$
Thus, $\Phi(X)=\rho_*(\tilde{\Phi}(X))$, and so 
$$\la\tilde{\Phi}((m_h)^*(\hat{\mu}')),\rho(f)\ra=\la\Phi(\Gamma(\rho_*(\hat{\mu}'))),f\ra=\la\rho_*(\hat{\mu}'),f\ra=\la\hat{\mu}',\rho(f)\ra.$$
It follows that $\tilde{\Phi}\circ(m_h)^*=\id_{M(\h{\G}')}$ so that $(m_h)^*$ is a complete isometry. 

Finally, given $\Psi\in\Gam^{2,r}_{\LOQ}(\LOQ,\LIQ)\subseteq\mc{CB}_{\LOQ}(\LOQ,\LIQ)$, as $\LOQ$ is self-induced, we have $\Psi=\Gam(x)$ for some $x\in\LIQ$ with $\Gam(x)\in\LIQ\whten\LIQ$. But then $x=\Phi(\Gamma(x))=\rho_*(\hat{\mu}')$ for some $\hat{\mu}'\in M(\h{\G}')$, and $(m_h)^*(\hat{\mu}')=\Gamma(\rho_*(\hat{\mu}'))=\Psi$. Hence, $(m_h)^*$ is a weak*-weak* continuous completely isometric isomorphism.

\end{proof}

\begin{cor} Let $G$ be a QSIN locally compact group. Then
$$\Gam^{2,r}_{A(G)}(A(G),VN(G))\cong M(G)$$
completely isometrically and weak*-weak* homeomorphically. Thus,
$$A(G)\hten_{A(G)}A(G)\cong C_0(G)$$
completely isometrically.\end{cor}

\begin{cor}\label{p:SS-Delta}
Let $G$ be a QSIN locally compact group. Then the diagonal $ D = \{(s, s) : s \in G\}$ is a set of spectral synthesis for $A_h(G):=A(G)\hten A(G)$.
\end{cor}

\begin{proof}
Let
\[
I( D) := \{u\in A_h(G)\mid  u(s,s)=0 \quad \text{for all $s\in D$}\},
\]
and 
\[
J( D) := \{u\in  A_h(G) \cap C_c(G \times G) \mid \supp(u) \cap  D =\emptyset \}.
\]
Let  $x \in J( D)^\perp \subseteq VN(G) \whten VN(G)=(A_h(G))^*$. Since $A(G \times G)=A(G)\pten A(G)$ contractively injects into $A_h(G)$, we have that $x$ annihilates
\[
 \{u \in  A(G \times G) \cap C_c(G \times G) \mid \supp(u) \cap  D =\emptyset \}.
\]
It is known that every closed subgroup of $G \times G$, in particular $ D$, is a set of spectral synthesis for $A(G\times G)$. Therefore, $x$ annihilates 
\[
\{u\in A(G\times G)\mid  u(s,s)=0 \quad \text{for all $s\in D$}\}.
\]
By \cite[Theorem~3]{TT}, this annihilator is the von Neumann subalgebra of $VN(G \times G)$ generated by $\{ \lambda_s \otimes \lambda_s\mid s \in G\}$, i.e., $\Gam(VN(G))$. Therefore, $x  \in \Gamma(VN(G)) \cap VN(G) \whten VN(G)$, so by Theorem~\ref{t:Gilbert} there exists a unique measure $\mu$ in $M(G)$ such that $x= \Gamma(\rho_*(\mu))$. Then for any $u\in I( D)$ we have
$$\la x,u\ra=\la\Gam(\rho_*(\mu)),u\ra=\int_G u(s,s)d\mu(s)=0,$$
that is, $x\in I( D)^\perp$. This implies that $ D$ is a set of spectral synthesis for $A_h(G)$.
\end{proof}

\begin{remark} There is an error in \cite[Theorem 1.1]{LSS} which the authors have since addressed by adding an additional assumption \cite{LSS2}. Corollary \ref{p:SS-Delta} implies that $A(G)$ satisfies this additional assumption for any QSIN locally compact group $G$. Thus, the statement of \cite[Theorem 2.5]{LSS} remains valid.
\end{remark}

\begin{cor}\label{c:oa} Let $\G$ be a locally compact quantum group whose dual $\h{\G}$ is QSIN. Then $\LOQ$ is completely isomorphic to an operator algebra if and only if $\G$ is finite.\end{cor}

\begin{proof} If $\LOQ$ is completely isomorphic to an operator algebra, then by \cite[Theorem 2.2]{B}, the multiplication
$$m_h:\LOQ\hten\LOQ\rightarrow\LOQ$$
is completely bounded. Then $(m_h)^*:\LIQ\rightarrow\Gam^{2,r}_{\LOQ}(\LOQ,\LIQ)$ is completely bounded. Composing with the map $\tilde{\Phi}:\Gam^{2,r}_{\LOQ}(\LOQ,\LIQ)\rightarrow M(\h{\G}')$ from Theorem \ref{t:Gilbert}, we see that $\rho_*:M(\h{\G}')\rightarrow\LIQ$ is surjective: given $x\in\LIQ$, $x=\Phi(\Gamma(x))=\tilde{\Phi}((m_h)^*(x))=\rho_*(\hat{\mu}')$ for some $\hat{\mu}'$. By the open mapping theorem, $\LIQ$ is isomorphic to $M(\h{\G}')$ as a Banach space, and is therefore weakly sequentially complete. By \cite[Proposition 2]{S}, $\LIQ$ must be finite-dimensional, i.e., $\G$ is finite.
\end{proof}

\begin{remark}\label{r:LSS} It was shown in \cite[Proposition 3.1]{LSS} that for a locally compact group $G$, $A(G)$ is completely isomorphic to an operator algebra if and only if $G$ is finite. A key portion of the argument is a result of Forrest \cite[Theorem 3.2]{Forr}, which entails the discreteness of $G$ from the Arens regularity of $A(G)$. We recover their result for QSIN groups, without explicitly resorting to Arens regularity. 

If, in addition to the hypotheses of Corollary \ref{c:oa} we assume that $\LOQ$ is separable, then \cite[Theorem 3.10]{HNR} implies that $\G$ must be compact, in which case the recent result \cite[Corollary 1.3]{Youn} forces $\G$ to be finite.
\end{remark}

\subsection{Quantum Groups with Bounded Degree} Observe that Theorem \ref{t:Gilbert} only recovers the original Gilbert representation theorem for amenable groups: $VN(G)$ is a QSIN quantum group if and only if $G$ is amenable. As we now show, the original Gilbert factorization relies implicitly on the fact that $VN(G)$ is a quantum group with bounded degree.

We say that a locally compact quantum group $\G$ has \textit{bounded degree} if 
$$\mathrm{deg}(\G):=\sup\{\mathrm{dim}(U)\mid U\in\Irr\}<\infty,$$
where $\Irr$ is the set of (equivalence classes) of irreducible co-representations of $\G$. This notion has recently appeared in the literature in connection with Property (T) \cite{DSV}, Thoma type results \cite{BC} and unimodularity for discrete quantum groups \cite{KSol}.

\begin{examples}${}$\vskip5pt
\begin{enumerate}
\item[(1)] A co-commutative quantum group $\G_s$ trivially has bounded degree, as $\mathrm{dim}(U)=1$ for all $U\in\mathrm{Irr}(\G_s)$.
\item[(2)] A commutative quantum group $\G_a$ has bounded degree if and only if its underlying group $G$ contains a closed abelian subgroup of finite index \cite{Moore}.
\item[(3)] If $(G_1,G_2)$ is a matched pair of a countable discrete group $G_1$ and a finite group $G_2$, then by \cite[Theorem 3.4]{FMV} the bicrossed product $VN(G_1)^\beta\bowtie_\alpha\ell^\infty(G_2)$ has bounded degree.
\end{enumerate}
\end{examples}

Recall that a $C^*$-algebra $A$ is \textit{subhomogeneous} if every irreducible representation of $A$ has dimension $\leq n$ for some $n\in\N$. It is known that $A$ is subhomogeneous if and only if $A$ is $*$-isomorphic to a $C^*$-subalgebra of $M_n(C_0(X))$ for some $n\in\N$ and locally compact Hausdorff space $X$ \cite[Proposition IV.1.4.2]{Black}. Moreover, $A$ is subhomogeneous if and only if $A^{**}$ is subhomogeneous if and only if there is some $n\in\N$ such that $A^{**}$ is a direct sum of type $I_m$ von Neumann algebras for $m\leq n$ \cite[Proposition IV.1.4.6]{Black}. It follows from these facts that subhomogeneity passes to quotients as well as $C^*$-subalgebras, and any subhomogenous $C^*$-algebra is nuclear.

\begin{lem}\label{l:va} Let $\G$ be a locally compact quantum group with bounded degree. Then $C_0(\h{\G})$ and $\LIQH$ are subhomogeneous $C^*$-algebras and the unitary antipode $\hat{R}:\LIQH\rightarrow\LIQH$ is completely bounded.
\end{lem}

\begin{proof} Any irreducible representation $\pi:C_u(\h{\G})\rightarrow\BH$ arises from an irreducible unitary co-representation $U\in M(C_u(\G)\ten^{\vee}\mc{K}(H))$ of $\G$ via $\pi(\lm_u(f))=(\id\ten f\circ\Lambda_{\G})(U)$, $f\in\LOQs$, where $\lm_u:\LOQs\rightarrow C_u(\h{\G})$ is the canonical map \cite[Corollary 4.3]{K}. Since $\G$ has bounded degree, there exists some $M\in\N$ such that $\mathrm{dim}(\pi)\leq M$. Hence, $C_u(\h{\G})$ is subhomogeneous. It follows from the discussion above that $C_0(\h{\G})$ and $\LIQH$ are subhomogeneous. Thus, the adjoint operation and hence $\hat{R}$ is completely bounded on $\LIQH$ (see the proof of \cite[Proposition 5.1]{ATT}).
\end{proof}

\begin{lem}\label{l:co-amen} Let $\G$ be a locally compact quantum group such that $\h{\G}$ has bounded degree and trivial scaling group. Then $\G$ is co-amenable.
\end{lem}

\begin{proof} By Lemma \ref{l:va}, $C_0(\G)$ is a nuclear subhomogeneous $C^*$-algebra. Moreover, it admits a tracial state. This may be seen, for instance, by taking an inclusion $C_0(\G)\subseteq M_n(C_0(X))$ for some $n\in\N$ and locally compact Hausdorff space $X$. Then $\tau_n\ten\mu$ is a tracial state on $M_n(C_0(X))$ where $\mu$ is a (regular Borel) probability measure on $X$ and $\tau_n$ is the unique tracial state on $M_n$. Let $a\in C_0(\G)$ be positive and non zero. If $\tau_n\ten\mu(a)=0$, for every probability measure $\mu$, then by faithfulness of $\tau_n$, $(\id\ten\mu)(a)=0$ for every for every such $\mu$, whence $a=0$. Hence, there exists some $\mu$ for which $\tau_n\ten\mu(a)\neq 0$, and $\tau_n\ten\mu|_{C_0(\G)}$ is a non-zero trace on $C_0(\G)$.

Since $\h{\G}$ has trivial scaling group, $C_0(\G)$ is nuclear and admits a tracial state, $\G$ is co-amenable by \cite[Corollary 4.5]{C}.
\end{proof}

Our proof of the Gilbert factorization will rely on two manifestations of a commutation relation expressing quantum group duality. The first, at the level of multipliers, is \cite[Theorem 5.1]{JNR}:

\begin{thm}\label{t:comm} Let $\G$ be a locally compact quantum group. Then
\begin{equation*}\h{\Theta}^r(M_{cb}^r(\LOQH))=\Theta^r(\McbQr)^c\cap\mc{CB}^{\sigma}_{\LIQ}(\BLTQ).\end{equation*}
\end{thm}

The second, at the level of co-multiplications, is \cite[Proposition 6.3(1)]{KS}. Since our notation differs from \cite{KS}, and for convenience of the reader, we provide a proof (of an equivalent statement).

\begin{prop}\label{p:comm} Let $\G$ be a locally compact quantum group. Then 
$$\Gam(\LIQ)=((1\ten U^*)\h{\Gam}'(\LIQHP)(1\ten U))'\cap \LIQ\oten\LIQ.$$
\end{prop}

\begin{proof} Observe that 
$$(1\ten U^*)\h{V}'(1\ten U)=(1\ten U^*)\sigma V^*\sigma(1\ten U)=W^*.$$
For every $x\in\LIQ$ and $\hat{x}'\in\LIQHP$, we then have
\begin{align*}(1\ten U^*)\h{\Gam}'(\hat{x}')(1\ten U)\Gam(x)&=(1\ten U^*)\h{V}'(\h{x}'\ten 1)\h{V}'(1\ten U)W^*(1\ten x)W\\
&=(1\ten U^*)\h{V}'(1\ten U)(\hat{x}'\ten 1)(1\ten U^*)\h{V}'(1\ten U)W^*(1\ten x)W\\
&=W^*(\hat{x}'\ten 1)WW^*(1\ten x)W\\
&=W^*(1\ten x)WW^*(\hat{x}'\ten 1)W\\
&=\Gam(x)(1\ten U^*)\h{\Gam}'(\hat{x}')(1\ten U).\end{align*}
This gives the inclusion $\Gam(\LIQ)\subseteq((1\ten U^*)\h{\Gam}'(\LIQHP)(1\ten U))'\cap \LIQ\oten\LIQ$. Conversely, if $X\in ((1\ten U^*)\h{\Gam}'(\LIQHP)(1\ten U))'\cap \LIQ\oten\LIQ$, then as above
$$XW^*(\hat{x}'\ten 1)W=W^*(\hat{x}'\ten 1)WX, \ \ \ \hat{x}'\ten\LIQHP,$$
which implies that 
$$WXW^*=\LIQH\oten\BLTQ\cap\LIQ\oten\BLTQ=\C\oten\BLTQ.$$
Let $T\in\BLTQ$ such that $WXW^*=1\ten T$. Then $X=W^*(1\ten T)W=\Gam^l(T)$. By co-associativity it follows that $X\in\LIQ\oten\LIQ$ with $(\Gam\ten\id)(X)=(\id\ten\Gam)(X)$. Hence, since $\LOQ$ is self-induced, $X\in\Gam(\LIQ)$. 
\end{proof}

Suppose $\G$ is a Kac algebra such that $\h{\G}$ has bounded degree. Then Lemma \ref{l:va} ensures that $\LIQ$ is subhomogeneous and $R:\LIQ\rightarrow\LIQ$ is completely bounded. It follows that $\norm{R}_{cb}=\norm{\ad}_{cb}\leq\mathrm{deg}(\h{\G})$, where $\ad:\LIQ\rightarrow\LIQ$ is the adjoint operation, viewed as an antilinear map.
In addition, the flip map $\Sigma:\LIQ\whten\LIQ\rightarrow\LIQ\whten\LIQ$ is bounded with $\norm{\Sigma}\leq\mathrm{deg}(\h{\G})^2$ by (the proof of) \cite[Proposition 5.1]{ATT}.

Below we work with completely bounded left multipliers of $\h{\G}'$. Since $\widehat{\h{\G}'}=\G^{op}$, the opposite quantum group of $\G$, with underlying von Neumann algebra $\LIQ$ and co-multiplication $\Sigma\Gam$ \cite[Proposition 4.2]{KV2}, $M_{cb}^l(L^1(\h{\G}'))\subseteq\LIQ$ is a subalgebra. Also by \cite[\S4]{KV2}, the left fundamental unitary of $\h{\G'}$ is
$$\widehat{W}'=\widehat{W^{op}}=\sigma(W^{op})^*\sigma=\sigma(\sigma V^*\sigma)^*\sigma=V,$$ 
the right fundamental unitary of $\G$. Thus, $\hat{\lm}'(\hat{f}')=\rho_*(\hat{f}')$, for all $\hat{f}'\in\LOQHP$. By (\ref{e:WV}) it follows that $\hat{\lm}'=\hat{\lm}\circ\widetilde{R}_*$, where $\widetilde{R}:\BLTQ\ni T\mapsto \h{J}T^*\h{J}\in\BLTQ$ is the canonical extension of the unitary antipode $R$. This entails that $b\in\LIQ$ defines an element of $M_{cb}^l(L^1(\h{\G}'))$ if and only if it defines an element of $\McbQHl$, with
$$\hat{m}^{'l}_b=\widetilde{R}_*\circ \hat{m}^l_b\circ\widetilde{R}_*, \ \ \ \hat{m}^l_b=\widetilde{R}_*\circ \hat{m}^{'l}_b\circ\widetilde{R}_*.$$

\begin{thm}\label{t:va} Let $\G$ be Kac algebra such that $\h{\G}$ has bounded degree. Then 
$$\Gam^{2,r}_{\LOQ}(\LOQ,\LIQ)\cong M_{cb}^l(L^1(\h{\G}'))$$
completely isomorphically and weak*-weak* homeomorphically, with
$$\mathrm{deg}(\h{\G})^{-1}\norm{[b_{ij}]}_{M_n(M_{cb}^l(L^1(\h{\G}')))}\leq\gamma^{2,r}([b_{ij}])\leq \mathrm{deg}(\h{\G})\norm{[b_{ij}]}_{M_n(M_{cb}^l(L^1(\h{\G}')))},$$
for all $[b_{ij}]\in M_n(M_{cb}^l(L^1(\h{\G}'))$.
\end{thm}

\begin{proof}  Let $[b_{ij}]\in M_n(M_{cb}^l(L^1(\h{\G}')))$. Then by \cite[Proposition 6.1]{D}, for any $\xi,\eta,\alpha,\beta\in\LTQ$ we have
\begin{equation}\label{e:Daws}\la\h{\Theta'}^l(b_{ij})(\xi\eta^*)\alpha,\beta\ra=\la(\id\ten R)\Sigma\Gam(b_{ij}),\om_{\alpha,\eta}\ten\om_{\xi,\beta}\ra.\end{equation}
Under the canonical completely isometric identification
\begin{equation}\label{e:iden}\mc{K}(\LTQ)\pten\TCQ\ni\xi\eta^*\ten\om_{\alpha,\beta}\mapsto\om_{\xi,\beta}\ten\om_{\alpha,\eta}\in\TCQ\hten\TCQ,\end{equation}
it follows that $\Sigma\circ(\id\ten R)\Sigma\Gam(b_{ij})=(R\ten\id)\Gam(b_{ij})\in\LIQ\whten\LIQ$ corresponds to $\h{\Theta'}^l(b_{ij})\in\mc{CB}^\sigma_{\LIQ'}(\BLTQ)$. Then
\begin{align*}\norm{[\Gam(b_{ij})]}_{M_n(\LIQ\whten\LIQ)}&=\norm{[(R\ten \id)\circ(R\ten\id)\Gam(b_{ij})]}_{M_n(\LIQ\whten\LIQ)}\\
&\leq \mathrm{deg}(\h{\G})\norm{[(R\ten\id)\Gam(b_{ij})]}_{M_n(\LIQ\whten\LIQ)}\\
&=\mathrm{deg}(\h{\G})\norm{[\h{\Theta'}^l(b_{ij})]}_{M_n(\LIQ\whten\LIQ)}\\
&=\mathrm{deg}(\h{\G})\norm{[b_{ij}]}_{M_n(M_{cb}^l(L^1(\h{\G}')))}.\end{align*}
Hence, $\Gam:M_{cb}^l(L^1(\h{\G}'))\rightarrow\Gam^{2,r}_{\LOQ}(\LOQ,\LIQ)$ is completely bounded. 

Given $\Phi\in\Gam^{2,r}_{\LOQ}(\LOQ,\LIQ)$, since $\LOQ$ is self-induced, we know that $\Phi=\Gam(x)\in\LIQ\whten\LIQ$ for some $x\in\LIQ$. It follows that $\Sigma\circ(\id\ten R)\Gam(x)\in\LIQ\whten\LIQ$. Let $\Phi'$ denote the corresponding map in $\mc{CB}^{\sigma}(\BLTQ)$. We now show that $R\circ\Phi'\circ R$ commutes with $\Ad(U^*)\circ\Theta^r(\rho(f))\circ \Ad(U)$ for all $f\in\LOQ$, which will allow us to appeal to the commutation relations above. 

Suppose $f=\om_\zeta|_{\LIQ}$. Then
\begin{align*}\Ad(U^*)\circ\Theta^r(\rho(f))\circ \Ad(U)(T)&=(\id\ten f)((U^*\ten 1)V(UTU^*\ten 1)V^*(U\ten 1))\\
&=(\id\ten f)(\sigma W\sigma(T\ten 1)\sigma W\sigma)\\
&=(\om_{\zeta}\ten\id)(W(1\ten T)W^*)\\
&=\sum_iW_iTW_i^*,
\end{align*}
where $W_i=(\om_{e_i,\zeta}\ten\id)(W)$, and $(e_i)$ is an orthonormal basis for $\LTQ$. Since $\G$ is a Kac algebra, we have $\sum_i W_iW_i^*,\sum_i W_i^*W_i<\infty$ (see Lemma \ref{l:swap} below). Using the identification (\ref{e:iden}), for every $\al,\beta,\xi,\eta\in\LTQ$ we have
\begin{align*}\la R\circ\Phi'\circ R\bigg(\sum_iW_i\xi\eta^*W_i^*\bigg)\al,\be\ra&=\sum_i\la R\circ\Phi'\circ R(W_i\xi\eta^*W_i^*)\al,\be\ra\\
&=\sum_i\la \Phi'(\h{J}W_i\eta\xi^*W_i^*\h{J})\h{J}\be,\h{J}\al\ra\\
&=\sum_i\la(\id\ten R)\Gam(x),\om_{\h{J}\be,\h{J}W_i\xi}\ten\om_{\h{J}W_i\eta,\h{J}\al}\ra\\
&=\sum_i\la\Gam(x),\om_{\h{J}\be,\h{J}W_i\xi}\ten\om_{\al,W_i\eta}\ra\\
&=\sum_i\la\Gam(x),\om_{\h{J}\be,\h{J}\xi}\cdot R(W_i)\ten\om_{\al,\eta}\cdot W_i^*\ra\\
&=\sum_i\la(R(W_i)\ten W_i^*)\Gam(x),\om_{\h{J}\be,\h{J}\xi}\ten\om_{\al,\eta}\ra.
\end{align*}
Now, $\sum_i R(W_i)\ten W_i^*\in\LIQHP\whten\LIQH$ and satisfies
\begin{align*}\sum_i \la R(W_i)\ten W_i^*,\om_{\xi,\beta}\ten \om_{\al,\eta}\ra&=\sum_i\la W_i\ten W_i^*,\om_{\h{J}\be,\h{J}\xi}\ten\om_{\al,\eta}\ra\\
&=\la(\om_{\zeta}\ten\id)(W(1\ten \h{J}\beta\eta^*)W^*)\al,\h{J}\xi\ra\\
&=\la(1\ten U^*)W(1\ten U)(1\ten J\beta\eta^*U)(1\ten U^*)W^*(1\ten U)(\zeta\ten U^*\al),\zeta\ten J\xi\ra\\
&=\la\sigma V\sigma(1\ten J\beta\eta^*U)\sigma V^*\sigma (\zeta\ten U^*\al),\zeta\ten J\xi\ra\\
&=\la V(J\beta\eta^*U\ten 1)V^*(U^*\al\ten \zeta),J\xi\ten\zeta\ra\\
&=\la \Theta^r(f)(J\beta\eta^*U)U^*\al,J\xi\ra\\
&=\la\h{\Gam}'(\rho(f)),\om_{U^*\al,U^*\eta}\circ\h{R}\ten\om_{J\be,J\xi}\ra\\
&=\la\h{\Gam}'(\rho(f)),\om_{U^*\al,U^*\eta}\circ\h{R}\ten\om_{\xi,\be}\circ\h{R}\ra\\
&=\la \Sigma \circ (\h{R}\ten\h{R})\circ \h{\Gam}'(\rho(f)),\om_{\xi,\beta}\ten U^*\cdot \om_{\al,\eta}\cdot U\ra\\
&=\la (1\ten U)\h{\Gam}'(\rho(f\circ R))(1\ten U^*),\om_{\xi,\beta}\ten \om_{\al,\eta}\ra.\\
\end{align*}
The seventh equality above follows from the right multiplier version of (\ref{e:Daws}), which can be deduced from the proof of \cite[Theorem 4.7]{HNR2}. It follows that $(1\ten U)\h{\Gam}'(\rho(f\circ R))(1\ten U^*)=\sum_i R(W_i)\ten W_i^*$ weak* in $\LIQ\oten\LIQ$. Noting that $U=U^*$, by Propopsition \ref{p:comm} we have
\begin{align*}\la R\circ\Phi'\circ R\bigg(\sum_iW_i\xi\eta^*W_i^*\bigg)\al,\be\ra&=\sum_i\la(R(W_i)\ten W_i^*)\Gam(x),\om_{\h{J}\be,\h{J}\xi}\ten\om_{\al,\eta}\ra\\
&=\la (1\ten U^*)\h{\Gam}'(\rho(f\circ R))(1\ten U)\Gam(x),\om_{\h{J}\be,\h{J}\xi}\ten\om_{\al,\eta}\ra\\
&=\la \Gam(x)(1\ten U^*)\h{\Gam}'(\rho(f\circ R))(1\ten U),\om_{\h{J}\be,\h{J}\xi}\ten\om_{\al,\eta}\ra\\
&=\sum_i\la \Gam(x)(R(W_i)\ten W_i^*),\om_{\h{J}\be,\h{J}\xi}\ten\om_{\al,\eta}\ra\\
&=\sum_i\la \Gam(x),R(W_i)\cdot\om_{\h{J}\be,\h{J}\xi}\ten W_i^*\cdot \om_{\al,\eta}\ra\\
&=\sum_i\la \Gam(x),\om_{\xi,W_i^*\be}\circ R\ten\om_{W_i^*\al,\eta}\ra\\
&=\sum_i\la \Sigma(\id \ten R)\Gam(x),\om_{\h{J}\eta,\h{J}W_i^*\al}\ten \om_{\h{J}W_i^*\be,\h{J}\xi}\ra\\
&=\sum_i\la \Phi'(\h{J}\eta\xi^*\h{J}),\h{J}W_i^*\be,\h{J}W_i^*\al\ra\\
&=\sum_i\la R\circ\Phi'\circ R(\xi\eta^*)W_i^*\al,W_i^*\be\ra\\
&=\sum_i\la W_i(R\circ\Phi'\circ R(\xi\eta^*))W_i^*\al,\be\ra.\end{align*}
Hence, 
$$[R\circ\Phi'\circ R,\Ad(U^*)\circ\Theta^r(\rho(f))\circ \Ad(U)]=0$$
for $f=\om_\zeta|_{\LIQ}$, $\zeta\in\LTQ$. By polarization, and the fact that $\LIQ$ is standardly represented on $\LTQ$, we have 
$$[R\circ\Phi'\circ R,\Ad(U^*)\circ\Theta^r(\rho(f))\circ \Ad(U)]=0$$
for all $f\in\LOQ$. By Lemma \ref{l:co-amen}, $\G$ is co-amenable, so that $\LOQ$ is weak* dense in $M_{cb}^r(\LOQ)=M(\G)$ \cite[Theorem 4.2]{HNR2}. By weak* continuity of $\Theta^r$ \cite[Theorem 6.1]{D}, it follows that 
$$[R\circ\Phi'\circ R,\Ad(U^*)\circ\Theta^r(\hat{b}')\circ \Ad(U)]=0$$
for all $\hat{b}'\in M_{cb}^r(\LOQ)$. By Theorem \ref{t:comm} we have
$$\h{R}\circ\Phi'\circ\h{R}=\Ad(U)\circ R\circ\Phi'\circ R\circ\Ad(U^*)\in \h{\Theta}^r(M_{cb}^r(\LOQH)).$$
Let $b'\in  M_{cb}^r(\LOQH)$ satisfy $\h{\Theta}^r(b')=\h{R}\circ\Phi'\circ\h{R}$. Then by \cite[Theorem 4.9]{HNR2} $\Phi'=\h{\Theta}^l(\h{R}(b'))$, in which case the analogue of (\ref{e:Daws}) implies $x=\h{R}(b')\in M_{cb}^l(\LOQH)=M_{cb}^l(\LOQHP)$, and, at last, $\Gam:M_{cb}^l(\LOQHP)\rightarrow\Gam^{2,r}_{\LOQ}(\LOQ,\LIQ)$ is surjective.

To conclude, if $[\Phi_{ij}]\in M_n(\Gam^{2,r}_{\LOQ}(\LOQ,\LIQ))$, then each $\Phi_{ij}=\Gam(x_{ij})\in\LIQ\whten\LIQ$. It follows from above that each $x_{ij}\in M_{cb}^l(\LOQHP)$. Moreover,
\begin{align*}\norm{[x_{ij}]}_{M_n(M_{cb}^l(\LOQHP))}&=\norm{[(R\ten\id)\Gam(x_{ij})]}_{M_n(\LIQ\whten\LIQ)}\\
&\leq \mathrm{deg}(\h{\G})\norm{[\Gam(x_{ij})]}_{M_n(\LIQ\whten\LIQ)}\\
&=\mathrm{deg}(\h{\G})\norm{[\Phi_{ij}]}_{M_n(\Gam^{2,r}_{\LOQ}(\LOQ,\LIQ))}.\end{align*}
If $(b_i)$ is a bounded net in $M_{cb}^l(\LOQHP))$ converging weak* to $b$, then by \cite[Theorem 6.1]{D} $\h{\Theta'}^l(b_i)\rightarrow\h{\Theta'}^l(b)$ weak*, that is, $(R\ten\id)\Gam(b_i)\rightarrow(R\ten\id)\Gam(b)$ weak* in $\LIQ\whten\LIQ$. By weak* continuity of $(R\ten\id)$, we have $\Gam(b_i)\rightarrow\Gam(b)$ weak* in $\LIQ\whten\LIQ$. Thus, $\Gam:M_{cb}^l(\LOQHP))\rightarrow\Gam^{2,r}_{\LOQ}(\LOQ,\LIQ)$ is a weak*-weak* homeomorphic completely bounded isomorphism.
\end{proof}

\begin{remark} If $G$ is a locally compact group, then the dual of $\G_a=\LI$ has bounded degree with $\mathrm{deg}(\widehat{\G}_a)=\mathrm{deg}(\G_s)=1$, so Theorem \ref{t:va} implies that 
$$\Gam^{2,r}_{\LO}(\LO,\LI)\cong M_{cb}A(G)$$
completely \textit{isometrically}. We therefore obtain a new proof of Gilbert's original representation theorem, revealing the link between its completely isometric nature and the bounded degree of $\G_s$.
\end{remark}

%The following generalizes the classical result that $L^1(G)$ is isomorphic to an operator algebra if and only if $G$ is finite \cite{Young}.

%\begin{cor}\label{c:oa2} Let $\G$ be Kac algebra such that $\h{\G}$ has bounded degree. Then $\LOQ$ is completely isomorphic to an operator algebra if and only if $\G$ is finite.\end{cor}

%\begin{proof} If $\LOQ$ is completely isomorphic to an operator algebra, then by \cite[Theorem 2.2]{B}, the multiplication
%$$m_h:\LOQ\hten\LOQ\rightarrow\LOQ$$
%is completely bounded. Then $(m_h)^*:\LIQ\rightarrow\Gam^{2,r}_{\LOQ}(\LOQ,\LIQ)$ is completely bounded. By Theorem \ref{t:va} $\Gam:M_{cb}^l(\LOQHP)\cong\Gam^{2,r}_{\LOQ}(\LOQ,\LIQ)$ and since $(m_h)^*|_{M_{cb}^l(\LOQHP)}=\LIQ$, it follows that $M_{cb}^l(\LOQHP)=\LIQ$. But then $M(C_0(\G))=\LIQ$ which by \cite[Theorem 3.7]{HNR} implies that $\G$ is discrete. Then 

%\end{proof}

\section{Completely Integral and Completely 1-Summing Multipliers}\label{s:ci}

Let $G$ be a locally compact group. Since $\LO$ is a max operator space, it follows from the classical Grothendieck inequality that a bounded linear map $\LO\rightarrow\LI$ is integral if and only if it factors through a Hilbert space (see \cite[(3.11)]{Pi}, together with \cite[(1.47),(A.7)]{BLM}). Using this fact along with work work of Gilbert \cite{Gil}, it was shown in \cite[Proposition 5.1]{R} that the integral $\LO$-module maps from $\LO$ to $\LI$ are precisely the completely bounded multipliers of the Fourier algebra. Specifically,
$$I_{\LO}(\LO,\LI)\cong \Mcb$$
isomorphically, with $\norm{u}_{cb}\leq I(u)\leq K\norm{u}_{cb}$, where $K$ is Grothendieck's constant. We now establish this equivalence for quantum groups whose dual is either QSIN or has bounded degree. Our techniques allow us to simultaneously characterize the completely 1-summing module maps $\LOQ\rightarrow\LIQ$.
 
\begin{lem}\label{l:swap} Let $\G$ be a locally compact quantum group whose dual $\h{\G}$ has trivial scaling group. For every $\hat{\mu}'\in M(\h{\G}')$, there exist families $(a_i),(b_i)$ in $\LIQ$ such that
$$\Gamma(\rho_*(\hat{\mu}'))=\sum_i a_i\ten b_i, \ \ \ \sum_i a_ia_i^*,\sum_i a_i^*a_i,\sum_i b_ib_i^*\sum_i b_i^*b_i<\infty.$$
In particular, $\Gamma(\rho_*(\hat{\mu}')),\Sigma\Gamma(\rho_*(\hat{\mu}'))\in\LIQ\whten\LIQ$.\end{lem}

\begin{proof} First, recall from the proof of Theorem \ref{t:Gilbert} that $\Gamma(\rho_*(\hat{\mu}'))=(m_h)^*(\hat{\mu}')\in\LIQ\whten\LIQ$. Next, observe that 
$$\Gamma(\rho_*(\hat{\mu}'))=\Gam((\hat{\mu}'\ten\id)(V))=(\hat{\mu}'\ten\id\ten\id)(\id\ten\Gam)(V)=(\hat{\mu}'\ten\id\ten\id)(V_{12}V_{13}).$$
From this, one can easily deduce that the normal $\LIQ'$-bimodule map $\Theta'(\hat{\mu}')$ on $\BLTQ$ corresponding to $\Gamma(\rho_*(\hat{\mu}'))$ is given by
$$\Theta'(\hat{\mu}')(T)=(\id\ten\hat{\mu}')(\sigma V\sigma)(T\ten 1)(\sigma V\sigma), \ \ \ T\in\BLTQ.$$
Write $\hat{\mu}'=\om_{\xi,\eta}\circ\pi$ for some representation $\pi:C_0(\h{\G}')\rightarrow\BH$, and $\xi,\eta\in H$. Then, with $U_{\hat{\mu}'}:=(\id\ten\pi)(\sigma V\sigma)$, we may resolve the identity $1\in\BH$ via an orthonormal basis $(e_i)$ to obtain the following Kraus decomposition
$$\Theta'(\hat{\mu}')(T)=\sum_i a_i T b_i, \ \ \ T\in\BLTQ,$$
where $a_i=(\id\ten\om_{e_i,\eta})(U_{\hat{\mu}'})$ and $b_i=(\id\ten\om_{\xi,e_i})(U_{\hat{\mu}'})$ belong to $\LIQ$ and satisfy $\sum_i a_ia_i^*,\sum_i b_i^*b_i<\infty$. 

Since $\h{\G}$ has trivial scaling group, $M(\h{\G}')$ is an involutive Banach algebra, where $\hat{\nu}'^o=\hat{\nu}'^*\circ \hat{R}$. In particular, 
$$(\hat{\nu}'\circ\hat{R}\ten\id)(V)=(\hat{\nu}'\ten\id)(V^*), \ \ \ \hat{\nu}'\in M(\h{\G}').$$
Fix a conjugate linear involution $J_{\h{\mu}'}$ on $H$, and define $R_{\hat{\mu}'}:\BH\ni A\mapsto J_{\hat{\mu}'} A^* J_{\hat{\mu}'}\in\BH$. It follows that $\hat{\mu}'^o=\om_{\eta,\xi}\circ\pi\circ\hat{R}=\om_{J_{\hat{\mu}'}\xi,J_{\hat{\mu}'}\eta}\circ R_{\hat{\mu}'}\circ \pi\circ\hat{R}$. As above, $\Gamma(\rho_*(\hat{\mu}'^o))\in\LIQ\whten\LIQ$ is represented by the map
$$\Theta'(\hat{\mu}'^o)(T)=(\id\ten\hat{\mu}'^o)(\sigma V\sigma)(T\ten 1)(\sigma V\sigma)=(\id\ten \om_{J_{\hat{\mu}'}\xi,J_{\hat{\mu}'}\eta})(U_{\hat{\mu}'^o}(T\ten 1)(U_{\hat{\mu}'^o}), \ \ \ T\in\BLTQ,$$
where $U_{\hat{\mu}'^o}=(\id\ten R_{\hat{\mu}'}\circ \pi\circ\hat{R})(\sigma V\sigma)$ is unitary. Resolving the identity $1\in\BH$ with respect to the orthonormal basis $(J_{\hat{\mu}'} e_i)$, it follows that 
$$\Theta'(\hat{\mu}'^o)(T)=\sum_i x_i T y_i, \ \ \ T\in\BLTQ,$$
where $x_i=(\id\ten\om_{J_{\hat{\mu}'} e_i,J_{\hat{\mu}'}\eta})(U_{\hat{\mu}'^o})$ and $y_i=(\id\ten\om_{J_{\hat{\mu}'}\xi,J_{\hat{\mu}'} e_i})(U_0)$ satisfy $\sum_i x_ix_i^*,\sum_i y_i^*y_i<\infty$. But
$$x_i =(\id\ten\om_{J_{\hat{\mu}'} e_i,J_{\hat{\mu}'}\eta})(U_{\hat{\mu}'^o})=(\id\ten\om_{\eta,e_i}\circ\pi\circ\hat{R})(\sigma V\sigma)=(\om_{\eta,e_i}\circ\pi\ten\id)(V^*)=a_i^*.$$
Similarly, $y_i=b_i^*$. Thus, $\sum_ia_i^*a_i,\sum_i b_ib_i^*<\infty$. In particular, 
$$\sum_i b_i\ten a_i=\Sigma\Gamma(\rho_*(\hat{\mu}'))\in\LIQ\whten\LIQ.$$
\end{proof}

\begin{thm}\label{t:ci} Let $\G$ be a locally compact quantum group for which $\h{\G}$ is QSIN and has trivial scaling group. Then
$$\mc{I}_{\LOQ}(\LOQ,\LIQ)\cong M(\h{\G}')$$
isomorphically and weak*-weak* homeomorphically, with
$$\norm{\hat{\mu}'}\leq \iota(\hat{\mu}')\leq 2\norm{\hat{\mu}'}, \ \ \ \hat{\mu}'\in M(\h{\G}').$$
\end{thm}

\begin{proof} Given $\hat{\mu}'\in M(\h{\G}')$, by Lemma \ref{l:swap} there exist families $(a_i),(b_i)$ in $\LIQ$ such that
$$\Gamma(\rho_*(\hat{\mu}'))=\sum_i a_i\ten b_i, \ \ \ \sum_i a_ia_i^*,\sum_i a_i^*a_i,\sum_i b_ib_i^*\sum_i b_i^*b_i<\infty.$$
More precisely, taking a representation $\hat{\mu}'=\om_{\xi,\eta}\circ\pi$ for some $\pi:C_0(\h{\G}')\rightarrow\BH$, and $\xi,\eta\in H$, we can take $a_i=(\id\ten\om_{e_i,\eta})(U_{\hat{\mu}'})$ and $b_i=(\id\ten\om_{\xi,e_i})(U_{\hat{\mu}'})$, where $U_{\hat{\mu}'}=(\id\ten\pi)(\sigma V\sigma)$. Hence, for every $\beta\in\LTQ$,
\begin{align*}\sum_i\la a_ia_i^*\beta,\beta\ra&=\sum_i\la(\id\ten\om_{e_i,\eta})(U_{\hat{\mu}'})^*\beta,(\id\ten\om_{e_i,\eta})(U_{\hat{\mu}'})^*\beta\ra\\
&=\sum_i\la U_{\hat{\mu}'}(1\ten e_ie_i^*)U_{\hat{\mu}'}^*\beta\ten\eta,\beta\ten\eta\ra=\norm{\beta}^2\norm{\eta}^2.\end{align*}
Similarly, adopting the same notion as above,
\begin{align*}\sum_i\la a_i^*a_i\beta,\beta\ra&=\sum_i\la(\id\ten\om_{J_{\hat{\mu}'} e_i,J_{\hat{\mu}'}\eta})(U_{\hat{\mu}'^o})^*\beta,(\id\ten\om_{J_{\hat{\mu}'} e_i,J_{\hat{\mu}'}\eta})(U_{\hat{\mu}'^o})^*\beta\ra\\
&=\sum_i\la U_{\hat{\mu}'^o}(1\ten J_{\hat{\mu}'} e_ie_i^*J_{\hat{\mu}'})U_{\hat{\mu}'^o}^*\beta\ten J_{\hat{\mu}'}\eta,\beta\ten J_{\hat{\mu}'}\eta\ra=\norm{\beta}^2\norm{\eta}^2.\end{align*}
It follows that $\norm{\sum_i a_ia_i^*}=\norm{\sum_i a_i^*a_i^*}=\norm{\eta}^2$. Analogously, one deduces that $\norm{\sum_i b_i^*b_i}=\norm{\sum_i b_ib_i^*}=\norm{\xi}^2$.

Now, let $f$ be a finite-rank tensor in $\LOQ\ten^{\vee}\LOQ\hookrightarrow(\LIQ\pten\LIQ)^*$. By the non-commutative Grothendieck inequality \cite{HM,PS}, there exist states $\vphi_1,\vphi_2,\psi_1,\psi_2\in\LIQ^*$ such that
$$\la f,a\ten b\ra\leq\norm{f}_{\LOQ\ten^{\vee}\LOQ}(\vphi_1(aa^*)^{1/2}\psi_1(b^*b)^{1/2}+\vphi_2(a^*a)^{1/2}\psi_2(bb^*)^{1/2})$$
for all $a,b\in\LIQ$. Note that for any state $\vphi\in\LIQ^*$, the net of finite sums $\sum_{i\in F}\vphi(a_ia_i^*)$ is increasing and bounded by $\vphi(\sum_ia_ia_i^*)\leq\norm{\sum_ia_ia_i^*}$. Thus, $\sum_i \vphi(a_ia_i^*)\leq \norm{\sum_ia_ia_i^*}$. Putting things together, we have
\begin{align*}&|\la\Gamma(\rho_*(\hat{\mu}')),f\ra|\leq\sum_i|\la f,a_i\ten b_i\ra|\\
&\leq \sum_i\norm{f}_{\LOQ\ten^{\vee}\LOQ}(\vphi_1(a_ia_i^*)^{1/2}\psi_1(b_i^*b_i)^{1/2}+\vphi_2(a_i^*a_i)^{1/2}\psi_2(b_ib_i^*)^{1/2})\\
&\leq\norm{f}_{\LOQ\ten^{\vee}\LOQ}\bigg(\bigg(\sum_i\vphi_1(a_ia_i^*)\bigg)^{1/2}\bigg(\sum_i\psi_1(b_i^*b_i)\bigg)^{1/2}+\bigg(\sum_i\vphi_2(a_i^*a_i)\bigg)^{1/2}\bigg(\sum_i\psi_2(b_ib_i^*)\bigg)^{1/2}\bigg)\\
&\leq\norm{f}_{\LOQ\ten^{\vee}\LOQ}\bigg(\norm{\sum_ia_ia_i^*}^{1/2}\norm{\sum_ib_i^*b_i}^{1/2}+\norm{\sum_ia_i^*a_i}^{1/2}\norm{\sum_ib_ib_i^*}^{1/2}\bigg)\\
&=2\norm{f}_{\LOQ\ten^{\vee}\LOQ}\norm{\xi}\norm{\eta}.\end{align*}
Since the representation $\hat{\mu}'=\om_{\xi,\eta}\circ\pi$ was arbitrary, it follows that
$$|\la\Gamma(\rho_*(\hat{\mu}')),f\ra|\leq 2\norm{f}_{\LOQ\ten^{\vee}\LOQ}\inf\{\norm{\xi}\norm{\eta}\mid\hat{\mu}'=\om_{\xi,\eta}\circ\pi\}=2\norm{f}_{\LOQ\ten^{\vee}\LOQ}\norm{\hat{\mu}'}.$$
Since $\LOQ$ is the predual of a von Neumann algebra it is locally reflexive \cite{EJR}, so 
$$\mc{I}(\LOQ,\LIQ)\cong(\LOQ\ten^{\vee}\LOQ)^*,$$
completely isometrically by \cite[Theorem 14.3.1]{ER}. It follows that 
$$\Gam\circ\rho_*:M(\h{\G}')\rightarrow \mc{I}_{\LOQ}(\LOQ,\LIQ)$$ 
is bounded by 2. Conversely, given $\Phi\in\mc{I}_{\LOQ}(\LOQ,\LIQ)$, we have $\Phi\in\Gam^{2,r}_{\LOQ}(\LOQ,\LIQ)$ so that $\Phi=\Gamma(\rho_*(\hat{\mu}'))$ for some measure $\hat{\mu}'\in M(\h{\G}')$, with $\norm{\hat{\mu}'}=\gamma^{2,r}(\Gamma(\rho_*(\hat{\mu}')))\leq\iota(\Phi)$, by Theorem \ref{t:Gilbert}. 

Finally, since $\Gam\circ\rho_*:M(\h{\G}')\rightarrow\Gam^{2,r}_{\LOQ}(\LOQ,\LIQ)$ is weak*-weak* continuous, if $(\hat{\mu}'_i)$ is a bounded net converging weak* to zero, then for any finite-rank tensor $f\in\LOQ\ten^{\vee}\LOQ$ we have $\la\Gam(\rho_*(\hat{\mu}'_i)),f\ra\rightarrow 0$. Since $(\Gam(\rho_*(\hat{\mu}'_i)))$ is bounded in the completely integral norm, it follows that $\la\Gam(\rho_*(\hat{\mu}'_i)),f\ra\rightarrow 0$ for any $f\in\LOQ\ten^{\vee}\LOQ$.
\end{proof}

Combining Theorems \ref{t:Gilbert} and \ref{t:ci}, together with the relations (\ref{e:chain}) we obtain:

\begin{cor} Let $\G$ be a locally compact quantum group for which $\h{\G}$ is QSIN and has trivial scaling group. Then
$$\LOQ\ten^{\vee}_{\LOQ}\LOQ\cong\LOQ\ten^{\vee/}_{\LOQ}\LOQ\cong\LOQ\hten_{\LOQ}\LOQ$$
isomorphically. In particular,
$$\mc{I}_{\LOQ}(\LOQ,\LIQ)=\Pi^1_{\LOQ}(\LOQ,\LIQ)=\Gam^{2,r}_{\LOQ}(\LOQ,\LIQ)$$
set-theoretically, with
$$\gamma^{2,r}(\Phi)\leq\pi^1(\Phi)\leq\iota(\Phi)\leq 2\gamma^{2,r}(\Phi), \ \ \ \Phi\in\mc{CB}_{\LOQ}(\LOQ,\LIQ).$$
\end{cor}

Similar techniques lead to a version of Theorem \ref{t:ci} for quantum groups with bounded degree. 

\begin{thm} Let $\G$ be a Kac algebra such that $\h{\G}$ has bounded degree. Then
$$\mc{I}_{\LOQ}(\LOQ,\LIQ)\cong M_{cb}^l(L^1(\h{\G}'))$$
isomorphically and weak*-weak* homeomorphically, with 
$$\mathrm{deg}(\h{\G})^{-1}\norm{b}_{M_{cb}^l(L^1(\h{\G}'))}\leq \iota(b)\leq \mathrm{deg}(\h{\G})(1+\mathrm{deg}(\h{\G})^2)\norm{b}_{M_{cb}^l(L^1(\h{\G}'))}, \ \ \ b\in M_{cb}^l(L^1(\h{\G}')).$$  
\end{thm}

\begin{proof} Let $b\in M_{cb}^l(L^1(\h{\G}'))$. As in the proof of Theorem \ref{t:va}, $(R\ten\id)\Gam(b)\in\LIQ\whten\LIQ$ corresponds to $\h{\Theta'}^l(b)\in\mc{CB}^\sigma_{\LIQ'}(\BLTQ)$. By complete boundedness of     $(R\ten\id)$, we have $\Gam(b)\in\LIQ\whten\LIQ$. Pick a norm attaining representation 
$$\Gam(b)=\sum_ia_i\ten b_i, \ \ \ \norm{\Gam(b)}_{\LIQ\whten\LIQ}=\norm{\sum_ia_ia_i^*}^{1/2}\norm{\sum_ib_i^*b_i}^{1/2}.$$
By subhomogeneity, it follows that 
$$\norm{\sum_ia_i^*a_i}^{1/2}\norm{\sum_ib_ib_i^*}^{1/2}\leq\mathrm{deg}(\h{\G})^2\norm{\sum_ia_ia_i^*}^{1/2}\norm{\sum_ib_i^*b_i}^{1/2}=\mathrm{deg}(\h{\G})^2\norm{\Gam(b)}_{\LIQ\whten\LIQ}$$
(see the proof of \cite[Proposition 5.1]{ATT}). Let $f$ be a finite-rank tensor in $\LOQ\ten^{\vee}\LOQ\hookrightarrow(\LIQ\pten\LIQ)^*$. Following the proof of Theorem \ref{t:ci}, the non-commutative Grothendieck inequality implies
\begin{align*}|\la\Gamma(b),f\ra|&\leq\sum_i|\la f,a_i\ten b_i\ra|\\
&\leq\norm{f}_{\LOQ\ten^{\vee}\LOQ}\bigg(\norm{\sum_ia_ia_i^*}^{1/2}\norm{\sum_ib_i^*b_i}^{1/2}+\norm{\sum_ia_i^*a_i}^{1/2}\norm{\sum_ib_ib_i^*}^{1/2}\bigg)\\
&=\norm{f}_{\LOQ\ten^{\vee}\LOQ}\bigg(\norm{\Gam(b)}_{\LIQ\whten\LIQ}+\norm{\sum_ia_i^*a_i}^{1/2}\norm{\sum_ib_ib_i^*}^{1/2}\bigg)\\
&\leq\norm{f}_{\LOQ\ten^{\vee}\LOQ}(\norm{\Gam(b)}_{\LIQ\whten\LIQ}+\mathrm{deg}(\h{\G})^2\norm{\Gam(b)}_{\LIQ\whten\LIQ})\\
&\leq\norm{f}_{\LOQ\ten^{\vee}\LOQ}\mathrm{deg}(\h{\G})(1+\mathrm{deg}(\h{\G})^2)\norm{b}_{M_{cb}^l(L^1(\h{\G}'))}.\end{align*}
By local reflexivity of $\LOQ$, it follows as above that $\Gam:M_{cb}^l(L^1(\h{\G}'))\rightarrow\mc{I}_{\LOQ}(\LOQ,\LIQ)$ is bounded by $\mathrm{deg}(\h{\G})(1+\mathrm{deg}(\h{\G})^2)$. 

Conversely, if $\Phi\in\mc{I}_{\LOQ}(\LOQ,\LIQ)$, then $\Phi=\Gam(x)$ for some $x\in\LIQ$ and $\gamma^{2,r}(\Gam(x))\leq\iota(\Gam(x))<\infty$. Thus, by Theorem \ref{t:va}, $x=b\in M_{cb}^l(L^1(\h{\G}'))$ and
$$\mathrm{deg}(\h{\G})^{-1}\norm{b}_{M_{cb}^l(L^1(\h{\G}'))}\leq\gamma^{2,r}(b)\leq\iota(b).$$
Weak*-homeomorphicity follows verbatim from the proof of Theorem \ref{t:ci}.

\end{proof}

\begin{cor} Let $\G$ be a Kac algebra such that $\h{\G}$ has bounded degree. Then
$$\LOQ\ten^{\vee}_{\LOQ}\LOQ\cong\LOQ\ten^{\vee/}_{\LOQ}\LOQ\cong\LOQ\hten_{\LOQ}\LOQ$$
isomorphically. In particular,
$$\mc{I}_{\LOQ}(\LOQ,\LIQ)=\Pi^1_{\LOQ}(\LOQ,\LIQ)=\Gam^{2,r}_{\LOQ}(\LOQ,\LIQ)$$
set-theoretically, with
$$\gamma^{2,r}(\Phi)\leq\pi^1(\Phi)\leq\iota(\Phi)\leq \mathrm{deg}(\h{\G})^2(1+\mathrm{deg}(\h{\G})^2)\gamma^{2,r}(\Phi),$$
for $\Phi\in\mc{CB}_{\LOQ}(\LOQ,\LIQ)$.
\end{cor}

\section{Completely Nuclear Multipliers and the Quantum Bohr Compactification}\label{s:cn}

Answering a question of Crombez and Govaerts \cite{CG}, Racher has shown that for any locally compact group $G$, the nuclear $\LO$-module maps from $\LO$ to $\LI$ are precisely convolution with almost periodic elements of the Fourier--Stieltjes algebra \cite[Theorem]{R}, specifically,
\begin{equation}\label{e:Racher}N_{\LO}(\LO,\LI)\cong B(G)\cap AP(G),\end{equation}
with $\norm{u}_{B(G)}\leq N(u)\leq K\norm{u}_{B(G)}$, where $K$ is Grothendieck's constant. Since $B(G)\cap AP(G)=A(bG)$, where $bG$ is the Bohr compactification of $G$ \cite[Proposition 2.1]{RS}, it is natural to pursue a quantum group version of (\ref{e:Racher}) in connection with the quantum Bohr compactification \cite{Sol}. We now establish this connection for the two classes of quantum groups considered above.

In \cite{Sol}, So\l tan introduced a quantization of the Bohr compactification valid for general locally compact quantum groups $\G$. This assignment associates to any $\G$ a compact quantum group $b\G$ and a quantum group morphism $\G\rightarrow b\G$, meaning a (multiplier) non-degenerate $*$-homomorphism $\pi:C_u(b\G)\rightarrow M(C_u(\G))$ intertwining the co-products, such that any quantum group morphism $\G\rightarrow\Hb$ into a compact quantum group $\Hb$ factors through $b\G$ \cite[Theorem 3.1]{Sol} (see also \cite[Proposition 3.4]{D2}). The image $\mathrm{AP}_u(\G):=\pi(C_u(b\G))$, denoted $\mathbb{AP}(C_u(\G))$ in \cite{D2}, is the norm closure of matrix coefficients of admissible finite-dimensional unitary co-representations of $\G$. In what follows the quantum Bohr compactification $b\G$ will always refer to the abstract compact quantum group and not the realization inside $M(C_u(\G))$. We let $\mathrm{AP}(\G)$, as opposed to $\mathbb{AP}(C_0(\G))$, denote the norm closure of coefficients of admissible finite-dimensional unitary co-representations of $\G$ inside $M(C_0(\G))$. For details we refer the reader to \cite{D2,Sol}.

The proof of the main result in this section combines techniques used in previous sections together with ideas from \cite{D2,DD}. In particular, we make heavy use of the quantum Eberlein compactification, defined and explored in \cite{DD}. We summarize a few properties we require, referring the reader to \cite{DD} for details.

Given a locally compact quantum group $\G$, its quantum Eberlien compactification $E(\G)$ is the unital $C^*$-algebra given by the norm closure of $\{(\hat{\mu}'\ten\id)(\mathbb{V}^{\G})\mid \hat{\mu}'\in C_u(\h{\G}')^*\}$ in $M(C_u(\G))$ \cite[\S7]{DD}, where $\mathbb{V}^{\G}\in M(C_u(\h{\G}')\ten^{\vee} C_u(\G))$ is the bi-universal right fundamental unitary of $\G$. Then $C_u(\G)\subseteq E(\G)\subseteq M(C_u(\G))$ and the restriction of the (strictly continuous extension of the) universal co-product satisfies $\Gam_u|_{E(\G)}:E(\G)\rightarrow E(\G)^{**}\oten E(\G)^{**}$. This in turn yields a Banach algebra structure on $E(\G)^*$ along with an $E(\G)^*$-module structure on $E(\G)$. By \cite[Theorem 4.7, Remark 4.8]{DD}, $E(\G)$ has a unique right invariant mean, that is, a state $M\in E(\G)^*$ satisfying
$$\la M,\mu\star x\ra=\la M,x\ra\la\mu,1\ra, \ \ \ x\in E(\G), \ \mu\in E(\G)^*.$$
Such a state is automatically left invariant and is the unique (two-sided) invariant mean (see \cite[Remark 4.8, Theorem 5.3]{DD}). Given $x\in C_u(\G)$, and $\mu\in E(\G)^*$, it follows from (the right version of) \cite[\S6]{K} that
\begin{align*}\Lambda_{\G}(\mu \star x)&=\Lambda_{\G}((\id\ten\mu)\Gam_u(x))=(\id\ten\mu)(\Lambda_{\G}\ten\id)(\Gam_u(x))=(\id\ten\mu)(\mathbb{V}(\Lambda_{\G}(x)\ten 1)\mathbb{V}^*)\\
&=(\mu|_{C_u(\G)})\star \Lambda_{\G}(x),\end{align*}
where $\mathbb{V}\in M(C_0(\h{\G}')\ten^{\vee}C_u(\G))$ is the right universal lift of $V$, i.e., $(\id\ten\Lambda_{\G})(\mathbb{V})=V$. By strict density of $C_u(\G)$ in $E(\G)$ and strict continuity of the module structure, it follows that $\Lambda_{\G}(\mu\star x)=(\mu|_{C_u(\G)})\star \Lambda_{\G}(x)$ for all $x\in E(\G)$ and $\mu\in E(\G)^*$. 

When $\G$ is a Kac algebra, the unique invariant mean $M$ on $E(\G)$ is a trace \cite[Proposition 7.5]{DD}. In that case $E_0(\G):=\{x\in E(\G)\mid M(x^*x)=0\}$ is a closed two-sided ideal in $E(\G)$ and there is an exact sequence of $C^*$-algebras
$$0\rightarrow E_0(\G)\hookrightarrow E(\G)\twoheadrightarrow C(\bG)\rightarrow 0.$$

If $U\in M(\mc{K}(H)\ten^{\vee} C_u(\G))$ is a unitary co-representation of $\G$, a vector $\xi\in H$ is said to be \textit{compact} if $(\om_{\xi,\eta}\ten \id)(U)\in\pi(C_u(\bG))$ for all $\eta\in H$ \cite[Definition 6.3]{DD}. The collection of compact vectors forms a closed subspace $H_c$ of $H$. When $\G$ is a Kac algebra, $U$ decomposes as $U=U_c\oplus U_0$ with respect to the decomposition $H=H_c\oplus H_c^{\perp}$ and $(\om_{\xi,\eta}\ten\id)(U_0)\in E_0(\G)$ for all $\xi,\eta\in H$ \cite[Theorem 7.3]{DD}.

Recall that the measure algebra $M(G)$ of a locally compact group $G$ decomposes as $M(G)=\ell^1(G_d)\oplus_1 M_c(G)$, where $\ell^1(G_d)$ and $M_c(G)$ are the discrete and continuous measures, respectively. Dually, it was shown in \cite[Theorem 2.3]{RS} that $B(G)=A_{\mc{F}}(G)\oplus_1 A_{\mc{F}}(G)^{\perp}$, where $A_{\mc{F}}(G)$ is the closed linear span of coefficient functions of finite-dimensional unitary representations of $G$. By \cite[Proposition 2.1]{RS}, $A_{\mc{F}}(G)=B(G)\cap AP(G)\cong A(bG)$, the Fourier algebra of the Bohr compactification of $G$. We now establish a corresponding result for general quantum groups.

\begin{lem}\label{l:Bohr} Let $\G$ be a locally compact quantum group and let $b\G$ be the Bohr compactification of $\G$. Then there exists a completely isometric, unital homomorphism $\ell^1(\widehat{b\G}')\hookrightarrow C_u(\h{\G}')^*$ onto a completely complemented subalgebra of $C_u(\h{\G}')^*$. In particular, $C_u(\h{\G}')^*\cong\ell^1(\widehat{b\G}')\oplus_1\ell^1(\widehat{b\G}')^{\perp}$.
\end{lem}

\begin{proof} The canonical morphism $\G\rightarrow\bG$ is implemented by a (multiplier) non-degenerate $*$-homomorphism $\pi:C_u(\bG)\rightarrow M(C_u(\G))$, with corresponding bicharacter $U^{b\G}\in M(c_0(\bGhp)\ten^{\vee} C_0(\G))$ satisfying
$$(\Gam_{\bGhp}\ten\id)(U^{b\G})=U^{b\G}_{23}U^{b\G}_{13}, \ \ \ (\id\ten\Gam_{\G})(U^{b\G})=U^{b\G}_{12}U^{b\G}_{13}.$$
The dual morphism $\h{\pi}':C_u(\h{\G}')\rightarrow M(c_0(\bGhp))=\ell^\infty(\bGhp)$ is a non-degenerate $*$-homomorphism satisfying
$$(\hat{\pi}'\ten \hat{\pi}') \circ\Gam^u_{\h{\G}'}=\Gam_{\bGhp}\circ \hat{\pi}',$$
and is related to $\pi$ through the formula
\begin{equation}\label{e:bi}(\id\ten(\Lambda_{\G}\circ\pi))(\mathbb{V}^{b\G})=U^{b\G}=(\hat{\pi}'\ten\Lambda_{\G})(\mathbb{V}^{\G}),\end{equation}
The co-product $\Gam^u_{\h{\G}'}:C_u(\h{\G}')\rightarrow M(C_u(\h{\G}')\ten^{\vee} C_u(\h{\G}'))$ canonically extends to a normal co-product $\Gam^u_{\h{\G}'}:C_u(\h{\G}')^{**}\rightarrow C_u(\h{\G}')^{**}\oten C_u(\h{\G}'))^{**}$ satisfying $\Gam^u_{\h{\G}'}=(m_{C_u(\h{\G}')^*})^*$, where
$$m_{C_u(\h{\G}')^*}:C_u(\h{\G}')^*\pten C_u(\h{\G}')^*\rightarrow C_u(\h{\G}')^*$$
is the multiplication on $C_u(\h{\G}')^*$. It follows that the normal cover $\widetilde{\hat{\pi}'}:C_u(\h{\G}')^{**}\twoheadrightarrow\ell^\infty(\bGhp)$ of $\hat{\pi}'$ is a normal surjective $*$-homomorphism intertwining the lifted co-product on $C_u(\h{\G}')^{**}$. Thus,
$$i:=(\widetilde{\hat{\pi}'})_*:\ell^1(\bGhp)\hookrightarrow C_u(\h{\G}')^*$$
is a completely isometric homomorphism. Moreover, if $e$ denotes the unit of $\ell^1(\bGhp)$ and $\rho_u:\LOQs\rightarrow C_u(\h{\G}')$ is the canonical map, then
\begin{align*}\la i(e),\rho_u(f)\ra&=\la i(e),(\id\ten f)((\id\ten\Lambda_{\G})(\mathbb{V}^{\G}))\ra=\la e\ten f,(\hat{\pi}'\ten\Lambda_{\G})(\mathbb{V}^{\G})\ra\\
&=\la e\ten f, U^{b\G}\ra=\la f,(e\ten\id)(U^{b\G})\ra=\la f,1\ra\\
&=\la \ep_u,\rho_u(f)\ra\end{align*}
for all $f\in\LOQs$. Hence, $i(e)=\ep_u$, the unit of $C_u(\h{\G}')^*$.

Now, $\mathrm{Ker}(\widetilde{\hat{\pi}'})=(1-z)C_u(\h{\G}')^{**}$ for a central projection $z\in C_u(\h{\G}')^{**}$. Hence, $z\cdot C_u(\h{\G}')^{**}\cong\ell^\infty(\widehat{b\G}')$, and it follows that $\mathrm{Im}(i)=z\cdot C_u(\h{\G}')^{*}$ is a completely contractively complemented subalgebra of $C_u(\h{\G}')^{*}$. 

\end{proof}

It is well-known that a quantum group $\G$ is compact if and only if $\h{\G}$ is discrete, and in that case, $\ell^1(\h{\G})\cong\bigoplus_{1} \{T_{n_\alpha}(\C)\mid\alpha\in\Irr\}$,
where $T_{n_\alpha}(\C)$ is the space of $n_\alpha\times n_\alpha$ trace class operators, and $\Irr$ denotes the set of (equivalence classes of) irreducible co-representations of the compact quantum group $\G$ \cite{Wo}. 

\begin{thm}\label{t:Nuclear} Let $\G$ be a locally compact quantum group for which $\h{\G}$ is a QSIN Kac algebra. Then
$$\mc{N}_{\LOQ}(\LOQ,\LIQ)\cong \ell^1(\bGhp)$$
isomorphically, with $\norm{\hat{f}'}\leq \nu(\hat{f}')\leq 2\norm{\hat{f}'}$, for all $\hat{f}'\in\ell^1(\bGhp)$. 
\end{thm}

\begin{proof} Throughout the proof we adopt the notation from Lemma \ref{l:Bohr} without comment. Viewing $U^{b\G}\in\ell^\infty(\bGhp)\oten\LIQ\cong\prod_{\alpha\in\mathrm{Irr}(\bG)} M_{n_\alpha}(\LIQ)$, we may write
$U^{b\G}=\sum_{\alpha} U^\alpha$, where each $U^\alpha=[u^\alpha_{ij}]$ is a unitary in $M_{n_\alpha}(\LIQ)$. Let $\hat{f}'\in\ell^1(\bGhp)\cong\bigoplus_\alpha T_{n_\alpha}$ be finitely supported. Then,
$$\Gam(\rho_*(i(\hat{f}')))=\Gam((i(\hat{f}')\ten\id)(V))=\Gam((\hat{f}'\ten \id)(U^{b\G}))=\sum_{\alpha}(\hat{f}'_\alpha\ten\id)((\id\ten\Gam)(U^\alpha)),$$
where for each $\alpha\in\mathrm{Irr}(\bG)$,
$$(\id\ten\Gam)(U^\alpha)=[\Gam(u^\alpha_{ij})]=\bigg[\sum_{k=1}^{n_\alpha}u_{ik}^\alpha\ten u_{kj}^\alpha\bigg]\in M_{n_\alpha}(\LIQ\ten\LIQ).$$
Thus,
\begin{align*}\norm{\Gam(\rho_*(i(\hat{f}')))}_h&\leq\sum_{\alpha}\norm{\hat{f}'_\alpha}_{T_{n_\alpha}}\norm{[\Gam(u^\alpha_{ij})]}_h\\
&\leq\sum_{\alpha}\norm{\hat{f}'_\alpha}_{T_{n_\alpha}}\norm{[u^\alpha_{ij}]}_{M_{n_\alpha}(\LIQ)}^2\\
&\leq\sum_{\alpha}\norm{\hat{f}'_\alpha}_{T_{n_\alpha}}\\
&=\norm{\hat{f}'}_1,\end{align*}
implying $\Gam(\rho_*(i(\hat{f}')))\in\LIQ\hten\LIQ$. 

Now, since $\G$ is a Kac algebra, it follows that $\bG$ is a compact Kac algebra (see the proof of \cite[Proposition 7.5]{DD}) . In particular, $u^{\overline{\alpha}}_{ij}=(u^\alpha_{ij})^*$ and $\hat{R}'(e^\alpha_{ij})=e^{\overline{\alpha}}_{ji}$, where $\overline{\alpha}$ is the conjugate co-representation to $\alpha$, and $e^\alpha_{ij}$ are the canonical matrix units in the decomposition $\ell^\infty(\bGhp)\cong\prod_{\alpha\in\mathrm{Irr}(\bG)} M_{n_\alpha}$. For notational simplicity, let $v^\alpha_{ij}:=u^{\overline{\alpha}}_{ji}$. Then
\begin{align*}\Sigma(\Gam(\rho_*(i(\hat{f}'))))&=
\sum_{\alpha}\sum_{i,j=1}^{n_\alpha}\la\hat{f}'_\alpha, e^\alpha_{ij}\ra\bigg(\sum_{k=1}^{n_\alpha} u_{kj}^\alpha\ten u_{ik}^\alpha\bigg)\\
&=\sum_{\alpha}\sum_{i,j=1}^{n_\alpha}\la\hat{f}'_\alpha\circ\hat{R}',e_{ji}^{\overline{\alpha}}\ra\bigg(\sum_{k=1}^{n_\alpha} u_{kj}^\alpha\ten u_{ik}^\alpha\bigg)\\
&=\sum_{\alpha}\sum_{i,j=1}^{n_\alpha}\la\hat{f}'_\alpha\circ\hat{R}',e_{ij}^{\alpha}\ra\bigg(\sum_{k=1}^{n_\alpha} u_{ki}^{\overline{\alpha}}\ten u_{jk}^{\overline{\alpha}}\bigg)\\
&=\sum_{\alpha}\sum_{i,j=1}^{n_\alpha}\la\hat{f}'_\alpha\circ\hat{R}',e_{ij}^{\alpha}\ra\bigg(\sum_{k=1}^{n_\alpha} v_{ik}^{\alpha}\ten v_{kj}^{\alpha}\bigg).\end{align*}
Thus,
\begin{align*}\norm{\Sigma(\Gam(\rho_*(i(\hat{f}'))))}_h&\leq\sum_{\alpha}\norm{\hat{f}'_\alpha\circ\hat{R}'}_{T_{n_\alpha}}\norm{[\sum_{k=1}^{n_\alpha} v_{ik}^{\alpha}\ten v_{kj}^{\alpha}]}_h\\
&\leq\sum_{\alpha}\norm{\hat{f}'_\alpha\circ\hat{R}'}_{T_{n_\alpha}}\norm{[v^\alpha_{ij}]}_{M_{n_\alpha}(\LIQ)}^2\\
&=\sum_{\alpha}\norm{\hat{f}'_\alpha}_{T_{n_\alpha}}\norm{[u^\alpha_{ij}]^*}_{M_{n_\alpha}(\LIQ)}^2\\
&=\norm{\hat{f}'}_1,\end{align*}
implying $\Sigma(\Gam(\rho_*(i(\hat{f}'))))\in\LIQ\hten\LIQ$. 

By the non-commutative Grothendieck inequality \cite{HM,PS}, it follows that
$$\norm{\Gam(\rho_*(i(\hat{f}')))}_{\wedge}\leq 2\max\{\norm{\Gam(\rho_*(i(\hat{f}')))}_h,\norm{\Sigma(\Gam(\rho_*(i(\hat{f}'))))}_h\}=2\norm{\hat{f}'}_1,$$
implying that the nuclear norm of the associated convolution map $\nu(\rho_*(i(\hat{f}')))\leq2\norm{\hat{f}'}_1$.
Hence, $\Gam\circ\rho_*\circ i:\ell^1(\bGhp)\rightarrow\mc{N}_{\LOQ}(\LOQ,\LIQ)$ is bounded by 2.

Now, suppose that $\Psi\in\mc{N}_{\LOQ}(\LOQ,\LIQ)$. Then $\Psi\in\mc{CB}_{\LOQ}(\LOQ,\LIQ)$, so there exists $x\in\LIQ$ such that $\Psi=\Gam(x)$. But then nuclearity implies the existence of $X\in\LIQ\pten\LIQ$ for which 
$$\Phi_{\vee}(X)=\Phi_{h,\vee}(\Phi_h(X))=\Gam(x)\in\LIQ\ten^{\vee}\LIQ\subseteq\mc{CB}(\LOQ,\LIQ).$$
Since $\Phi_h(X)\in\LIQ\hten\LIQ$, it follows that $\Gam(x)\in\LIQ\hten\LIQ$. By Theorem \ref{t:Gilbert}, there exists (a unique) $\hat{\mu}_0'\in M(\h{\G}')$ such that $x=\rho_*(\hat{\mu}_0')$. Let $A=\overline{\rho_*(M(\h{\G}'))}$. As the right fundamental unitary $V$ of $\LIQ$ lives in $M(C_0(\h{\G}')\ten_{\min}\mc{K}(\LTQ))$, we may view it inside $C_0(\h{\G}')^{**}\oten \BLTQ$. Given $m\in\BLTQ^*$, let $c=(\id\ten m)(V)\in C_0(\h{\G}')^{**}$. Then 
\begin{align*}(m\ten\id)\Gam(x)&=(m\ten\id)\Gam(\rho_*(\hat{\mu}_0'))=(m\ten\id)(\h{\mu}_0'\ten\id\ten\id)(\id\ten\Gam^r)(V)\\
&=(\hat{\mu}_0'\ten\id)(\id\ten m\ten\id)(V_{12}V_{13})=(\hat{\mu}_0'\ten\id)((c\ten 1)V)\\
&=\rho_*(\hat{\mu}_0'\cdot c),\end{align*}
where $\hat{\mu}_0'\cdot c$ is the canonical action of the von Neumann algebra $C_0(\h{\G}')^{**}$ on its predual $M(\h{\G}')$. Similarly, one can show that $(\id\ten m)\Gam(x)=\rho_*(c\cdot\hat{\mu}_0')$. Hence, 
$$\Gam(x)\in\mc{F}(A,A;\LIQ\hten\LIQ)=A\hten A,$$
where the last equality follows from \cite[Corollary 4.8]{Smith}.

Now, let $(\xi_i)$ be a net of unit vectors in $\LTQ$ witnessing the QSIN property for $\h{\G}$, and assume as in Theorem \ref{t:Gilbert} that $M=w^*-\lim_i\om_{\xi_i}\in\BLTQ^*$. Then property (ii) of Definition \ref{d:QSIN} implies that $M|_{\LIQ}$ is a right invariant mean, as
$$\la M,\om_\eta\star x\ra=\lim_i\la V(x\ten 1)V^*\xi_i\ten\eta,\xi_i\ten\eta\ra=\lim_i\la(x\ten 1)\xi_i\ten\eta,\xi_i\ten\eta\ra=\la M,x\ra\la\om_\eta,1\ra$$
for every $\eta\in\LTQ$ and $x\in\LIQ$.  Then $\widetilde{M}:=M\circ\Lambda_{\G}\in M(C_u(\G))^*$ yields a right invariant mean on $E(\G)$ by restriction, as
$$\la \widetilde{M},\mu\star x\ra=\la M,(\mu|_{C_u(\G)})\star \Lambda_{\G}(x)\ra=\la M,f\star(\mu|_{C_u(\G)})\star \Lambda_{\G}(x)\ra=\la M,\Lambda_{\G}(x)\ra\la f\star \mu,1\ra=\la \widetilde{M},x\ra\la\mu,1\ra$$
for all $x\in E(\G)$, $\mu\in E(\G)^*$ and states $f\in\LOQ$. By \cite[Theorem 5.3]{DD}, $\widetilde{M}|_{E(\G)}$ is the unique (two-sided) invariant mean. Let $(\pi_{\widetilde{M}},H,\xi_{\widetilde{M}})$ be its associated GNS construction. We show that
$$\pi_{\widetilde{M}}(\rho^u_*(\hat{\mu}'))=\pi_{\widetilde{M}}(\rho^u_*(\hat{\mu}'\cdot z)), \ \ \ \hat{\mu}'\in M(\h{\G}').$$
To this end, fix a positive $\hat{\mu}'\in M(\h{\G}')$ and let $(\pi_{\hat{\mu}'},H_{\hat{\mu}'},\xi_{\hat{\mu}'})$ be a cyclic GNS representation of $\hat{\mu}'$. Since $\h{\G}'$ is co-amenable, $C_0(\h{\G}')$ is the universal envelopping $C^*$-algebra of $\LOQ$, so there exists a unique unitary co-representation $U_{\hat{\mu}'}\in M(\mc{K}(H_{\hat{\mu}'})\ten^{\vee} C_u(\G))$ for which 
$$\pi_{\hat{\mu}'}(\rho(f))=(\id\ten f\circ\Lambda_{\G})(U_{\hat{\mu}'}), \ \ \ f\in\LOQ,$$
namely $U_{\hat{\mu}'}=(\pi_{\hat{\mu}'}\ten\id)(\mathbb{V}^{\G})$. By \cite[Theorem 7.3]{DD}, $U_{\hat{\mu}'}$ decomposes into $U_{\hat{\mu}'}=U_c\oplus U_0$, with respect to $H_{\hat{\mu}'}=H_c\oplus H_c^{\perp}$. At the level of $*$-homomorphisms, we see that $\pi_{\hat{\mu}'}$ decomposes 
as $\pi_{\hat{\mu}'}=\pi_c\oplus\pi_0$. Moreover, $\pi_c$ is non-degenerate by \cite[Corollary 6.8]{DD}. Also, by \cite[Proposition 6.6, Proposition 6.7]{DD}, there is a unitary co-representation $Y_c\in M(\mc{K}(H_c)\ten^{\vee} C_u(b\G))$ of $L^1(\bG)$ such that $(\id\ten\pi)(Y_c)=U_c$. Let $\pi_c^{\bG}$ denote the $*$-homomorphism $C_u(\bG)\rightarrow\mc{B}(H_c)$ satisfying $Y_c=(\pi_c^{\bG}\ten\id)(\mathbb{V}^{\bG})$. From the duality relation (\ref{e:bi}) we have 
\begin{align*}(\pi_c\ten\Lambda_{\G})(\mathbb{V}^{\G})&=(\id\ten\Lambda_{\G})(U_c)=(\id\ten\Lambda_{\G}\circ\pi)(Y_c)\\
&=(\pi_c^{\bG}\ten\Lambda_{\G}\circ\pi)(\mathbb{V}^{\bG})=(\pi_c^{\bG}\circ\hat{\pi}'\ten\Lambda_{\G})(\mathbb{V}^{\G}).\end{align*}
By density of $\{(\id\ten f\circ\Lambda_{\G})(\mathbb{V}^{\G})\mid f\in\LOQ\}$ in $C_0(\h{\G}')$, it follows that $\pi_c=\pi_c^{\bG}\circ\hat{\pi}'$. Keeping the same notation for the normal covers of the $*$-homomorphisms $\pi_c$, $\pi_c^{\bG}$ and $\hat{\pi}'$, it follows from non-degeneracy that $\pi_c(z)=\pi_c^{\bG}(\hat{\pi}'(z))=\pi_c^{\bG}(1)=1_{H_c}=p_c$, where $p_c:H_{\hat{\mu}'}\rightarrow H_c$ is the orthogonal projection. Whence, 
$$\pi_{\hat{\mu}'}(z)=\pi_c(z)\oplus\pi_0(z)=p_c\oplus\pi_0(z).$$
From this it follows that $U_c=(\pi_{\hat{\mu}'}(z)\ten 1)U_c$. 

Finally, by \cite[Theorem 7.3]{DD} we have $\pi_{\widetilde{M}}((\om\ten\id)(U_0))=0$ for all $\om\in\mc{T}(H_{\hat{\mu}'})$. Putting things together we see that 
\begin{align*}\pi_{\widetilde{M}}(\rho_*^u(\hat{\mu}'))&=\pi_{\widetilde{M}}((\hat{\mu}'\ten\id)(\mathbb{V}^{\G}))\\
&=\pi_{\widetilde{M}}((\om_{\xi_{\hat{\mu}'}}\ten\id)(U_{\hat{\mu}'}))\\
&=\pi_{\widetilde{M}}((\om_{\xi_{\hat{\mu}'}}\ten\id)(U_c)+(\om_{\xi_{\hat{\mu}'}}\ten\id)(U_0))\\
&=\pi_{\widetilde{M}}((\om_{\xi_{\hat{\mu}'}}\ten\id)(U_c))\\
&=\pi_{\widetilde{M}}((\om_{\xi_{\hat{\mu}'}}\ten\id)((\pi_{\hat{\mu}'}(z)\ten 1)U_c))\\
&=\pi_{\widetilde{M}}((\om_{\xi_{\hat{\mu}'}}\cdot\pi_{\hat{\mu}'}(z)\ten\id)(U_c))\\
&=\pi_{\widetilde{M}}((\om_{\xi_{\hat{\mu}'}}\cdot\pi_{\hat{\mu}'}(z)\ten\id)(U_{\hat{\mu}'}))\\
&=\pi_{\widetilde{M}}((\om_{\xi_{\hat{\mu}'}}\ten\id)((\pi_{\hat{\mu}'}(z)\ten 1)U_{\hat{\mu}'}))\\
&=\pi_{\widetilde{M}}((\om_{\xi_{\hat{\mu}'}}\circ\pi_{\hat{\mu}'}\ten\id)((z\ten 1)\mathbb{V}^{\G}))\\
&=\pi_{\widetilde{M}}((\hat{\mu}'\ten\id)((z\ten 1)\mathbb{V}^{\G}))\\
&=\pi_{\widetilde{M}}(\rho_*^u(\hat{\mu}'\cdot z)).\end{align*}
By Jordan decomposition, it follows that $\pi_{\widetilde{M}}(\rho^u_*(\hat{\mu}'))=\pi_{\widetilde{M}}(\rho^u_*(\hat{\mu}'\cdot z))$ for all $\hat{\mu}'\in M(\h{\G}')$. 

We are now in position to quantize the argument from \cite[Lemma 5.8]{D2}. By above, for any $\hat{\mu}'\in M(\h{\G}')$, 
\begin{align*}&\lim_i\norm{\rho_*(\hat{\mu}')\xi_i-\rho_*(\hat{\mu}'\cdot z)\xi_i}^2\\
&=\lim_i\bigg(\la\rho_*(\hat{\mu}')\xi_i,\rho_*(\hat{\mu}')\xi_i\ra-2\mathrm{Re}\la\rho_*(\hat{\mu}')\xi_i,\rho_*(\hat{\mu}'\cdot z)\xi_i\ra+\la\rho_*(\hat{\mu}'\cdot z)\xi_i,\rho_*(\hat{\mu}'\cdot z)\xi_i\ra\bigg)\\
&=\la M,\rho_*(\hat{\mu}')^*\rho_*(\hat{\mu}')\ra-2\mathrm{Re}\la M,\rho_*(\hat{\mu}'\cdot z)^*\rho_*(\hat{\mu}')\ra+\la M,\rho_*(\hat{\mu}'\cdot z)^*\rho_*(\hat{\mu}'\cdot z)\ra\\
&=\la \widetilde{M},\rho^u_*(\hat{\mu}')^*\rho^u_*(\hat{\mu}')\ra-2\mathrm{Re}\la\widetilde{M},\rho^u_*(\hat{\mu}'\cdot z)^*\rho^u_*(\hat{\mu}')\ra+\la\widetilde{M},\rho^u_*(\hat{\mu}'\cdot z)^*\rho^u_*(\hat{\mu}'\cdot z)\ra\\
&=\norm{\pi_{\widetilde{M}}(\rho^u_*(\hat{\mu}'))\xi_{\widetilde{M}}}^2-2\mathrm{Re}\la\pi_{\widetilde{M}}(\rho^u_*(\hat{\mu}'))\xi_{\widetilde{M}},\pi_{\widetilde{M}}(\rho^u_*(\hat{\mu}'\cdot z))\xi_{\widetilde{M}}\ra+\norm{\pi_{\widetilde{M}}(\rho^u_*(\hat{\mu}'\cdot z))\xi_{\widetilde{M}}}^2\\
&=0.\end{align*}
Let $\Phi:\LIQ\oten\LIQ\rightarrow\LIQ$ denote the left inverse of $\Gam$ from the proof of Theorem \ref{t:Gilbert}. Then for any $\hat{\mu}'\in M(\h{\G}')$, $x\in\LIQ$, and $f\in\LOQ$ we have
\begin{align*}\la\Phi(x\ten\rho_*(\hat{\mu}')),f\ra&=\lim_i\la U^*xU\rho(f)\rho_*(\hat{\mu}')\xi_i,\xi_i\ra\\
&=\lim_i\la U^*xU\rho(f)\rho_*(\hat{\mu}'\cdot z)\xi_i,\xi_i\ra\\
&=\la\Phi(x\ten\rho_*(\hat{\mu}'\cdot z)),f\ra.\end{align*}
Moreover, if $\hat{f}'_\al\in T_{n_\alpha}\subseteq\ell^1(\widehat{\bG}')$, then
\begin{align*}\la\Phi(x\ten\rho_*(i(\hat{f}'_\al))),f\ra&=\la M\ten f,(U^*xU\ten 1)\Gam(\rho_*(i(\hat{f}'_\al)))\ra\\
&=\sum_{i,j,k=1}^{n_\al}\la\hat{f}'_\al,e_{ij}^\al\ra\la M\ten f, (U^*xU\ten 1)(u_{ik}^\al\ten u_{kj}^\al)\ra\\
&=\sum_{i,j,k=1}^{n_\al}\la\hat{f}'_\al,e_{ij}^\al\ra\la M,U^*xUu_{ik}^\al\ra\la f,u_{kj}^\al\ra\\
&=\sum_{i,j,k=1}^{n_\al}\la\hat{f}'_\al,e_{ij}^\al\ra\la M,U^*xUu_{ik}^\al\ra\la f,\rho_*(i(\hat{f}'_{kj}))\ra \ \ \ \ \textnormal{(some $\hat{f}'_{kj}\in\ell^1(\widehat{\bG}')$)}\\
&=\bigg\la\sum_{i,j,k=1}^{n_\al}\la\hat{f}'_\al,e_{ij}^\al\ra\la M,U^*xUu_{ik}^\al\ra \rho_*(i(\hat{f}'_{kj})),f\bigg\ra.\end{align*}
Letting $\widetilde{\Phi}$ denote the corresponding map $\LIQ\whten\LIQ\rightarrow M(\h{\G}')$ (see proof of Theorem \ref{t:Gilbert}), it follows that 
$$\widetilde{\Phi}(\rho_*(M(\h{\G}'))\ten \rho_*(M(\h{\G}')))=\widetilde{\Phi}(\rho_*(M(\h{\G}'))\ten \rho_*(i(\ell^1(\widehat{\bG}'))))\subseteq i(\ell^1(\widehat{\bG}')).$$
By continuity in the Haagerup norm $\widetilde{\Phi}(A\hten A)\subseteq i(\ell^1(\widehat{\bG}'))$, and we have
$$\hat{\mu}_0'=\widetilde{\Phi}(\Gam(x))\in\widetilde{\Phi}(A\hten A)\subseteq i(\ell^1(\widehat{\bG}')).$$
Letting $\hat{f}'_0\in \ell^1(\widehat{\bG}')$ be the corresponding element, we have 
$$ \nu(\Psi)=\nu(\Gam(\rho_*(i(\hat{f}'_0))))\geq\gamma^{2,r}(\Gam(\rho_*(i(\hat{f}'_0))))=\norm{\hat{f}'_0}.$$
\end{proof}

For a locally compact quantum group $\G$, and $x\in\LIQ$, we let $L_x:\LOQ\ni f\mapsto x\star f\in\LIQ$. Inspection of the proof of Theorem \ref{t:Nuclear} yields the following.

\begin{cor}\label{c:APQSIN} Let $\G$ be a Kac algebra for which $\h{\G}$ is QSIN. The following are equivalent for an element $x\in\LIQ$:
\begin{enumerate}
\item $x\in\rho_*(i(\ell^1(\h{b\G}')))$;
\item $L_x\in\mc{N}_{\LOQ}(\LOQ,\LIQ)$;
\item $\Gam(x)\in\LIQ\hten\LIQ$.
\end{enumerate}
In particular, $\mathrm{AP}(\G)=\la\{x\in\LIQ\mid\Gam(x)\in\LIQ\hten\LIQ\}\ra$. 
\end{cor}

\begin{proof} The only claim which requires proof is the final one, but this follows readily from the fact that 
$\mathrm{AP}(\G)$ is the norm closure of $\rho_*(i(\ell^1(\h{b\G}')))$, since any finite-dimensional co-representation of a Kac algebra is admissible \cite[Corollary 4.17]{D2}.
\end{proof}

In the co-commutative setting, the quantum Bohr compactification of $\G_s=VN(G)$ is $b\G_s=G_d$, the discretized group $G$. Also, $\rho_*:M(G)\rightarrow VN(G)$ is nothing but $\lm$, and $\mathrm{AP}(\G_s)=C^*_\delta(G)=\overline{\lm(\ell^1(G_d))}$. Thus, we immediately obtain:

\begin{cor} Let $G$ be a QSIN locally compact group. Then 
$$\mc{N}_{A(G)}(A(G),VN(G))\cong \ell^1(G_d)$$
isomorphically, with $\norm{f}\leq \nu(f)\leq 2\norm{f}$, $f\in\ell^1(G_d)$. Moreover, an element $x\in VN(G)$ lies in $\lm(\ell^1(G_d))$ if and only if $\Gam(x)\in VN(G)\hten VN(G)$, and $C^*_\delta(G)=\la\{x\in VN(G)\mid \Gam(x)\in VN(G)\hten VN(G)\}\ra$.\end{cor}

\begin{remark} For a locally compact group $G$, it is well-known that the Bohr compactification $bG$ of $G$, or rather $C(bG)$, is isomorphic to $\mathrm{AP}(G)$ -- the set functions $f\in C_b(G)$ which generate compact $\LO$-module maps $\LO\rightarrow\LI$ under convolution. As shown in \cite[\S4]{D2}, this relationship does not persist beyond commutative quantum groups. That is, one cannot recover the quantum Bohr compactification by considering (completely) compact module maps $\LOQ\rightarrow\LIQ$. Indeed, by \cite[\S4.1]{D2},\cite[Proposition 3.1]{Chou} and \cite[Proposition 1]{Rindler}, if $G$ is any infinite tall compact group with the mean-zero weak containment property (see \cite{Chou,Rindler}) then the rank-one projection $p\in VN(G)$ onto the constant functions in $\LT$ defines a completely compact $A(G)$-module map, but $p\notin C^*_\delta(G)=\mathrm{AP}(\G_s)$. For such $G$, $VN(G)$ is injective so that 
$$\mc{CK}_{A(G)}(A(G),VN(G))=\{x\in VN(G)\mid\Gam(x)\in VN(G)\iten VN(G)\}$$
by \cite[Theorem 2.4]{Run}. As $G$ is also QSIN, by Corollary \ref{c:APQSIN} we see that 
$$\{x\in VN(G)\mid\Gam(x)\in VN(G)\iten VN(G)\}\neq\la\{x\in VN(G)\mid\Gam(x)\in VN(G)\hten VN(G)\}\ra.$$
Theorem \ref{t:Nuclear} shows, however, that for a large class of quantum groups one \textit{can} recover the (discrete dual of the) quantum Bohr compactification by considering completely \textit{nuclear} module maps $\LOQ\rightarrow\LIQ$, as opposed to completely compact ones.\end{remark}

Given a locally compact quantum group, we let 
$$\mathrm{CAP}(\G):=\{x\in\LIQ\mid L_x\in\mc{CK}_{\LOQ}(\LOQ,\LIQ)\}.$$
We always have $\mathrm{AP}(\G)\subseteq \mathrm{CAP}(\G)$ \cite[Proposition 4.9]{D2}, but as remarked above, $\mathrm{CAP}(\G)$ does not identify with either the quantum Bohr compactification, nor the almost periodic elements, even in the case of $\G=VN(G)$ for certain infinite tall compact groups $G$. Recall that $G$ is tall if $|\h{G}_n|<\infty$ for every $n\in\N$, where $\h{G}_n:=\{\pi\in\h{G}\mid \mathrm{dim}(\pi)=n\}$. Note that $VN(G)$ cannot be subhomogeneous for an infinite tall compact group. We now show that in the presence of subhomogeneity, which is automatic in the commutative case, $\mathrm{CAP}(\G)$ \textit{does} recover the quantum Bohr compactification. 

%The following lemma is a consequence of \cite[Theorem 4.6]{Qu}, but we provide a proof for completeness.

%\begin{lem}\label{l:msub} Let $A$ be a subhomogenous $C^*$-algebra. Then multiplication $m_A:A\ten A\rightarrow A$ extends to a completely bounded map $A\iten A\rightarrow A$.\end{lem}

%\begin{proof} Since $A$ is $*$-isomorphic to a $C^*$-subalgebra of $M_n(C_0(X))$ for some $n\in\N$ and locally compact Hausdorff space $X$, by injectivity of $\iten$ it suffices to prove the Lemma for $A=M_n(C_0(X))$. Under the canonical identification $M_n(C_0(X))=M_n\iten C_0(X)$, the multiplication
%$$m_{A}:M_n\ten C_0(X)\ten M_n\ten C_0(X)\rightarrow M_n\ten C_0(X)$$
%satisfies $m_A=m_{M_n}\ten m_{C_0(X)}\circ \Sigma_{23}$. Since $C_0(X)$ is a min space, $C_0(X)\iten C_0(X)\cong \mathrm{min}(C_0(X)\ten^{\lm} C_0(X))$, completely isometrically where $\ten^{\lm}$ is the Banach space injective tensor norm. As multiplication $m_{C_0(X)}:C_0(X)\ten^{\lm} C_0(X)\rightarrow C_0(X)$ is bounded (reference), it follows that $m_{C_0(X)}:C_0(X)\iten C_0(X)\rightarrow C_0(X)$ is completely bounded \cite[\S3.3]{ER}. Also, as $M_n\iten M_n$ is finite dimensional, $m_{M_n}:M_n\iten M_n\rightarrow M_n$ is completely bounded. Since the flip map is completely isometric with respect to $\iten$, it follows that $m_A=m_{M_n}\ten m_{C_0(X)}\circ \Sigma_{13}$ extends to a completley bounded map 
%$$m_A:M_n\iten C_0(X)\iten M_n\iten C_0(X)\rightarrow M_n\iten C_0(X).$$
%\end{proof}

\begin{thm}\label{t:AP} Let $\G$ be a Kac algebra such that $\h{\G}$ has bounded degree. Then $\mathrm{CAP}(\G)=\mathrm{AP}(\G)$, and $\mathrm{CAP}(\G)$ is $*$-isomorphic to $C(b\G)$.\end{thm}

\begin{proof} First, if $x\in\mathrm{CAP}(\G)$, then as $\G$ is co-amenable (Lemma \ref{l:co-amen}), it follows by compactness of $L_x$ that $x$ lies in the norm closure of $\{x\star f\mid f\in \LOQ\}$, so that $x\in\LUC\subseteq M(C_0(\G))$ \cite[Theorem 2.4]{Run2}, and therefore $\mathrm{CAP}(\G)\subseteq M(C_0(\G))$. The idea is to show that $(\mathrm{CAP}(\G),\Gam)$ is a compact quantum group and then appeal to the universal property of the quantum Bohr compactification to deduce that the inclusion $\mathrm{CAP}(\G)\subseteq M(C_0(\G))$ must factor through $C(b\G)$. For notation simplicity we let $A:=\mathrm{CAP}(\G)$. Since $A\subseteq M(C_0(\G))$, it follows that $A$ is a subhomogeneous unital $C^*$-algebra. Also, as $\LIQ$ is injective, we have $A=\{x\in\LIQ\mid\Gam(x)\in\LIQ\iten\LIQ\}$ \cite[Theorem 2.4]{Run}. 

Let $x\in A$ and $m\in\LIQ^*$. As $\Gam(x)\in\LIQ\iten\LIQ$, $L_x$ is a cb-norm limit of finite rank maps. Let $L_m, R_m\in\mc{CB}(\LIQ)$ be given by
$$L_m(y)=(\id\ten m)\Gam(y), \ \ \ R_m(y)=(m\ten\id)\Gam(y), \ \ \ y\in\LIQ.$$
Then $L_m$ is a right $\LOQ$-module map so that
$$L_{L_m(x)}(f)=L_m(x)\star f=L_m(x\star f)=L_m(L_x(f))=L_m\circ L_m(f), \ \ \ f\in\LOQ.$$
Since $L_x$ is a cb-norm limit of finite rank maps, it follows that $\Gam(L_m(x))\in\LIQ\iten\LIQ$. Similarly, if $R_x:\LOQ\ni f\mapsto f\star x\rightarrow\LIQ$, then $R_m(x)=R_x^*(m)$ and so
$$L_{R_m(x)}(f)=R_x^*(m)\star f=R_x^*(m\star f)=R_x^*(R_m^*(f))=R_x^*\circ R_m^*(f), \ \ \ f\in\LOQ.$$
As $\Gam(x)\in\LIQ\iten\LIQ$, $R_x$ is also a cb-norm limit of finite rank maps, so that $R_x^*$ is completely compact \cite[Proposition 1.6]{Run}. Thus, $\Gam(R_m(x))\in\LIQ\iten\LIQ$. Since $A$ and $\LIQ$ are subhomogeneous $C^*$-algebras, they are nuclear, and therefore have the $\iten$-slice map property with respect to subspaces of any operator space \cite[Theorem 11.3.1]{ER}. Thus, 
$$\Gam(x)\in\mc{F}(A,A;\LIQ\iten\LIQ)=A\iten A,$$ 
so that $\Gam|_A:A\rightarrow A\iten A$ induces a co-product on $A$. We show that 
$$\la(1\ten A)\Gam(A)\ra=\la(A\ten 1)\Gam(A)\ra=A\iten A,$$
from which it follows that $(A,\Gam)$ is a compact quantum group in the sense of \cite{Sol,Wo}.

Fix $x\in A$. Then $\Gam(x)\in A\iten A$, and by complete boundedness of the unitary antipode $R$ we also have 
$$((\id\ten R)\ten \Gam)(1\ten\Sigma\Gam(x))\in (A\iten A)\iten (A\iten A).$$
By \cite[Theorem 4.6]{Qu} and injectivity of $\iten$, it follows that multiplication on a subhomogeneous $C^*$-algebra is bounded in the injective norm. In particular, as $A\iten A$ is subhomogeneous, 
$$m_{A\iten A}(((\id\ten R)\ten \Gam)(1\ten\Sigma\Gam(x)))\in A\iten A.$$
We show that 
$$m_{A\iten A}(((\id\ten R)\ten \Gam)(1\ten\Sigma\Gam(x)))=x\ten 1.$$
To this end, let
$$\mc{T}_\psi:=\{x\in\mc{N}_\psi\cap\mc{N}^*_\psi\mid \textnormal{$x$ is analytic with respect to $\sigma^{\psi}$ and $\sigma^{\psi}_z(x)\in\mc{N}_\psi\cap\mc{N}^*_\psi$, for $z\in\C$}\}$$ 
denote the (weak*-dense) Tomita $*$-algebra associated to the right Haar weight $\psi$, and let $\mc{A}_\psi:=\la\mc{T}_\psi^2\ra$. Then for any $x\in\mc{A}_\psi$ we have $x\cdot\psi,\psi\cdot x\in\LOQ$, where $x\cdot\psi(y)=\psi(yx)$ and $\psi\cdot x(y)=\psi(xy)$. For instance, if $a,b\in\mc{T}_\psi$, then for all $y\in\mc{M}_\psi$
$$ab\cdot\psi(y)=\psi(yab)=\psi(\sigma^{\psi}_{i}(b)ya)=\la y\Lpsi(a),\Lpsi(\sigma^{\psi}_{i}(b)^*)\ra,$$
from which it follows that $ab\cdot\psi$ extends to the normal linear functional $\om_{\Lpsi(a),\Lpsi(\sigma^{\psi}_{i}(b)^*)}|_{\LIQ}$. By weak*-density of $\mc{T}_\psi$ it follows that $\mc{A}_\psi$ is a weak*-dense $*$-subalgebra of $\LIQ$. 

By Kaplansky's density theorem, pick a net $(e_i)$ in $\mc{A}_\psi$ satisfying $\norm{e_i}_{\LIQ}\leq 1$ and $e_i\rightarrow 1$ strongly. Then $\Gam(e_i)\rightarrow 1\ten 1$ weak*, and by separate weak* continuity of multiplication we have
\begin{align*}&m_{A\iten A}(((\id\ten R)\ten \Gam)(1\ten\Sigma\Gam(x)))=m_{\LIQ\iten \LIQ}(((\id\ten R)\ten \Gam)(1\ten\Sigma\Gam(x)))\\
&=\lim_im_{\LIQ\iten\LIQ}(((\id\ten R)\ten \Gam)(1\ten\Sigma\Gam(x))\cdot(1\ten 1\ten\Gam(e_i)))\\
&=\lim_im_{\LIQ\iten\LIQ}(((\id\ten R)\ten \Gam)(1\ten\Sigma(\Gam(x)(e_i\ten 1)))),\end{align*}
where the limit is taken in the weak* topology of $\LIQ\oten\LIQ$. 

Let $\Gam(x)=\sum_n x_n\ten y_n\in A\iten A$ with each $x_n,y_n\in A$. Such a representation is possible by an operator space version of \cite[Proposition 3.23]{Bi}. Then for each $i$, $\sum_n x_ne_i\ten y_n\in\LIQ\iten A$ and 
\begin{align*}&m_{\LIQ\iten\LIQ}(((\id\ten R)\ten \Gam)(1\ten\Sigma(\Gam(x)(e_i\ten 1))))\\
&=\sum_n m_{\LIQ\iten\LIQ}(((\id\ten R)\ten \Gam)(1\ten y_n\ten x_ne_i))\\
&=\sum_n(1\ten R(y_n))\Gam(x_ne_i))\in \LIQ\iten\LIQ.\end{align*}
Below we make use of the following identity \cite[Proposition 5.24]{KV1} 
\begin{equation}\label{e:anti}R((\psi\ten\id)((a^*\ten 1)\Gam(b)))=(\psi\ten\id)(\Gam(a^*)(b\ten 1)), \ \ \ a,b\in\mc{N}_\psi.\end{equation}
Also, as $\psi=\vphi\circ R$, and element $a\in\mc{M}_\psi$ if and only if $R(a)\in\mc{M}_\vphi$ and in that case $\psi(a)=\vphi(R(a))$. Moreover, if $a\in\mc{A}_{\psi}$ and $b\in\LIQ$ then
$$(\psi\cdot a)\circ R(b)=\psi(aR(b))=\psi(R(bR(a)))=\vphi(bR(a))=(R(a)\cdot\vphi)(b).$$
Let $a,b\in\mc{A}_\psi$, so that $\psi\cdot a^*,\psi\cdot b\in\LOQ$. Then

\begin{align*}\la\psi\cdot a^*\ten\psi\cdot b,\sum_n(1\ten R(y_n))\Gam(x_ne_i)\ra
&=\sum_n(\psi\cdot a^*\ten\psi\cdot b)((1\ten R(y_n))\Gam(x_ne_i))\\
&=\sum_n(\psi\cdot a^*\ten\psi\cdot (bR(y_n)))(\Gam(x_ne_i))\\
&=\sum_n(\psi\cdot (bR(y_n)))((\psi\ten\id)((a^*\ten 1)\Gam(x_ne_i)))\\
&=\sum_n(\psi\cdot (bR(y_n)))(R(\psi\ten\id)(\Gam(a^*)(x_ne_i\ten 1)))\\
&=\sum_n((y_nR(b))\cdot\vphi)((\psi\ten\id)(\Gam(a^*)(x_ne_i\ten 1)))\\
&=\sum_n(e_i\cdot\psi\ten R(b)\cdot\vphi)(\Gam(a^*)(x_n\ten y_n))\\
&=(e_i\cdot\psi\ten R(b)\cdot\vphi)(\Gam(a^*)\Gam(x))\\
&=(e_i\cdot\psi\ten R(b)\cdot\vphi)(\Gam((x^*a)^*))\\
&=(R(b)\cdot\vphi)((\psi\ten\id)(\Gam((x^*a)^*)(e_i\ten 1)))\\
&=(\psi\cdot b)(R((\psi\ten\id)(\Gam((x^*a)^*)(e_i\ten 1)))\\
&=(\psi\cdot b)((\psi\ten\id)((a^*x\ten 1)\Gam(e_i))) \ \ \ \ (x^*a, e_i\in\mc{N}_\psi)\\
&=\la\psi\cdot(a^*x)\ten\psi\cdot b,\Gam(e_i)\ra\\
&\rightarrow\la\psi\cdot(a^*x)\ten\psi\cdot b,1\ten 1\ra\\
&=\la\psi\cdot a^*\ten\psi\cdot b, x\ten 1\ra.\end{align*}
By density of $\{\psi\cdot a^*\ten\psi\cdot b\mid a,b\in\mc{A}_\psi\}$ in $\LOQ\pten\LOQ$, and boundedness of the convergent sum $\sum_n(1\ten R(y_n))\Gam(x_ne_i)$, it follows that 
$$m(((\id\ten R)\ten \Gam)(1\ten\Sigma\Gam(x)))=\lim_i\sum_n(1\ten R(y_n))\Gam(x_ne_i))=x\ten 1.$$
But then
$$x\ten 1=m(((\id\ten R)\ten \Gam)(1\ten\Sigma\Gam(x)))=\sum_n(1\ten R(y_n))\Gam(x_n)),$$
where the latter sum converges in $A\iten A$. Since $x\in A$ was arbitrary, it follows that $\la(1\ten A)\Gam(A)\ra=A\iten A$. An analogous argument using the Tomita algebra $\mc{T}_\vphi$ of the left Haar weight and the left version of (\ref{e:anti}), which is \cite[Corollary 5.35]{KV1}, shows that 
$$1\ten x=\sum_n (R(x_n)\ten 1)\Gam(y_n)\in A\iten A.$$
Thus, $\la(A\ten 1)\Gam(A)\ra=A\iten A$, and it follows that $(A,\Gam)$ is a compact quantum group in the sense of \cite{Wo}. Thus, by the universal property of the quantum Bohr compactification \cite[Theorem 3.1]{Sol} (see also \cite[Proposition 3.4]{D2}), the image of the quantum group morphism $A\rightarrow M(C_0(\G))$ given by inclusion necessarily lies in $\mathrm{AP}(\G)$. Since $\mathrm{AP}(\G)\subseteq A$ \cite[Proposition 4.9]{D2}, we have $\mathrm{AP}(\G)=A$. 

Finally, since $\G$ is co-amenable (Lemma \ref{l:co-amen}), the quantum Eberlein compactification $E(\G)\subseteq M(C_0(\G))$. Hence, $E(\G)$ and therefore $C(b\G)$ is subhomogeneous. Since the Haar state on $C(b\G)$ is tracial, and $C(b\G)$ is nuclear, it follows from \cite[Corollary 4.5]{C} that $b\G$ is co-amenable. Hence, $C(b\G)$ is $*$-isomorphic to $\mathrm{AP}(\G)$ (see \cite[Theorem 7.7]{D2}). 

\end{proof}

\begin{thm}\label{t:Nuclearsub} Let $\G$ be a Kac algebra such that $\h{\G}$ has bounded degree. Then
$$\mc{N}_{\LOQ}(\LOQ,\LIQ)\cong \ell^1(\bGhp)$$
isomorphically, with $\mathrm{deg}(\h{\G})^{-1}\norm{\hat{f}'}\leq \nu(\hat{f}')\leq 2\norm{\hat{f}'}$, for all $\hat{f}'\in\ell^1(\bGhp)$. 
\end{thm}

\begin{proof} The fact that $\Gam\circ\rho_*\circ i:\ell^1(\bGhp)\rightarrow\mc{N}_{\LOQ}(\LOQ,\LIQ)$ is bounded by 2 follows verbatim from Theorem \ref{t:Nuclear}. Suppose that $\Psi\in\mc{N}_{\LOQ}(\LOQ,\LIQ)$. Then $\Psi\in\mc{CB}_{\LOQ}(\LOQ,\LIQ)$, so there exists $x\in\LIQ$ such that $\Psi=\Gam(x)$. But then $x$ defines a completely compact map under convolution, so that $x\in\mathrm{CAP}(\G)=\mathrm{AP}(\G)$ by Theorem \ref{t:AP}, and moreover, $\Gam(x)\in\LIQ\hten\LIQ$ as in the proof of Theorem \ref{t:Nuclear}. As $\mathrm{CAP}(\G)$ is a mapping ideal \cite[Proposition 1.3, Lemma 1]{Run,Saar}, for any $m\in\LIQ^*$ we have
$$(\id\ten m)\Gam(x)=l_m(x), (m\ten\id)\Gam(x)=r_m(x)\in \mathrm{CAP}(\G).$$ 
Identifying $C(b\G)$ with its image $\mathrm{CAP}(\G)$ under the canonical morphism, it follows from \cite[Theorem 4.5]{Smith} that
$$\Gam(x)\in\mc{F}(C(b\G),C(b\G);\LIQ\hten\LIQ)=C(b\G)\hten C(b\G)\subseteq L^{\infty}(b\G)\hten L^{\infty}(b\G).$$
By subhomogeneity, we also have $\Sigma\Gam(x)\in C(b\G)\hten C(b\G)\subseteq L^{\infty}(b\G)\hten L^{\infty}(b\G)$, so the non-commutative Grothendieck inequality \cite{HM,PS} implies $\Gam(x)\in L^{\infty}(b\G)\pten L^{\infty}(b\G)$. Thus, left multiplication by $x$ defines an element of $\mc{N}_{L^1(b\G)}(L^1(b\G),L^{\infty}(b\G))$. Since $\widehat{b\G}$ is a discrete Kac algebra, it is QSIN. Hence, Theorem \ref{t:Nuclear} implies that $x=(\rho_{b\G})_*(\hat{f}')$ for some $\hat{f}'\in\ell^1(\h{b\G}')$. But, under our present convention, equation (\ref{e:bi}) implies $x=(\rho_{b\G})_*(\hat{f}')=\rho_*(i(\hat{f}'))$. Then
$$\nu(\Psi)=\nu(\Gam(\rho_*(i(\hat{f}'))))\geq\gamma^{2,r}(\Gam(\rho_*(i(\hat{f}'))))\geq\mathrm{deg}(\h{\G})^{-1}\norm{\hat{f}'},$$
where the last inequality follows from Theorem \ref{t:va}.
\end{proof}

\begin{cor} Let $\G$ be a Kac algebra for which $\h{\G}$ has bounded degree. The following are equivalent for an element $x\in\LIQ$:
\begin{enumerate}
\item $x\in\rho_*(i(\ell^1(\h{b\G}')))$;
\item $l_x\in\mc{N}_{\LOQ}(\LOQ,\LIQ)$;
\item $\Gam(x)\in\LIQ\hten\LIQ$.
\end{enumerate}
In particular, $\mathrm{AP}(\G)=\la\{x\in\LIQ\mid\Gam(x)\in\LIQ\hten\LIQ\}\ra$. 
\end{cor}

Although our techniques have differed between the QSIN and bounded degree cases, the nature of the results suggests a potential link between these two classes of quantum groups. We summarize some preliminary observations below and leave the general question open.

\begin{prop} Let $\G$ be a co-amenable quantum group with bounded degree. Then $\G$ is QSIN in each of the following cases
\begin{enumerate}
\item $\G$ is commutative;
\item $\G$ is co-commutative;
\item $\G$ is compact;
\item $\G$ is discrete.
\end{enumerate}
\end{prop}

\begin{proof} (1) When $\G=\LI$ is commutative, the assumptions imply that $G$ is virtually abelian, and hence amenable, and hence QSIN.
(2) When $\G=VN(G)$ is co-commutative, the QSIN condition is equivalent to co-amenability.
(3) When $\G$ is compact and has bounded degree, then by \cite{KSol} $\G$ is necessarily a Kac algebra. Since it is also co-amenable, it is therefore QSIN.
(4) When $\G$ is discrete with bounded degree, the dual $C^*$-algebra $C_0(\h{\G})$ is residually finite-dimensional, which, by \cite[Corollary A.3]{Sol} forces $\h{\G}$ and hence $\G$ to be a Kac algebra. Any discrete Kac algebra is QSIN.
\end{proof}

\section*{Acknowledgements}

The third author would like to thank Gero Fendler and Michael Leinert for helpful discussions at an early stage of the project. The second author was partially supported by the NSERC Discovery Grant 1304873, the third author was partially supported by an NSERC Discovery Grant. The first author was supported by a Carleton-Fields Institute Postdoctoral Fellowship.

\end{spacing}


\begin{thebibliography}{00}

%\bibitem{AC} M. Alaghmandan and J. Crann, 
%{\it Character density in central subalgebras of compact quantum groups}. 
%Canadian Math. Bull. 60 (2017), no. 3, 449-461.
 
\bibitem{ATT} M. Alaghmandan, I. G. Todorov and L. Turowska,
\textit{Completely bounded bimodule maps and spectral synthesis}. 
Internat. J. Math. 28 (2017), no. 10, 1750067, 40 pp.


\bibitem{ARS} O. Y. Aristov, V. Runde and N. Spronk,
\textit{Operator biflatness of the Fourier algebra and approximate indicators for subgroups}.
J. Funct. Anal. 209 (2004), no. 2, 367-387.

\bibitem{BC} T. Banica and A. Chirvasitu,
\textit{Thoma type results for discrete quantum groups}.  
Internat. J. Math. 28 (2017), no. 14, 1750103, 23 pp.

\bibitem{Bi} K. D. Bierstedt,
\textit{Introduction to topological tensor products}.
University of Paderborn lecture notes, 2009. Typeset by S.-A. Wegner, Mathematical Institute, University of Paderborn.
 
\bibitem{Black} B. Blackadar,
\textit{Operator Algebras: Theory of $C^*$-algebras and von Neumann algebras}.
Encyclopaedia of Mathematical Sciences 122, Springer-Verlag Berlin Heidelberg 2006.

%\bibitem{B2} D. P. Blecher, 
%\textit{Geometry of the tensor product of $C^*$-algebras}. 
%Math. Proc. Cambridge Philos. Soc. 104 (1988), no. 1, 119-127.
 
\bibitem{B} D. P. Blecher,
\textit{A completely bounded characterization of operator algebras}. 
Math. Ann. 303 (1995), no. 2, 227-239.  

\bibitem{BLM} D. P. Blecher, C. Le Merdy,
\e{Operator Algebras and Their Modules: An Operator Space Approach}.
London Mathematical Society Monographs, New Series 30, Clarendon Press, Oxford, 2004.

\bibitem{BS} D. P. Blecher and R. R. Smith, 
\textit{The dual of the Haagerup tensor product}. 
J. London Math. Soc. (2) 45 (1992), no. 1, 126-144.

\bibitem{BF} M. Bojeko, and G. Fendler,
\textit{Herz-Schur multipliers and completely bounded multipliers of the Fourier algebra of a locally compact group}.
Boll. Un. Mat. Ital. A (6) 3 (1984), no. 2, 297-302. 

\bibitem{BY} M. Brannan and S.-G. Youn,
\textit{On the similarity problem for locally compact quantum groups}.
arXiv:1709.08032.

\bibitem{CLR} M. Caspers, H. H. Lee and E. Ricard,
\textit{Operator biflatness of the $L^1$-algebras of compact quantum groups}.
J. Reine Angew. Math. 700 (2015), 235-244.

\bibitem{Chou} C. Chou, 
\textit{Almost periodic operators in $VN(G)$}. 
Trans. Amer. Math. Soc. 317 (1990), no. 1, 229-253.

\bibitem{CH} M. Cowling and U. Haagerup,
\textit{Completely bounded multipliers of the Fourier algebra of a simple Lie group of real rank one}. 
Invent. Math. 96 (1989), no. 3, 507-549.

\bibitem{C} J. Crann,
\textit{Inner amenability and approximation properties of locally compact quantum groups}.
Indiana Univ. Math. J. to appear, arXiv:1709.01770.

\bibitem{CT} J. Crann and Z. Tanko,
\textit{On the operator homology of the Fourier algebra and its $cb$-multiplier completion}.
J. Funct. Anal. 273 (2017), no. 7, 2521-2545.

\bibitem{CG} G. Crombez and W. Govaerts, 
\textit{Nuclear multipliers from $L_1(G)$ into $L_p(G)$}. 
Simon Stevin 60 (1986), no. 4, 347-351. 

\bibitem{DL} H. G. Dales and A. T.-M. Lau, 
\textit{The second duals of Beurling algebras}.
Mem. Amer. Math. Soc. 177 (2005), no. 836.

\bibitem{DD} B. Das and M. Daws, 
\textit{Quantum Eberlein compactifications and invariant means}.  
Indiana Univ. Math. J. 65 (2016), no. 1, 307-352.

\bibitem{DFJP} W. J. Davis, T. Figiel, W. B. Johnson, and A. Pelczy\'{n}ski, 
\textit{Factoring weakly compact operators}.
J. Funct. Anal. 17 (1974), 311-327.

\bibitem{D3} M. Daws, 
\textit{Multipliers of locally compact quantum groups via Hilbert $C^*$-modules}. 
J. Lond. Math. Soc. (2) 84 (2011), no. 2, 385-407.
 
\bibitem{D} M. Daws,
\textit{Completely positive multipliers of quantum groups}.
Internat. J. Math. 23 (2012), no. 12, 125-132.

\bibitem{D2} M. Daws, 
\textit{Remarks on the quantum Bohr compactification}. 
Illinois J. Math. 57 (2013), no. 4, 1131-1171.

\bibitem{DSV} M. Daws, A. Skalski, and A. Viselter, 
\textit{Around property (T) for quantum groups}. 
Comm. Math. Phys. 353 (2017), no. 1, 69-118.

\bibitem{DR} C. F. Dunkl and D. E. Ramirez, 
\textit{Weakly almost periodic functionals on the Fourier algebra}.
Trans. Amer. Math. Soc. 185 (1973), 501-514.

\bibitem{EJR} E. G. Effros, M. Junge and Z.-J. Ruan, 
\textit{Integral mappings and the principle of local reflexivity for noncommutative $L^1$-spaces}. 
Ann. of Math. (2) 151 (2000), no. 1, 59-92. 

\bibitem{ER2} E. G. Effros and Z.-J. Ruan,
\textit{The Grothendieck--Pietsch and Dvoretzky--Rogers theorems for operator spaces}. 
J. Funct. Anal. 122 (1994), no. 2, 428-450.

\bibitem{ER} E. G. Effros, Z.-J. Ruan,
\textit{Operator spaces}. 
London Mathematical Society Monographs. New Series, 23. The Clarendon Press, Oxford University Press, New York, 2000.
 
\bibitem{FMV} P. Fima, K. Mukherjee and I. Patri,
\textit{On compact bicrossed products}.
J. Noncommut. Geom. 11 (2017), no. 4, 1521-1591.

%\bibitem{EV} M. Enock and L. Va\u{i}nerman, 
%\textit{Deformation of a Kac algebra by an abelian subgroup}.  
%Comm. Math. Phys. 178 (1996), no. 3, 571?596. 

\bibitem{Forr} B. E. Forrest, 
\textit{Arens regularity and discrete groups}. 
Pacific J. Math. 151 (2) (1991) 217-227.


%\bibitem{FR} B. E. Forrest and V. Runde,
%{Amenability and weak amenability of the Fourier algebra}.  
%Math. Z. 250 (2005), no. 4, 731-744.

\bibitem{Gil} J. E. Gilbert,
\textit{$L^p$-convolution operators and tensor products of Banach spaces I,II,III}.
preprints, 1973-1974.

\bibitem{Haa} U. Haagerup, 
\textit{Decomposition of completely bounded maps on operator algebras}.
Unpublished manuscript 1980.

\bibitem{HdL} U. Haagerup and T. de Laat, 
\textit{Simple Lie groups without the approximation property}.
Duke Math. J. 162 (2013), no. 5, 925-964. 

\bibitem{HK} U. Haagerup and J. Kraus,
\textit{Approximation properties for group $C^*$-algebras and group von Neumann algebras}.
Trans. Amer. Math. Soc. 44 (1994), no. 2, 667-699.
 
\bibitem{HM} U. Haagerup and M. Musat,
\textit{The Effros-Ruan conjecture for bilinear forms on $C^*$-algebras}.  
Invent. Math. 174 (2008), no. 1, 139-163.

\bibitem{HNR2} Z. Hu, M. Neufang and Z.-J. Ruan, 
\textit{Completely bounded multipliers over locally compact quantum groups}.
Proc. Lond. Math. Soc. (3) 103 (2011), no. 1, 1-39.

\bibitem{HNR} Z. Hu, M. Neufang and Z.-J. Ruan,
\textit{Module maps over locally compact quantum groups}. 
Studia Math. 211 (2012), no. 2, 111-145. 

\bibitem{Jol} P. Jolissaint, 
\textit{A characterization of completely bounded multipliers of Fourier algebras}. 
Colloq. Math. 63 (1992), no. 2, 311-313.

\bibitem{JNR} M. Junge, M. Neufang, and Z.-J. Ruan,
\textit{A representation theorem for locally compact quantum groups}.
Internat. J. Math. 20 (2009), no. 3, 377-400.

\bibitem{KS} P. Kasprzak and P. M. So\l tan, 
\textit{Embeddable quantum homogeneous spaces}.
J. Math. Anal. Appl. 411 (2014), no. 2, 574-591.

\bibitem{Kn} S. Knudby,
\textit{The weak Haagerup property}.
Trans. Amer. Math. Soc. 368 (2016), no. 5, 3469-3508. 

\bibitem{KSol} J. Krajczok and P. M. So\l tan,
\textit{Compact quantum groups with representations of bounded degree}
arXiv:1710.04863.

\bibitem{K} J. Kustermans,
\textit{Locally compact quantum groups in the universal setting}.
Internat. J. Math. 12 (2001), no. 3, 289-338.

\bibitem{KV1} J. Kustermans and S. Vaes,
\textit{Locally compact quantum groups}.
Ann. Sci. \'{E}cole Norm. Sup. (4) 33 (2000), no. 6, 837-934.

\bibitem{KV2} J. Kustermans and S. Vaes,
\textit{Locally compact quantum groups in the von Neumann algebraic setting}.
Math. Scand. 92 (2003), no. 1, 68-92.

\bibitem{Lance} E. C. Lance, 
\textit{Hilbert $C^*$-Modules}.
London Mathematical Society Lecture Note Series 210, Cambridge University Press, Cambridge, 1995.


\bibitem{LSS} H. H. Lee, E. Samei and N. Spronk,
\textit{Similarity degree of Fourier algebras}.
J. Funct. Anal. 271 (2016), no. 3, 593-609.

\bibitem{LSS2} H. H. Lee, E. Samei and N. Spronk,
\textit{Corrigendum: Similarity degree of Fourier algebras}.
to appear.

%\bibitem{Lep} H. Leptin,
%\textit{Sur l'alg\`{e}bre de Fourier d'un groupe localement compact}.
%C. R. Acad. Sci. Paris S\'{e}r. A-B 266 (1968), A1180-A1182.

\bibitem{LR} Losert, V., Rindler, H.,
\textit{Asymptotically central functions and invariant extensions of Dirac measure}.
Probability measures on groups, VII (Oberwolfach, 1983), 368-378,
Lecture Notes in Math., 1064, Springer, Berlin, 1984.

\bibitem{Majid} S. Majid, 
\textit{Hopf-von Neumann algebra bicrossproducts, Kac algebra bicrossproducts, and the classical Yang-Baxter equations}. 
J. Funct. Anal. 95 (1991), 291-319.

\bibitem{Moore} C. C. Moore, 
\textit{Groups with finite dimensional irreducible representations}.
Trans. Amer. Math. Soc. 166 (1972), 401-410.

\bibitem{NRS} M. Neufang, Z.-J. Ruan and N. Spronk, 
\textit{Completely isometric representations of $M_{cb}A(G)$ and $UCB(\widehat{G})$}. 
Trans. Amer. Math. Soc. 360 (2008), no. 3, 1133-1161. 

\bibitem{PS} H. Pfitzner and G. Schl\"{u}echtermann, 
\textit{Factorization of completely bounded weakly compact operators}, 
preprint, 1997, http://front.math.ucdavis.edu/9706.210.

\bibitem{Pi} G. Pisier,
\textit{Grothendieck's theorem, past and present}.
Bull. Amer. Math. Soc. (N.S.) 49 (2012), no. 2, 237-323.
 
\bibitem{PS} G. Pisier and D. Shlyakhtenko,
\textit{Grothendieck's theorem for operator spaces}.  
Invent. Math. 150 (2002), no. 1, 185-217. 

\bibitem{Qu} J. C. Quigg,
\textit{Approximately periodic functionals on $C^*$-algebras and von Neumann algebras}.
Canad. J. Math. 37 (1985), no. 5, 769-784. 
 
\bibitem{R} G. Racher,
\textit{The nuclear multipliers from $\LO$ into $\LI$}. 
J. Funct. Anal. 122 (1994), no. 2, 279-306.

\bibitem{Rindler} H. Rindler,
\textit{On weak containment properties}. 
Proc. Amer. Math. Soc. 114 (1992), no. 2, 561-563. 

\bibitem{Ru4} Z.-J. Ruan,
\textit{Amenability of Hopf von Neumann algebras and Kac algebras}.
J. Funct. Anal. 139 (1996), no. 2, 466--499.
 
\bibitem{RX} Z.-J. Ruan and G. Xu,
\textit{Splitting properties of operator bimodules and operator amenability of Kac algebras}.
Operator theory, operator algebras and related topics (Timisoara, 1996), 193-216, Theta Found., Bucharest, 1997.

\bibitem{RS} V. Runde and N. Spronk, 
\textit{Operator amenability of Fourier--Stieltjes algebras}.  
Math. Proc. Cambridge Philos. Soc. 136 (2004), no. 3, 675-686.

\bibitem{Run2} V. Runde,
\textit{Uniform continuity over locally compact quantum groups}.
J. Lond. Math. Soc. (2) 80 (2009), no. 1, 55-71.

\bibitem{Run} V. Runde, 
\textit{Completely almost periodic functionals}. 
Arch. Math. (Basel) 97 (2011), no. 4, 325-331.


\bibitem{Saar} H. Saar, 
\textit{Kompakte, vollst\"{a}ndig beschr\"{a}nkte Abbildungen mit Werten in einer nuklearen $C^*$-Algebra}.
Diplomarbeit, Universit\"{a}t des Saarlandes, 1982.

\bibitem{S} S. Sakai,
\textit{Weakly compact operators on operator algebras}. 
Pacific J. Math. 14 (1964), 659-664.

\bibitem{Smith} R. R. Smith,
\textit{Completely bounded module maps and the Haagerup tensor product}.
J. Funct. Anal. 102 (1991), 156-175.

\bibitem{Sol} P. M. So\l tan,
\textit{Quantum Bohr compactification}. 
Illinois J. Math. 49 (2005), no. 4, 1245-1270.

\bibitem{Sp} N. Spronk,
\textit{Measurable Schur multipliers and completely bounded multipliers of the Fourier algebras}. 
Proc. London Math. Soc. (3) 89 (2004), no. 1, 161-192.

\bibitem{T} M. Takesaki,
\textit{A characterization of group algebras as a converse of Tannaka--Stinespring--Tatsuuma duality theorem}.
Amer. J. Math. 91 (1969), 529-564.

\bibitem{T2} M. Takesaki,
\textit{Theory of Operator Algebras II}.
Encyclopedia of Mathematical Sciences 125, Springer-Verlag Berlin--Heidelberg--New York, 2003.

\bibitem{TT} M. Takesaki and N. Tatsuuma,
\textit{Duality and subgroups II.}
J. Funct. Anal. 11 (1972), 184-190.

\bibitem{Vaes} S. Vaes,
\textit{The unitary implementation of a locally compact quantum group action}.
J. Funct. Anal. 180 (2001), no. 2, 426-480.

\bibitem{VV} S. Vaes and L. Vainerman, 
\textit{Extensions of locally compact quantum groups and the bicrossed product construction}.
Adv. Math. 175 (2003), no. 1, 1-101.

\bibitem{VD} A. Van Daele,
\textit{Locally compact quantum groups. A von Neumann algebra approach}.
SIGMA 10 (2014), 082, 41 pp.

%\bibitem{Webster} C. J. Webster,
%\textit{Local Operator Spaces and Applications}.
%PhD Thesis, University of California Los Angeles, 1997.

\bibitem{Wil} B. Wilson,
\textit{A Hilbert space approach to approximate diagonals for locally compact quantum groups}. 
Banach J. Math. Anal. to appear.

\bibitem{Wo} S. L. Woronowicz,
\e{Compact quantum groups}.
Sym\'{e}tries quantiques (Les Houches, 1995), 845-884, North-Holland, Amsterdam, 1998.

\bibitem{X} G. Xu, 
\textit{Herz-Schur multipliers and weakly almost periodic functions on locally compact groups}.
Trans. Amer. Math. Soc. 349 (1997), 2525-2536.

\bibitem{Youn} S.-G. Youn,
\textit{A theorem for random Fourier series on compact quantum groups}.
arXiv:1710.06123

%\bibitem{Young} N. J. Young,
%\textit{The irregularity of multiplication in group algebras}.
%Quart J. Math. Oxford Ser. 24 (1973), 59-62.



%\bibitem{Yl}
%K. Ylinen, 
%\textit{Non-commuattive Fourier transforms of bounded bilinear forms and completely bounded multilinear operators}.
% J. Funct. Anal. 79 (1988) 144-165.
 
 
%\bibitem{Y}
%K. Yosida, 
%\textit{Functional analysis}. 
%Reprint of the sixth (1980) edition. Classics in Mathematics. Springer-Verlag, Berlin, 1995.

\end{thebibliography}
\end{document}